\documentclass[a4paper,11pt]{article}
\usepackage{amsmath}
\usepackage{amsthm}
\usepackage[top=3.5cm, bottom=3.5cm, left=2.5cm, right=2.5cm]{geometry}
\usepackage[utf8]{inputenc}
\usepackage[T1]{fontenc}
\usepackage[french,english]{babel} 
\usepackage{enumerate}
\usepackage[shortlabels]{enumitem}
\usepackage{color}
\usepackage[]{pdfpages}
\usepackage{titling}
\usepackage[colorlinks=true, linkcolor=red,citecolor=blue,linktoc=page]{hyperref} 
\usepackage{csquotes}
\usepackage[nottoc, notlof, notlot, numbib]{tocbibind}
\usepackage[style=alphabetic,giveninits,doi=false,isbn=false,url=false,eprint=false,clearlang=true,sorting=nyt,maxbibnames=99]{biblatex}
\usepackage{mathtools}
\usepackage{thmtools}
\usepackage{bbm}
\usepackage{amssymb}
\usepackage{mathrsfs}
\usepackage{float}
\usepackage{stmaryrd}
\usepackage{array}
\usepackage{tabularx}
\usepackage{ltablex} 
\DeclareUnicodeCharacter{221E}{$\infty$}
\DeclareUnicodeCharacter{00B4}{\'}
\DeclareUnicodeCharacter{03BB}{$\lambda$}
\DeclareUnicodeCharacter{0393}{$\Gamma$}
\DeclareUnicodeCharacter{211D}{$\mathbb{R}$}
\DeclareUnicodeCharacter{2212}{-} 

\newcommand{\A}{\forall}
\newcommand{\R}{\mathbb{R}}
\newcommand{\N}{\mathbb{N}}
\newcommand{\1}{\mathbbm{1}}
\newcommand{\E}{\mathbb{E}}
\newcommand{\ps}{\mathcal{P}}

\renewcommand{\S}{\mathcal{S}}
\renewcommand{\t}{t}
\renewcommand{\d}{\mathrm{d}}
\renewcommand{\P}{\mathbb{P}} 
\renewcommand{\L}{\mathcal{L}} 

\newcommand{\C}{\mathcal{C}_{\overline{\mu}}}
\newcommand{\F}{\mathcal{F}}
\newcommand{\I}{\mathcal{T}}
\newcommand{\T}{\mathcal{T}}
\newcommand{\M}{\mathcal{M}}

\newcommand{\Ad}{\mathcal{A}_\Psi^\zeta}
\newcommand{\Adeq}{\mathcal{A}^{\zeta}}
\newcommand{\Adineq}{\mathcal{A}_{\Psi}}
\newcommand{\AdF}{\mathcal{D}_\mathcal{I}}

\DeclareMathOperator{\argmin}{argmin}
\newtheorem{theorem}{Theorem}[section]
\newtheorem{proposition}[theorem]{Proposition}
\newtheorem{lemma}[theorem]{Lemma}
\newtheorem{corollary}[theorem]{Corollary}
\theoremstyle{definition}
\newtheorem{definition}[theorem]{Definition}
\theoremstyle{remark}
\newtheorem{rem}[theorem]{Remark}
\newtheorem{example}[theorem]{Example}
\theoremstyle{plain}
\newtheorem{assumptiona}{Assumption}

\newtheorem{assumptionb}{Assumption}

\newtheorem{assumptionc}{Assumption}

\newcommand{\dd}{\mathrm{d}}

\usepackage{authblk}

\title{\bf Gibbs principle with infinitely many constraints: optimality conditions and stability}
\author[1,2,4]{Louis-Pierre \textsc{Chaintron}}
\author[3]{Giovanni \textsc{Conforti}}
\author[4]{Julien \textsc{Reygner}}
\affil[1]{\small
DMA, École normale supérieure, Université PSL, CNRS, 75005 Paris, France}
\affil[2]{\small
Inria, Team M${\sf \Xi}$DISIM, Inria Saclay, 91128 Palaiseau, France
}
\affil[3]{\small
Università degli Studi di Padova, Via Trieste, 63, 35131 Padova, Italia
}
\affil[4]{\small
CERMICS, École des Ponts, IP Paris, Marne-La-Vall\'ee, France
}
\date{\today}

\begin{document}

\maketitle

\abstract{We extend the Gibbs conditioning principle to an abstract setting combining infinitely many linear equality constraints and non-linear inequality constraints, which need not be convex.
A conditional large large deviation principle (LDP) is proved in a Wasserstein-type topology,
and optimality conditions are written in this abstract setting.
This setting encompasses versions of the Schrödinger bridge problem with marginal non-linear inequality constraints at every time.
In the case of convex constraints, stability results for perturbations both in the constraints and the reference measure are proved.
We then specify our results when the reference measure is the path-law of a continuous diffusion process, whose law is constrained at each time.
We obtain a complete description of the constrained process through an atypical mean-field PDE system involving a Lagrange multiplier.}

\tableofcontents

\section{Introduction} \label{sec:Introduction}

\subsection{Motivation}

A fundamental question in statistical mechanics is to estimate the most likely configuration of a large system of exchangeable particles, given some macroscopic observation on it.
More precisely, let us consider a $N$-tuple of exchangeable random variables $\vec{X}^N := ( X^{1,N}, \ldots , X^{N,N})$ in some abstract Polish space $E$.
We assume that the $X^{i,N}$ are either independent or interact through their empirical measure
\[ \pi ( \vec{X}^N ) := \frac{1}{N} \sum_{i =1}^N \delta_{X^{i,N}} \in \ps (E). \]
The purpose is then to estimate $\pi ( \vec{X}^N )$ as $N \rightarrow +\infty$, given the knowledge that $\{ \pi ( \vec{X}^N ) \in A \}$ for some measurable $A \subset E$.
This amounts to determining the behaviour of a \emph{typical particle} in the system given the \emph{observation} that $\{ \pi ( \vec{X}^N ) \in A \}$.
In a physical context, measurements are often submitted to uncertainties.
Therefore, a closely related question is the stability of the computed behaviour when perturbing the observation. 

These questions can be specified when $\L ( \pi ( \vec{X}^N ) )$ satisfy the large deviation principle (LDP) with rate function $\mathcal{I}$.
Informally speaking, the LDP states that
\begin{equation} \label{eq:IntroLDP}
\P ( \pi ( \vec{X}^N ) \in B ) \asymp \exp \big[ - N \inf_{ \mu \in B } \mathcal{I} ( \mu ) \big], 
\end{equation} 
at the exponential-in-$N$ scale, for any sufficiently regular $B \subset E$. 
In this setting, the \emph{Gibbs conditioning principle} suggests the heuristic 
\begin{equation} \label{eq:IntroGibbs}
\L ( X^{1,N} \vert \pi ( \vec{X}^N ) \in A ) \xrightarrow[N \rightarrow + \infty]{} \argmin_{A} \mathcal{I} , 
\end{equation}
provided that $ \argmin_{A} \mathcal{I} $ is well-defined, and $ \argmin_{ A} \mathcal{I} $ is also the weak limit of $ \pi ( \vec{X}^N )$ conditionally on $\{ \pi ( \vec{X}^N ) \in A \}$. More precisely, a LDP can be established for the conditional law.
Some reference textbooks on this approach are \cite{lanford1973entropy,ruelle1965correlation,dembo2009large,dupuis2011weak,ellis2006entropy}.
To explicitly compute $ \argmin_{A} \mathcal{I} $, some further knowledge of $\mathcal{I}$ is needed.

The scope of this article is to establish \eqref{eq:IntroGibbs}, to compute minimisers of $\mathcal{I}$, and to prove stability properties when $\mathcal{I}$ is the sum of relative entropy and an interaction term:
\[ \mathcal{I} ( \mu ) := H ( \mu \vert \nu ) + \F ( \mu ), \]
the definition of the relative entropy $H ( \mu \vert \nu )$ being recalled in Section \ref{subsec:notations}, and $A = \Ad$, where
\[ \Ad := \big\{ \mu \in \ps ( E ), \; \forall s \in \S, \, {\textstyle \int_E \zeta_s \d \mu = 0}, \, \forall t \in \T, \, \Psi_t ( \mu ) \leq 0 \big\} \]
is given by (possibly) infinitely many linear equality constraints and non-linear inequality constraints.
Precise assumptions on $\F$, $( \zeta_s )_{s \in \S}$ and $( \Psi_t )_{t \in \T}$ are detailed in Section \ref{s:abstract-results}, under which  we prove that minimisers are Gibbs measures, whose density involves suitable Lagrange multipliers.
Our approach relies on the Hahn-Banach theorem, combining tools from functional analysis in the spirit of \cite{csiszar1984sanov,leonard2000minimizers,pennanen2019convex} with differential calculus on $\ps (E )$.
In particular, $\F$ and the $\Psi_t$ need not be convex.
When $\F \equiv 0$, the LDP \eqref{eq:IntroLDP} is the well-known Sanov theorem \cite{dembo2009large} for independent particles. 
The case $\F \neq 0$ allows for mean-field interaction, modelling particles that are distributed according to mean-field Gibbs measures as in \cite{leonard1987large,arous1990methode,wang2010sanov,dupuis2020large}.
When $\S = \emptyset$, our results extend the standard Gibbs principle to settings with infinitely many constraints, see Section \ref{subsec:IntroGibbs} below.
When $E$ is a product space, we can further impose marginal laws for $\argmin_{\Ad} \mathcal{I}$ through the linear equality constraints.
This last setting includes the famous Schrödinger bridge problem, with additional inequality constraints, see Section \ref{subsec:IntroSchrod} below. 
A notable example, which is studied in Section \ref{s:processes}, is given by the space of continuous paths $E = C ( [0,T], \R^d )$.

When $\F$ and the $\Psi_t$ are convex, uniqueness holds for the minimiser $\argmin_{\Ad} \mathcal{I}$, corresponding to the notion of \emph{entropic projection} introduced by \cite{csiszar1975divergence,csiszar1984sanov}.
Stability results for entropic projections have enjoyed many recent developments \cite{ghosal2022stability,eckstein2022quantitative,nutz2023stability,chiarini2023gradient,divol2024tight}, mainly motivated by the surge of interest around the Schr\"odinger problem and its applications in machine learning.  
In our abstract setting, we prove two kinds of apparented results:
\begin{itemize}
    \item A quantitative stability result when changing $\Psi_t$ into $\Psi_t - \varepsilon$ for small $\varepsilon > 0$.
    \item A weak stability result when perturbing $\nu$, $\F$, $( \Psi_t )_{t \in \T}$ at the same time, when $\S = \emptyset$.
\end{itemize}
These new stability results hold in $\ps (E)$ at a great level of generality, under minimal assumptions on $\nu$, $\F$, $( \Psi_t )_{t \in \T}$ that are stated in Section \ref{ss:convex}.

Let us now illustrate our results with the examples of the Gibbs conditioning principle and the Schrödinger bridge problem. For the clarity of exposition, the next two subsections provide an overview of our results only in the particular case $\F \equiv 0$ with linear constraints $\Psi_t ( \mu) = \int_E \psi_t \d \mu$.

\subsection{The Gibbs conditioning principle} \label{subsec:IntroGibbs}

When $\mathcal{I} ( \mu ) = H ( \mu \vert \nu )$ and the constraints are given by a finite number of moments against $\psi_1, \ldots \psi_T : E \rightarrow \R$, namely  
\[ A = \big\{ \mu \in \ps (E), \quad \forall \t \in \{1, \ldots T \}, \; {\textstyle\int_E \psi_t \d \mu \leq 0} \big\}, \]
then it is well-known that the density of $\overline{\mu} := \argmin_{ \mu \in A} H ( \mu \vert \nu )$ is given by
\begin{equation} \label{eq:gibbs-density}
\frac{\d\overline\mu}{\d\nu} (x) = {\overline Z}^{-1} \exp \bigg[ -\sum_{t=1}^T \overline\lambda_t \psi_t (x) \bigg], 
\end{equation} 
where $\overline{Z}$ is a normalising constant, and the $\overline\lambda_1, \ldots, \overline{\lambda}_T$ are non-negative Lagrange multipliers. 
A similar result holds when imposing ${\textstyle\int_E \psi_t \d \mu = 0}$ instead of inequalities, corresponding to the \emph{canonical ensemble} in statistical physics.
In addition to the aforementioned textbooks on the Gibbs principle, we mention the prominent contributions \cite{borel1906principes,diaconis1987dozen,stroock1991microcanonical} for precise statements, and \cite{dembo1996refinements,dembo1998re,cattiaux2007deviations} for quantitative versions of \eqref{eq:IntroLDP} in this setting.

Our abstract results now allow for an infinite number of constraints $( \psi_t )_{t \in \T}$, provided a continuous dependence on $t$ in the compact space $\T$.
In particular, Theorem \ref{thm:abstractGibbs} provides existence for a positive Radon measure $\overline\lambda \in \M_+ ( \T )$, which generalises the previous multipliers, such that
\[ \frac{\d\overline\mu}{\d\nu} (x) = {\overline Z}^{-1} \exp \bigg[ - \int_{\T } \psi_t (x) \overline{\lambda} ( \d t ) \bigg]. \]
In the particular setting $E = C ( [0,T], \R^d )$ and $\T = [0,T]$, $x = ( x_t )_{t \in [0,T]}$ being a continuous path, a natural choice is $\psi_t ( x ) = \psi ( x_t )$, for some continuous $\psi : \R^d \rightarrow \R$.
In this case, if $\nu$ is the path-law of the solution to the stochastic differential equation (SDE), $(B_t)_{t \geq 0}$ being a Brownian motion,
\[ \d X_t = b_t ( X_t ) \d t + \sigma_t ( X_t ) \d B_t, \]
under standard Lipschitz assumptions,
Theorem \ref{thm:LinkPath} further identifies $\overline\mu$ as being the path-law of the solution to the SDE
\[ \d \overline{X}_t = b_t ( \overline{X}_t ) \d t - \sigma_t \sigma_t^\top \nabla \varphi_t ( \overline{X}_t ) \d t + \sigma_t ( \overline{X}_t ) \d B_t, \quad \overline{X}_0 \sim \overline{Z}^{-1} e^{-\varphi_0 (x)} \d x, \]
where $\varphi$ is the weak solution of the Hamilton-Jacobi-Bellman (HJB) equation
\[ - \varphi_t + \int_t^T b_s \cdot \nabla \varphi_s - \frac{1}{2} \vert \sigma_s^\top \nabla \varphi_s \vert^2 + \frac{1}{2} \mathrm{Tr} [ \sigma_s \sigma_s^\top \nabla^2 \varphi_s ] \, \d s + \int_{[t,T]} \psi_s \, \overline{\lambda} ( \d s ) = 0, \]
for which we prove well-posedness in a suitable sense detailed in Section \ref{sec:Gibbsdif}.
This HJB approach is reminiscent of mean-field control problems under constraints that were recently studied in \cite{daudin2020optimal,daudin2023optimal,daudin2023mean}.

\subsection{Schrödinger bridge with additional constraints} \label{subsec:IntroSchrod}

When $E = C ( [0,T], \R^d )$ and $\T = [0,T]$, another famous example is the Schr\"odinger bridge problem \cite{Schr32}, which is the prototype of a stochastic mass transport problem. 
For this problem, 
\[ A = \big\{ \mu \in \ps ( C ( [0,T], \R^d ) ), \; \mu_0 = \mu_{\mathrm{ini}}, \, \mu_T = \mu_{\mathrm{fin}} \big\}, \]
where the marginal laws $\mu_{\mathrm{ini}}, \mu_{\mathrm{fin}} \in \ps ( \R^d )$ are imposed.
The literature on Schr\"odinger bridges has recently enjoyed thriving developments.
Some seminal references are \cite{csiszar1975divergence,cattiaux1995large,cattiaux1996minimization}.
For more recent results, we refer to the survey article \cite{leonard2013survey}, the lecture notes \cite{nutz2021introduction}, and references therein. 
We also mention \cite{backhoff2020mean} for an extension to a more general rate function $\mathcal{I}$. 
Denoting $\overline{\mu} := \argmin_{\mu \in A} H (\mu \vert \nu )$, a standard result \cite{nutz2021introduction} is the existence of measurable functions $\xi_0, \xi_T : \R^d \rightarrow \R$, called \emph{Schrödinger potentials}, such that
\[ \frac{\d \overline{\mu}}{\d \nu} ( x ) = \exp \big[ - \xi_0 ( x_0 ) - \xi_T ( x_T ) \big], \]
under suitable assumptions on $(\nu, \mu_{\mathrm{ini}}, \mu_{\mathrm{fin}})$.
Our results extend this decomposition to the case
\[ A = \big\{ \mu \in \ps ( C ( [0,T], \R^d ), \; \mu_0 = \mu_{\mathrm{ini}}, \, \mu_T = \mu_{\mathrm{fin}}, \, \forall t \in [0,T], \, {\textstyle \int_{\R^d} \psi \d \mu_t \leq 0} \big\}. \]
In this setting, as a consequence of Theorem \ref{thm:abstractGibbs}, Theorem \ref{thm:ConsSchröd} provides $\overline{\lambda} \in \M_+ ( [0,T] )$ and measurable $\xi_0, \xi_T : \R^d \rightarrow \R$ such that
\[ \frac{\d \overline{\mu}}{\d \nu} ( x ) = \exp \bigg[ - \xi_0 ( x_0 ) - \xi_T ( x_T ) - \int_{[0,T]} \psi ( x_t ) \overline{\lambda} ( \d t) \bigg]. \]
In fact, Theorem \ref{thm:ConsSchröd} states a stronger result, since it allows for non-linear and non-convex constraints $\Psi ( \mu_t ) \leq 0$ instead of $\int_{\R^d} \psi \d \mu_t \leq 0$.

\subsection{Outline}

This article is organised as follows.
Our main abstract results are stated in Section
\ref{s:abstract-results}.
The optimality conditions are presented in Section~\ref{ss:gibbs-density}, whereas the stability results are given in Section~\ref{ss:convex}.
Section \ref{s:processes} develops the case of continuous paths $E = C([0,T],\R^d)$. 
Our results on the Schrödinger bridge problem with additional constraints are stated in Section \ref{ss:schro}. 
The specific case of Gibbs principle for diffusion processes is further studied in Section \ref{sec:Gibbsdif}, and some examples for Gaussian processes can be found in Section \ref{ss:brown-expl}.
The proofs of the results are written in Sections \ref{s:proofAbs}-\ref{sec:Difproofs}.
Appendix \ref{sec:appLin} finally presents an alternative proof of Theorem~\ref{thm:abstractGibbs} (from Section \ref{ss:gibbs-density}) in the linear setting of Section \ref{subsec:IntroGibbs}, with improved assumptions.

\subsection{Notations} \label{subsec:notations}

\begin{itemize}
    \item $\vec{x}^N$ denotes  a generic element of a product set $E^N$, for an integer $N \geq 1$. 
    We also write $\vec{x}^N = ( x^i )_{1 \leq i \leq N}$. 
    \item $\ps (E)$ denotes the set of probability measures over a Polish space $E$ endowed with its Borel $\sigma$-algebra.
    \item $\delta_x$ denotes the Dirac measure at some point $x$ in $E$.
    \item $\pi$ denotes the function that maps a vector $\vec{x}^N$ in $E^N$ to the empirical measure 
    \[ \pi ( \vec{x}^N ) := \frac{1}{N} \sum_{i=1}^N \delta_{x^i} \in \ps(E). \] 
    \item $\langle\mu, f \rangle$ denotes the integral (when it exists) of a measurable function $f$ against the measure $\mu$. 
    The convention will often be adopted that $\langle\mu, f\rangle = +\infty$ if the integral is not well-defined. 
    \item $\mu( \cdot \vert A)$ stands for the conditional measure $\mu(A)^{-1} \1_A \mu$, for $\mu$ in $\ps(E)$ and $A$ measurable with $\mu(A) >0$.   
    \item $\M_+(\I)$ denotes the convex cone of positive Radon measures. In this setting, a Radon measure is defined as a signed finite measure that is both inner and outer regular as defined in \cite[Definition 2.15]{rudin1970real}.
    \item $\ps_\phi (E)$ denotes the set of measures $\mu \in \ps (E)$ with $\langle \mu , \phi \rangle < +\infty$ for some measurable $\phi : E \rightarrow \R$. When $\phi(x) = d(x,x_0)^p$ for some distance $d$ on $E$, some $x_0$ in $E$ and some $p \geq 1$, we will often write $\ps_p (E)$ instead.
    \item $W_p$ denotes the Wasserstein distance on $\ps_p(E)$, defined by
    \[ W_p(\mu,\mu') := \bigg[ \inf_{X \sim \mu, \, Y \sim \mu'} \E [ \vert X -Y \rvert^p ] \bigg]^{1/p}.  \]
    \item $T > 0$ is a given real number, and $d \geq 1$ is an integer.
    \item $x_{[0,T]} \in C([0,T],\R^d)$ denotes a continuous function $x_{[0,T]} : [0,T] \rightarrow \R^d$.
    \item $\mu_{[0,T]}$ denotes a path measure in $\ps( C([0,T], \R^d))$. Its marginal measure at time $t$ will be denoted by $\mu_t$.
    \item $\mu_{\cdot}$ denotes a continuous curve $t \mapsto \mu_t$ in $C([0,T], \ps(\R^d))$.
    \item ${\sf{X}}_t$ denotes the coordinate map $x_{[0,T]} \mapsto x_t$. It can be seen as a random variable on the canonical space $\Omega = C([0,T],\R^d)$.
    \item $\tfrac{\delta\Psi}{\delta\mu}(\mu) : x \mapsto \tfrac{\delta\Psi}{\delta\mu}(\mu,x)$ denotes the linear functional derivative at $\mu$ (when it exists) of a function $\Psi : \ps(\R^d) \rightarrow \R$. See Definition \ref{def:diff} below for more details.  
    The convention is adopted throughout the paper that $\langle \mu, \tfrac{\delta\Psi}{\delta\mu}(\mu) \rangle =0$. 
    \item $\cdot^\top$ and $\mathrm{Tr}[\cdot]$ respectively denote the transpose and the trace of matrices.
\end{itemize}

\section{Abstract setting and main results}\label{s:abstract-results}

This section contains our main two theorems for the generalisation of the Gibbs principle. We first present our abstract setting in Section~\ref{ss:abstract-setting}, allowing for interaction terms and infinitely many linear equality constraints and nonlinear inequality constraints. We then establish a conditional LDP (Theorem~\ref{thm:AbsLDP}), taking only inequality constraints into account, in Section~\ref{ss:LDP}. Next, we establish a Gibbs-like formula for the minimiser of the rate function of the conditional LDP (Theorem~\ref{thm:abstractGibbs}) in Section~\ref{ss:gibbs-density}, and study more properties, in particular stability, of this minimiser in the case where the rate function and the inequality constraints are convex, in Section~\ref{ss:convex}.

\subsection{Global setting}\label{ss:abstract-setting}

\subsubsection{The rate function}\label{sss:def-I}

Let $E$ be a complete separable metric space, endowed with its Borel $\sigma$-algebra. 
The space $\ps(E)$ of probability measures over
$E$ is endowed with the \emph{weak topology} \cite[Chapter 2]{billingsley2013convergence}. 
We recall that a sequence $( \mu_k )_{k \geq 1}$ weakly converges towards $\mu$ in $\ps(E)$ if and only if 
\[ \langle \mu_k, \varphi \rangle \xrightarrow[k \rightarrow +\infty]{} \langle \mu, \varphi \rangle, \]
for every bounded continuous $\varphi : E \rightarrow \R$.
The space $\ps(E)$ is endowed with its Borel $\sigma$-algebra.
For any continuous $\phi : E \rightarrow \R_+$, we define the set
\[ \ps_\phi (E) := \{ \mu \in \ps(E) \, , \, \langle \mu , \phi \rangle < +\infty \}. \]
The topology on $\ps_\phi (E)$ is now defined through its converging sequences. 

\begin{definition}[Weak convergence in $\ps_\phi(E)$] \label{def:weakCVphi}
A sequence $( \mu_{k} )_{k \geq 1}$ weakly converges towards $\mu$ in $\ps_\phi (E)$ if and only if 
\[ \langle \mu_k, \phi \rangle \xrightarrow[k \rightarrow +\infty]{} \langle \mu, \phi \rangle \quad \text{and} \quad \langle \mu_k, \varphi \rangle \xrightarrow[k \rightarrow +\infty]{} \langle \mu, \varphi \rangle, \]
for every bounded continuous $\varphi : E \rightarrow \R$.
\end{definition}
    
This topology on $\ps_\phi(E)$ is stronger than the one induced by the weak topology on $\ps(E)$. 
It was used for instance in \cite{leonard1987large}.
When $\phi(x) = 1$ for every $x$ in $E$, we notice that $\ps(E)$ and $\ps_\phi(E)$ coincide, together with their topology.
When $\phi(x) = d(x,x_0)^p$ for some $x_0$ in $E$ and $p \geq 1$, this topology corresponds to the one defined in \cite[Definition 6.8]{villani2009optimal}.
This latter topology can be metrised using the $p$-Wasserstein distance $W_p$ \cite[Theorem 6.9]{villani2009optimal}. 

\medskip
The dual representation formula 
\begin{equation}\label{eq:DualEntrop}
H( \mu \vert \nu ) = \sup_{\substack{\Phi \text{ measurable} \\ \langle \nu, e^{\Phi} \rangle < \infty}} \langle \mu , \Phi \rangle - \log \, \langle \nu, e^\Phi \rangle
\end{equation} 
of the relative entropy~\cite[Proposition 3.1-(iii)]{leonard2012girsanov} yields the following statement.

\begin{lemma}[Integrability condition]\label{lem:integ}
    Let $\phi: E \rightarrow \R_+$ be a continuous function and $\nu \in \ps(E)$. If there exists $\alpha>0$ such that $\langle \nu, e^{\alpha \phi}\rangle < \infty$, then any probability measure $\mu \in \ps(E)$ such that $H(\mu|\nu)<+\infty$ is in $\ps_\phi(E)$.
\end{lemma}

We now fix a continuous function $\phi: E \rightarrow \R_+$ and a reference probability measure $\nu \in \ps(E)$.

\begin{lemma}[Relative entropy as a good rate function on $\ps_\phi(E)$]\label{lem:Hgoodratepsphi}
    If the condition
    \begin{equation}\label{eq:phi-nu}
        \forall \alpha > 0, \qquad \langle \nu, e^{\alpha \phi}\rangle < \infty
    \end{equation}
    holds, then $\mu \mapsto H(\mu \vert \nu)$ is lower semi-continuous on $\ps_\phi(E)$, with compact level sets.
\end{lemma}
In the terminology of \cite[Section 1.2]{dembo2009large}, under the condition~\eqref{eq:phi-nu}, the relative entropy with respect to $\nu$ is a \emph{good rate function} on $\ps_\phi(E)$.
\begin{proof}
    The compactness of the level sets of $H$ for the usual weak topology is part of the Sanov theorem. 
    From \eqref{eq:DualEntrop} applied to $\Phi = \alpha \phi \1_{\lvert \phi \rvert \geq M} $ for every $\alpha, M > 0$, $\phi$ is uniformly integrable over any level set of $H$. The compactness for the weak topology on $\ps_\phi (E)$ then follows from \cite[Theorem 3.5]{billingsley2013convergence}.    
\end{proof}

As is stated in the proof of Lemma~\ref{lem:Hgoodratepsphi}, by Sanov's theorem, the relative entropy describes the large deviations of systems of independent particles. To take interaction between particles into account, we now consider a measurable map $\F : \ps_\phi(E) \rightarrow (-\infty,+\infty]$, and define the function $\mathcal{I} : \ps_\phi(E) \rightarrow (-\infty,+\infty]$ by 
\begin{equation} \label{eq:defI}
    \mathcal{I} (\mu) := H( \mu \vert \nu ) + \F(\mu),
\end{equation}
the domain of $\mathcal{I}$ being the set
\begin{equation}\label{eq:defAdF} 
    \AdF := \{ \mu \in \ps_\phi(E), \; H(\mu\vert \nu) < +\infty, \, \F(\mu) < +\infty\}.
\end{equation}

\begin{assumptiona}[Rate function] \label{ass:FiniteF}
Condition~\eqref{eq:phi-nu} holds and $\mathcal{I}$ is lower semi-continuous on $\ps_\phi (E)$.
\end{assumptiona}

This assumption tells that $\mathcal{I}$ is a \emph{rate function} on $\ps_\phi(E)$. Of course, by Lemma~\ref{lem:Hgoodratepsphi}, $\mathcal{I}$ is lower semi-continuous on $\ps_\phi(E)$ as soon as $\F$ is lower semi-continuous on $\ps_\phi(E)$. Then a sufficient condition for $\mathcal{I}$ to be a \emph{good} rate function on $\ps_\phi(E)$ is given as follows. The proof is omitted.

\begin{lemma}[Sufficient condition for $\mathcal{I}$ to be a good rate function]\label{lem:Igood}
    Under Assumption~\ref{ass:FiniteF}, if $\F$ is bounded from below on $\ps_\phi(E)$, then $\mathcal{I}$ has compact level sets in $\ps_\phi(E)$.
\end{lemma}

\subsubsection{Constraints}

The first goal of this article is to study the minimisation of the rate function $\mathcal{I}$ on the space $\ps_\phi(E)$, subject to two types of constraints: linear equality constraints, and non-linear inequality constraints.

\medskip
Linear equality constraints are encoded by a family $(\zeta_s)_{s \in \S}$ of functions $E \rightarrow \R$, where $\S$ is an arbitrary set. They will be assumed to satisfy the following properties.

\begin{assumptiona}[On the linear equality constraints]\label{ass:linear-constr}
    For any $s \in \S$, $\zeta_s$ is continuous on $E$, and there exists $C^\zeta_s \in [0,\infty)$ such that
    \[ \forall x \in E, \quad \lvert \zeta_s (x) \rvert \leq C^\zeta_s [ 1 + \phi(x) ]. \]
\end{assumptiona}

We define the associated constrained subset of $\ps_\phi(E)$ by
\[\Adeq := \{ \mu \in \ps_\phi(E), \; \forall s \in \S, \;  \langle \mu , \zeta_s \rangle = 0 \}.\]

Nonlinear inequality constraints are encoded by a family $( \Psi_{\t} )_{\t \in \T}$ of functions $\ps_\phi (E) \rightarrow \R$, where $\T$ is a compact metric space. The basic assumption on these functions is the following.

\begin{assumptiona}[On the nonlinear inequality constraints]\label{ass:nonlinear-constr}
    For any $\t \in \T$, the function $\Psi_{\t}$ is lower semi-continuous on $\ps_\phi(E)$.
\end{assumptiona}

We define the associated constrained subset of $\ps_\phi(E)$ by 
\[\Adineq := \{ \mu \in \ps_\phi(E), \; \forall t \in \T, \;  \Psi_t ( \mu ) \leq 0 \}.\]

\begin{rem}[Linear inequality constraints]\label{rem:lin-ineq-constr}
    A particular case of inequality constraints is given by linear functions $\Psi_t(\mu) = \langle \mu, \psi_t\rangle$ for some function $\psi_t : E \to \R$. Then Assumption~\ref{ass:nonlinear-constr} holds as soon as, for any $t \in \T$, $\psi_t$ is lower semi-continuous and there exists $C^\psi_t \in [0,\infty)$ such that 
    \[ \forall x \in E, \quad [\psi_t (x)]_- \leq C^\psi_t [ 1 + \phi(x) ]. \]
\end{rem}

Our assumptions on the linear equality and nonlinear inequality constraints yield the following statement.

\begin{lemma} \label{lem:AdClosed}
Under~\ref{ass:linear-constr}-\ref{ass:nonlinear-constr}, the constrained sets $\Adeq$, $\Adineq$ and $\Ad := \Adeq \cap \Adineq$ are closed in $\ps_\phi(E)$.  
\end{lemma}

The elements of $\AdF \cap \Ad$ are the \emph{admissible measures}. Under~\ref{ass:FiniteF}-\ref{ass:linear-constr}-\ref{ass:nonlinear-constr}, the sequel of this section is dedicated to the study of the minimisation problem
\begin{equation}\label{eq:min-pb}
    \overline{\mathcal{I}}^\zeta_\Psi := \inf_{\mu \in \AdF \cap \Ad} \mathcal{I}(\mu).
\end{equation}
By Lemma~\ref{lem:AdClosed}, any limit $\overline\mu$ of a minimising sequence for this problem belongs to $\AdF \cap \Ad$, and by the semi-continuity assumption~\ref{ass:FiniteF} it satisfies 
\begin{equation*}
    \mathcal{I}(\overline\mu) = \overline{\mathcal{I}}^\zeta_\Psi < \infty.
\end{equation*}
We call $\overline\mu$ a \emph{minimiser} for~\eqref{eq:min-pb}. 
If $\AdF \cap \Ad$ is non-empty and $\mathcal{I}$ is a good rate function, as in Lemma~\ref{lem:Igood}, then there exists at least one minimiser.

\subsection{Conditional LDP for nonlinear inequality constraints}\label{ss:LDP}

As is explained in Section~\ref{sec:Introduction}, the first step of the proof of the Gibbs principle is the derivation of a LDP for the conditional distribution of the empirical measure of the particle system under inequality constraints. Thus, in this section, we consider a sequence $(\Pi^N)_{N \geq 1}$ of probability measures on $\ps_\phi(E)$, which is assumed to satisfy a LDP with rate function $\mathcal{I}$ defined in~\eqref{eq:defI}. Under the assumption that
\begin{equation}\label{eq:PiNAdineq-pos}
    \forall N \geq 1, \qquad \Pi^N(\Adineq) > 0,
\end{equation}
we may define the conditional distribution $\Pi^N(\cdot|\Adineq)$. The goal of this section is to state a LDP for the sequence $\Pi^N(\cdot|\Adineq)$.

\subsubsection{Statement of the result}

In order to state the LDP for this sequence, we need to introduce the following assumptions.

\begin{assumptiona}[Regularity of $\Psi_t$]\label{ass:psi-ldp}
    The function $\mu \mapsto \sup_{t \in \T} \Psi_t(\mu)$ is well-defined and continuous on $\ps_\phi(E)$.
\end{assumptiona}

\begin{assumptiona}[Constraint qualification]\label{ass:constr-qual-ldp}
    For any $\mu \in \AdF \cap \Adineq$, there exists $\tilde\mu \in \AdF$ such that:
    \begin{enumerate}[label=(\roman*),ref=\roman*]
        \item\label{ass:constr-qual-ldp:1} for every $\varepsilon > 0$ small enough, we have $\Psi_t(\mu + \varepsilon(\tilde\mu-\mu)) < 0$ for all $t \in \T$;
        \item\label{ass:constr-qual-ldp:2} $\limsup_{\varepsilon \to 0} \F(\mu + \varepsilon(\tilde\mu-\mu)) \leq \F(\mu)$.
    \end{enumerate}
\end{assumptiona}

\begin{theorem}[Conditional LDP]\label{thm:AbsLDP}
    Let~\ref{ass:FiniteF}-\ref{ass:nonlinear-constr}-\ref{ass:psi-ldp}-\ref{ass:constr-qual-ldp} hold, and assume that
    \begin{equation}\label{eq:min-pb-ineq}
        \overline{\mathcal{I}}_\Psi := \inf_{\mu \in \AdF \cap \Adineq} \mathcal{I}(\mu) < +\infty.
    \end{equation}
    Let $(\Pi^N)_{N \geq 1}$ be a sequence of probability measures on $\ps_\phi(E)$ which satisfies a LDP with rate function $\mathcal{I}$, such that~\eqref{eq:PiNAdineq-pos} holds. Then $( \Pi^N ( \cdot \vert \Adineq ))_{N \geq 1}$ satisfies a LDP with rate function $\mathcal{I}_\Psi : \ps_\phi(E) \rightarrow [0,+\infty]$ defined by
    \[
        \mathcal{I}_\Psi ( \mu ) :=
        \begin{cases}
            \mathcal{I} (\mu) - \overline{\mathcal{I}}_\Psi, &\text{if $\mu \in \Adineq$,} \\
            + \infty, &\text{otherwise}.
        \end{cases} 
    \]
\end{theorem}

Theorem~\ref{thm:AbsLDP} is proved in Section~\ref{ss:proofs-LDP}. 
This result can be seen as a variation of \cite[Theorem 1]{la2015general} or
\cite[Theorem 3.1]{chafai2021coulomb}, without requiring the rate function to be good.
Under the assumptions of Theorem~\ref{thm:AbsLDP}, if $\mathcal{I}$ is a good rate function (as in Lemma~\ref{lem:Igood}), then by Lemma~\ref{lem:AdClosed}, so is $\mathcal{I}_\Psi$ and therefore $( \Pi^N ( \cdot \vert \Adineq ))_{N \geq 1}$ satisfies the LDP with a good rate function. In particular, there is at least one minimiser for~\eqref{eq:min-pb-ineq}. If this minimiser $\overline\mu$ is moreover unique, then $\Pi^N ( \cdot \vert \Adineq)$ weakly converges towards $\delta_{\overline{\mu}}$. Therefore, in this case, the Gibbs principle holds in the general setting of Theorem~\ref{thm:AbsLDP}, namely with infinitely many constraints $\Psi_t$, which need not be linear functionals of $\mu$, and interaction between particles, which are encoded by the functional $\mathcal{F}$. Moreover, it is stated in the abstract metric space $E$, for the Wasserstein-like topology on $\mathcal{P}_\phi(E)$ which is stronger than the usual weak topology. 

The detailed study of minimisers for~\eqref{eq:min-pb-ineq} is carried out in Sections~\ref{ss:gibbs-density} and~\ref{ss:convex}, in the more general setting of~\eqref{eq:min-pb} which takes equality constraints into account. In the sequel of the present section, we discuss how to check the continuity assumption~\ref{ass:psi-ldp} when the functions $\Psi_t$ have a certain regularity property, and then we provide an example of application of Theorem~\ref{thm:AbsLDP}, when $\F=0$. We therefore obtain a conditional version of Sanov's Theorem, which is stated in Wasserstein topology. Some examples with $\F \neq 0$
are LDPs for Gibbs measures like \cite{arous1990methode} which typically enter the framework of Theorem \ref{thm:AbsLDP}, and suitable examples in the Wasserstein topology are given in \cite{leonard1987large,liu2020large,dupuis2020large,chaintronLDP}.

\subsubsection{Checking~\ref{ass:psi-ldp}
for differentiable constraints}

In the sequel of this article, we shall be particularly interested in the case where the functions $\Psi_t$ are differentiable with respect to $\mu$, using the notion of linear functional derivative, whose definition and basic properties are gathered in Appendix~\ref{app:Diff}. In this context, explicit, sufficient conditions for the continuity assumption~\ref{ass:psi-ldp} may be formulated. 

\begin{proposition}[Assumption~\ref{ass:psi-ldp} for differentiable constraints]\label{prop:ass-ldp-diff}
    Assume that for any $\mu \in \ps_\phi(E)$, for any $t \in \T$, the function $\Psi_t$ is differentiable at $\mu$ w.r.t. the set of directions $\ps_\phi(E)$, in the sense of Definition~\ref{def:diff}. If the conditions:
    \begin{enumerate}[label=(\roman*),ref=\roman*]
        \item\label{it:ass-ldp-diff:1} for any compact set $K \subset \ps_\phi(E)$, there exists $D_K^\Psi \in [0,\infty)$ such that
        \[
            \forall x \in E, \qquad \sup_{(t,\mu) \in \T \times K} \left|\frac{\delta \Psi_t}{\delta\mu}(\mu,x)\right| \leq D_K^\Psi[1+\phi(x)];
        \]
        \item\label{it:ass-ldp-diff:2} for any compact set $K \subset \ps_\phi(E)$, the family of functions $(x \mapsto \frac{\delta \Psi_t}{\delta \mu}(\mu,x))_{t \in \T, \mu \in K}$ is equi-continuous;
        \item\label{it:ass-ldp-diff:3} for any $x \in E$, the family of functions $(\mu \mapsto \frac{\delta \Psi_t}{\delta \mu}(\mu,x))_{t \in \T}$ is (sequentially) equi-continuous on $\ps_\phi(E)$;
    \end{enumerate}
    are satisfied, then Assumption~\ref{ass:psi-ldp} holds.
\end{proposition}

Proposition~\ref{prop:ass-ldp-diff} is proved in Section~\ref{ss:proofs-LDP}. 

\begin{rem}[Linear constraints]\label{rem:lin-ineq-constr-diff}
    In the setting of Remark~\ref{rem:lin-ineq-constr}, the conditions (\ref{it:ass-ldp-diff:1},\ref{it:ass-ldp-diff:2},\ref{it:ass-ldp-diff:3}) hold as soon as the family $(\psi_t)_{t \in \T}$ is equi-continuous on $E$ and there exists $D^\psi \in [0,\infty)$ such that
    \[
        \forall x \in E, \qquad \sup_{t \in \T} |\psi_t(x)| \leq D^\psi [1+\phi(x)].
    \]
\end{rem}

\subsubsection{Application: the conditional Sanov Theorem with linear constraints}

A particular setting where Theorem~\ref{thm:AbsLDP} may be applied is the case where $\Pi^N$ is the law of the empirical measure of independent random variables distributed according to $\nu$. In this case, the LDP is known to hold with $\mathcal{I}(\mu) = H(\mu|\nu)$ on $\ps_\phi(E)$, either with $\phi \equiv 1$ (this is the usual Sanov Theorem), or with $\phi(x) = d(x,x_0)^p$, $p \geq 1$ under the exponential integrability assumption~\eqref{eq:phi-nu} in~\ref{ass:FiniteF} (this is the Sanov Theorem in Wasserstein distance~\cite[Theorem 1.1]{wang2010sanov}).

In this case, and if the inequality constraints are actually linear, as in Remark~\ref{rem:lin-ineq-constr}, then defining
\[
    E_\eta := \{ x \in E: \; \forall \t \in \T, \, \psi_{\t} (x) \leq -\eta \},
\]
we have that a sufficient condition for~\ref{ass:constr-qual-ldp}-\eqref{ass:constr-qual-ldp:1} to hold is 
\begin{equation}\label{eq:cond-suff-lin}
    \exists \eta > 0: \quad \nu(E_\eta)>0, 
\end{equation}
because then we may take $\tilde\mu = \nu(\cdot|E_\eta)$, which satisfies $\langle \tilde\mu,\psi_t\rangle \leq -\eta$ for any $t$. This condition also guarantees that the conditions~\eqref{eq:PiNAdineq-pos} and~\eqref{eq:min-pb-ineq} are satisfied. Thus, the combination of this observation with Remark~\ref{rem:lin-ineq-constr-diff} yields the following statement.

\begin{proposition}[Conditional Sanov Theorem with linear constraints]\label{prop:cond-sanov}
    Let $\phi(x)=d(x,x_0)^p$ for $p=0$ or $p \geq 1$, and let $\nu \in \mathcal{P}(E)$ be such that the integrability condition~\eqref{eq:phi-nu} holds. Let $\Pi^N$ be the law of the empirical measure of $N$ independent random variables distributed according to $\nu$. Let $(\psi_t)_{t \in \T}$ be a family of functions $E \to \R$ satisfying the conditions of Remark~\ref{rem:lin-ineq-constr-diff}, and let the set $\Adineq$ be defined accordingly. If this family satisfies~\eqref{eq:cond-suff-lin}, then the sequence $(\Pi^N(\cdot|\Adineq))_{N \geq 1}$ is well-defined and satisfies a LDP on $\ps_\phi(E)$ with good rate function 
    \begin{equation*}
        \begin{cases}
            \displaystyle H(\mu|\nu) - \inf_{\mu' \in \Adineq} H(\mu'|\nu), & \text{if $\mu \in \Adineq$,}\\
            +\infty, & \text{otherwise.}
        \end{cases}
    \end{equation*}
\end{proposition}

In the setting of Proposition~\ref{prop:cond-sanov}, the set $\Adineq$ is convex and the (good) rate function is strictly convex, so it admits a unique minimiser $\overline\mu$ and $\Pi^N(\cdot|\Adineq)$ converges to $\delta_{\overline\mu}$. A more general discussion on the convex case is postponed to Section~\ref{ss:convex}.

\begin{rem}[Convergence for conditioned particles]
Denoting by $\vec{X}^N = (X^1, \ldots, X^N)$ the underlying particle system, whose empirical measure $\pi(\vec{X}^N )$ has distribution $\Pi^N$, this statement implies that the conditional distribution $\mathcal{L}(X^1 \vert \pi(\vec{X}^N ) \in \Adineq )$ of the first particle converges weakly to $\overline\mu$. This convergence can be quantified: we deduce from~\cite[Theorem~1]{csiszar1984sanov} that
\[ \forall N \geq 1, \quad H( \mathcal{L}(X^1 \vert \pi(\vec{X}^N ) \in \Adineq ) \vert \overline\mu ) \leq -\frac{1}{N} \log \Pi^N(\Adineq) - H( \overline\mu \vert \nu). \]
More explicit rates are given in \cite{dembo1996refinements,dembo1998re,cattiaux2007deviations}.
\end{rem}

\subsection{Gibbs density for minimisers}\label{ss:gibbs-density}

In this section we focus on the minimisation problem~\eqref{eq:min-pb}, which takes both equality and inequality constraints into account. We show that minimisers, if they exist, have a density with respect to $\nu$, whose form generalises~\eqref{eq:gibbs-density} to the infinite-dimensional setting. 
To take into account both the interaction part $\mathcal{F}$ in the rate function, and the nonlinearity of the inequality constraints $\Psi_t$, we need to assume differentiability properties of these functionals, at a given $\overline{\mu} \in \ps_\phi(E)$.

\begin{assumptiona}[Differentiability of $\F$ at $\overline{\mu}$] \label{ass:DiffF-barmu}
    The set $\AdF$ is convex, and $\F$ is differentiable at $\overline\mu$ w.r.t. the set of directions $\AdF$.
\end{assumptiona}

\begin{assumptiona}[Regularity of $\Psi_t$ and constraint qualification at $\overline{\mu}$]\label{ass:regPsi-barmu}
$\phantom{a}$
    \begin{enumerate}[label=(\roman*),ref=\roman*]
        \item\label{ass:regPsi-barmu:1} The function $t \mapsto \Psi_t(\overline\mu)$ is continuous.
        \item\label{ass:regPsi-barmu:2} For any $t \in \T$, $\Psi_t$ is differentiable at $\overline\mu$ w.r.t. the set of directions $\ps_\phi(E)$, uniformly in $t$, in the sense that for any $\mu \in \ps_\phi(E)$,
        \begin{equation*}
            \lim_{\varepsilon \to 0} \sup_{t \in \T} \left|\frac{\Psi_t((1-\varepsilon)\overline\mu + \varepsilon\mu)-\Psi_t(\overline\mu)}{\varepsilon} - \left\langle \mu-\overline\mu, \frac{\delta \Psi_t}{\delta\mu}(\overline\mu)\right\rangle\right|=0.
        \end{equation*}
        \item\label{ass:regPsi-barmu:3} For any $x \in E$, the function $t \mapsto \frac{\delta \Psi_t}{\delta \mu}(\overline\mu,x)$ is continuous.
        \item\label{ass:regPsi-barmu:4} There exists $D^\Psi \in [0,\infty)$ such that
        \begin{equation*}
            \forall x \in E, \qquad \sup_{t \in \T} \left|\frac{\delta \Psi_t}{\delta\mu}(\overline\mu,x)\right| \leq D^\Psi [1+\phi(x)].
        \end{equation*}
        \item\label{ass:regPsi-barmu:5} There exist $\tilde\mu$ in $\AdF \cap \Adeq$ and $\tilde\varepsilon > 0$ such that
        \begin{equation*}
            \forall \t \in \T, \quad \Psi_{\t}(\overline\mu) + \tilde\varepsilon \bigg\langle \tilde\mu - \overline\mu, \frac{\delta\Psi_{\t}}{\delta\mu}(\overline\mu) \bigg\rangle < 0.
        \end{equation*}
    \end{enumerate}
\end{assumptiona}

\begin{rem}\label{rem:continuity-new}
    By~\ref{ass:regPsi-barmu}-(\ref{ass:regPsi-barmu:1},\ref{ass:regPsi-barmu:3},\ref{ass:regPsi-barmu:4}), for any $\mu \in \ps_\phi(E)$ and any $\varepsilon>0$, the function $t \mapsto \Psi_{\t}(\overline\mu) + \varepsilon \langle \mu - \overline\mu, \frac{\delta\Psi_{\t}}{\delta\mu}(\overline\mu) \rangle$ is continuous. In particular, \ref{ass:regPsi-barmu}-\eqref{ass:regPsi-barmu:5} equivalently rewrites
    \[
        \sup_{t \in \T} \Psi_{\t}(\overline\mu) + \tilde\varepsilon \bigg\langle \tilde\mu - \overline\mu, \frac{\delta\Psi_{\t}}{\delta\mu}(\overline\mu) \bigg\rangle < 0.
    \]
The condition \ref{ass:regPsi-barmu}-\eqref{ass:regPsi-barmu:2} can be verified under equi-continuity assumptions on $\tfrac{\delta\Psi_t}{\delta\mu}$ in the spirit of Proposition \ref{prop:ass-ldp-diff}.
\end{rem}

The proof of Theorem~\ref{thm:abstractGibbs} also requires the following technical condition.

\begin{assumptiona} \label{ass:Fboundedens-barmu}
    Every probability measure on $E$ that has a bounded density with respect to $\overline\mu$ or to the measure $\tilde\mu$ given by~\ref{ass:regPsi-barmu}-\eqref{ass:regPsi-barmu:5} belongs to $\AdF$.
\end{assumptiona}

For instance, Assumption \ref{ass:Fboundedens-barmu} holds when $\F \equiv 0$, as in Sanov's theorem.

\begin{theorem}[Gibbs density for minimisers of~\eqref{eq:min-pb}] \label{thm:abstractGibbs} 
Let~\ref{ass:FiniteF}-\ref{ass:linear-constr}-\ref{ass:nonlinear-constr} hold, and let $\overline{\mu}$ be a minimiser for \eqref{eq:min-pb}, which satisfies~\ref{ass:DiffF-barmu}-\ref{ass:regPsi-barmu}-\ref{ass:Fboundedens-barmu}. 
Then, the following holds
\begin{enumerate}
    \item\label{it:thmAbsAbsCont} Any $\mu \in \AdF \cap \Adeq$ is absolutely continuous w.r.t. $\overline\mu$.
    \item\label{it:thmAbsGibbs} There exist $(\overline\zeta,\overline{\lambda}) \in L^1(E,\d \overline\mu) \times \M_+(\T)$ and a measurable $\overline{S} \subset E$ 
    such that $\overline\mu(\overline{S})=1$,
    \begin{equation}\label{eq:GibbsMeasure:Z}
        \overline{Z} := \int_E \1_{\overline{S}}(x) \exp \biggl[ - \frac{\delta \F}{\delta \mu}(\overline\mu,x) - \overline\zeta (x)- \int_\T \frac{\delta\Psi_{\t}}{\delta \mu}(\overline\mu,x) \overline{\lambda}( \d \t) \biggr] \nu(\d x) \in (0,\infty),
    \end{equation}
    and $\overline\mu$ has density
    \begin{equation} \label{eq:Gibbsmeasure}
        \frac{\d \overline\mu}{\d \nu} \bigl( x \bigr) = \frac{1}{\overline{Z}} \1_{\overline{S}} (x) 
        \exp \biggl[ - \frac{\delta \F}{\delta \mu}(\overline\mu,x) - \overline\zeta (x)- \int_\T \frac{\delta\Psi_{\t}}{\delta \mu}(\overline\mu,x) \overline{\lambda}( \d \t) \biggr]
    \end{equation}
    with respect to $\nu$.
    \item\label{it:thmAbsSlack} The function $\overline\zeta$ belongs to the closure of $\mathrm{Span}(\zeta_s, s \in \S)$ in $L^1(E, \d\overline\mu)$ hence 
    \begin{equation}\label{eq:int-zeta-barmu}
        \int_E \overline\zeta \, \d \overline\mu = 0,
    \end{equation}
    and the complementary slackness condition is satisfied:
    \begin{equation} \label{eq:compSlackness}
        \Psi_{\t}(\overline{\mu}) = 0 \quad \text{for } \overline{\lambda}\text{-a.e. } \t\in\T.
    \end{equation}
\end{enumerate}
\end{theorem}

The proof of Theorem~\ref{thm:abstractGibbs} is detailed in Section~\ref{ss:proofs-gibbs-density}. Remark \ref{rem:BoundLambda} there provides a bound on the total mass of $\overline\lambda$ in terms of $\overline{\mu}, \, \tilde\mu, \, \tilde\varepsilon \text{ and } \tfrac{\delta\F}{\delta\mu}(\overline{\mu})$. 
In the particular case $\F \equiv 0$ and $\T = \emptyset$ (no inequality constraint), we recover the result of \cite[Theorem 3]{csiszar1975divergence}.

\begin{rem}[Equivalence between $\overline\mu$ and $\nu$] \label{rem:equivalence}
In general, $\overline{\mu}$ is not equivalent to $\nu$; the set $\overline{S}$ where $\tfrac{\d\overline\mu}{\d\nu}$ is positive is thus needed, contrary to the Gibbs density \eqref{eq:gibbs-density} in Section \ref{subsec:IntroGibbs}.
This loss of equivalence only comes from the equality constraints. 
Indeed, if $\S = \emptyset$, we notice that $\nu \in \AdF \cap \Adeq$, so that $\nu \ll \overline{\mu}$ and $\nu ( \overline{S} ) = 1$ from Theorem \ref{thm:abstractGibbs}-\ref{it:thmAbsAbsCont}. 
If $\S = \emptyset$, we of course have $\overline\zeta \equiv 0$.
The need for $\overline{S}$ for handling equality constraints is well-known in the setting of the Schrödinger bridge, see Section \ref{ss:schro}.
From Theorem \ref{thm:abstractGibbs}-\ref{it:thmAbsAbsCont}, we eventually deduce that a necessary and sufficient condition for $\overline{\mu} \sim \nu$ is the existence of $\tilde\nu \in \AdF \cap \Adeq$ such that $\tilde\nu \sim \nu$.
\end{rem}

\begin{rem}[Assumption~\ref{ass:regPsi-barmu} for linear inequality constraints] \label{rem:LinIneqCons} 
In the case of linear inequality constraints of Remark~\ref{rem:lin-ineq-constr}, Assumption~\ref{ass:regPsi-barmu}-\eqref{ass:regPsi-barmu:2} always holds true. Assumptions~\ref{ass:regPsi-barmu}-(\ref{ass:regPsi-barmu:1},\ref{ass:regPsi-barmu:3},\ref{ass:regPsi-barmu:4}) are satisfied as soon as for any $x \in E$, $t \mapsto \psi_t(x)$ is continuous and there exists $D^\psi \in [0,\infty)$ such that 
\[
    \forall x \in E, \qquad \sup_{t \in \T} |\psi_t(x)| \leq D^\psi[1+\phi(x)].
\]
Last, Assumption~\ref{ass:regPsi-barmu}-\eqref{ass:regPsi-barmu:5} amounts to assuming that there exists $\tilde\mu\in\AdF \cap \Adeq$ such that $\langle \tilde\mu, \psi_t \rangle < 0$ for all $t \in \T$. 
For the present case of linear constraints, the assumptions of Theorem~\ref{thm:abstractGibbs} may be relaxed. 
In particular, no upper bound on the $\psi_t$ needs to be assumed.
This alternative result is detailed in Appendix \ref{sec:appLin}.
\end{rem}

\begin{rem}[Sufficient condition for~\ref{ass:regPsi-barmu}-\eqref{ass:regPsi-barmu:5} without equality constraints]
If $\S = \emptyset$ (no equality constraints), and $\overline{\mu}$ gives a positive measure to the set
\[ U := \bigg\{ x \in E, \; \forall \t \in \T, \, \frac{\delta\Psi_t}{\delta\mu}(\overline\mu, x) < 0 \bigg\}, \]
we notice that $\tilde\mu := \overline{\mu}(\cdot | U)$ satisfies \ref{ass:regPsi-barmu}-\eqref{ass:regPsi-barmu:5}, for any $\tilde\varepsilon>0$, if we know that $\F ( \tilde\mu ) < +\infty$. The latter condition may follow from~\ref{ass:Fboundedens-barmu} since $\tilde\mu$ has a bounded density w.r.t. $\overline\mu$.
\end{rem}

\subsection{Sufficient conditions and stability in the convex case}\label{ss:convex}

In this section, we work under the following assumption.

\begin{assumptiona}[Convexity]\label{ass:diff-glob-conv}
    The functions $\F$ and $\Psi_{\t}$, $t \in \T$, are convex on $\ps_\phi(E)$.
\end{assumptiona}

Then, the function $\mu \mapsto H(\mu|\nu)$ being strictly convex, the sets $\AdF$ and $\Ad$ are convex, and there is at most one minimiser for~\eqref{eq:min-pb}. 

\medskip

In this setting, we first show a converse statement to Theorem~\ref{thm:abstractGibbs}, namely that an admissible measure which satisfies the conditions~\eqref{eq:Gibbsmeasure}-\eqref{eq:compSlackness} is the minimiser for~\eqref{eq:min-pb}.

\begin{theorem}[Sufficient conditions in the convex case]\label{thm:convex-case}
    Let~\ref{ass:FiniteF}-\ref{ass:linear-constr}-\ref{ass:nonlinear-constr}-\ref{ass:diff-glob-conv} hold. Let $\overline\mu \in \AdF \cap \Ad$ satisfy the following conditions.
    \begin{enumerate}[label=(\roman*),ref=\roman*]
        \item\label{it:SuffDiff} The functions $\F$ and $\Psi_t$, $t \in \T$, are differentiable at $\overline\mu$ w.r.t. the set of directions $\AdF \cap \Ad$.
        \item\label{it:SuffDom} The function $(t,x) \mapsto \frac{\delta \Psi_t}{\delta\mu}(\overline\mu,x)$ is measurable, and there exists $D^\Psi \in [0,\infty)$ such that
        \begin{equation*}
            \forall x \in E, \qquad \sup_{t \in \T} \left|\frac{\delta \Psi_t}{\delta\mu}(\overline\mu,x)\right| \leq D^\Psi[1+\phi(x)].
        \end{equation*}
        \item\label{it:SuffEq} There exist $\overline{\zeta}$ in the $L^1(E,\d\overline\mu)$-closure of $\mathrm{Span}(\zeta_s, s \in \S)$ (which necessarily satisfies~\eqref{eq:int-zeta-barmu}), $\overline\lambda \in \M_+(\I)$ and a measurable subset $\overline{S} \subset E$ such that $\overline\mu(\overline{S})=1$ and~\eqref{eq:GibbsMeasure:Z}-\eqref{eq:Gibbsmeasure}-\eqref{eq:compSlackness} hold.
        \item\label{it:SuffSupp} $\nu(\overline{S})=1$ and for any $\mu \in \AdF \cap \Ad$, $\overline{\zeta} \in L^1(E,\d\mu)$ and $\langle \mu, \overline{\zeta}\rangle = 0$.
    \end{enumerate}
    Then, the measure $\overline\mu$ is the unique minimiser for \eqref{eq:min-pb}.
\end{theorem}

When only equality constraints are considered, corresponding to $\T = \emptyset$, Condition \eqref{it:SuffSupp} can be verified using \cite[Lemma 2.18]{nutz2021introduction}.

We now prove stability for~\eqref{eq:min-pb} when moving the inequality constraints $\Psi_t$ or the reference measure. 
To this aim, we will use the following sufficient condition for the qualification condition~\ref{ass:regPsi-barmu}-\eqref{ass:regPsi-barmu:5} to hold. 

\begin{lemma}[Sufficient condition for qualification]\label{lem:EquivQualif}
    Under~\ref{ass:diff-glob-conv}, if there exists $\tilde\mu \in \AdF \cap \Adeq$ such that
    \[
        \forall t \in \T, \qquad \Psi_t(\tilde\mu) < 0,
    \]
    then the condition~\ref{ass:regPsi-barmu}-\eqref{ass:regPsi-barmu:5} holds, with $\tilde\varepsilon=1$, at any $\overline\mu \in \AdF \cap \Ad$ at which all the $\Psi_t$ are differentiable (w.r.t. the set of directions $\AdF \cap \Adeq$).
\end{lemma}

Lemma~\ref{lem:EquivQualif} is an immediate consequence of the convexity of the constraints (see e.g.~\eqref{eq:convex-diff} below) and we do not detail it.

\medskip 
We first state a quantitative result when only the inequality constraints are relaxed. To this aim, for any $\varepsilon \geq 0$, we introduce the notation

\newcommand{\Adineqeps}{\mathcal{A}_{\Psi,\varepsilon}}
\newcommand{\Adeps}{\mathcal{A}_{\Psi,\varepsilon}^\zeta}
\newcommand{\Adepsstar}{\mathcal{A}_{\Psi,\varepsilon_*}^\zeta}

\begin{equation*}
    \Adineqeps := \{ \mu \in \ps_\phi(E), \; \forall t \in \T, \;  \Psi_t ( \mu ) \leq \varepsilon \}, \qquad \Adeps := \Adeq \cap \Adineqeps.
\end{equation*}

\begin{proposition}[Quantitative stability w.r.t. inequality constraints] \label{pro:StrongStab}
    Let~\ref{ass:FiniteF}-\ref{ass:linear-constr}-\ref{ass:nonlinear-constr}-\ref{ass:diff-glob-conv} hold, and assume that $\F$ is bounded from below on $\ps_\phi(E)$, and that for any $\mu \in \Adeq$, the function $t \mapsto \Psi_t(\mu)$ is measurable. We consider the minimisation problem 
    \begin{equation} \label{eq:min-pb-eps}
        \overline{\mathcal{I}}^\zeta_{\Psi,\varepsilon} := \inf_{\mu \in \AdF \cap \Adeps} \mathcal{I}(\mu).
    \end{equation} 
    We assume that there is $\varepsilon_*>0$ such that~\ref{ass:DiffF-barmu}, \ref{ass:regPsi-barmu}-(\ref{ass:regPsi-barmu:1},\ref{ass:regPsi-barmu:2},\ref{ass:regPsi-barmu:3},\ref{ass:regPsi-barmu:4}) hold at every $\overline\mu \in \AdF \cap \Adepsstar$, that the condition of Lemma~\ref{lem:EquivQualif} is satisfied, and that~\ref{ass:Fboundedens-barmu} holds for every $\overline\mu \in \AdF \cap \Adepsstar$ and for $\tilde\mu$ given by Lemma~\ref{lem:EquivQualif}.

    \begin{enumerate}
        \item\label{it:StabMin} For every $\varepsilon \in [0,\varepsilon_*]$, \eqref{eq:min-pb-eps} has a unique minimiser $\overline{\mu}_\varepsilon$, and this minimiser satisfies the conclusions of Theorem~\ref{thm:abstractGibbs} for some $(\overline\zeta_\varepsilon,\overline{\lambda}_\varepsilon) \in L^1(E,\d \overline\mu_\varepsilon) \times \M_+(\T)$ and a measurable set $\overline{S}_\varepsilon \subset E$.
        \item\label{it:StabValue} Assume in addition that for any $\varepsilon \in [0,\varepsilon_*]$, $\nu(\overline{S}_\varepsilon)=1$, and that for any $\mu \in \Adeq$, $\overline\zeta_\varepsilon \in L^1(E,\d\mu)$ with $\langle \mu, \overline\zeta_\varepsilon\rangle=0$. Then the function $\varepsilon \mapsto \overline{\mathcal{I}}^\zeta_{\Psi,\varepsilon}$ is absolutely continuous,
        \[ \frac{\d}{\d \varepsilon} \overline{\mathcal{I}}^\zeta_{\Psi,\varepsilon} = -\overline\lambda_\varepsilon (\T), \quad \text{Lebesgue-a.e.}, \]
        and there exists $C_\mathrm{stab} \in [0,\infty)$ such that, for any $\varepsilon \in [0,\varepsilon_*]$,
        \begin{equation}\label{eq:quant-stab-I}
            0 \leq \overline{\mathcal{I}}^\zeta_{\Psi,0} - \overline{\mathcal{I}}^\zeta_{\Psi,\varepsilon} \leq C_\mathrm{stab} \varepsilon.
        \end{equation}
    \end{enumerate}
\end{proposition}

In the case $\F \equiv 0$, the estimate~\eqref{eq:quant-stab-I} reads
$0 \leq H(\overline\mu_0|\nu)-H(\overline\mu_\varepsilon|\nu) \leq C_\mathrm{stab} \varepsilon$. 
The Pythagorean identity for entropic projections~\cite[Theorem 2.2]{csiszar1975divergence} then provides the following quantitative stability estimate on the minimiser $\overline\mu_\varepsilon$ of~\eqref{eq:min-pb-eps}.

\begin{corollary}
If $\F \equiv 0$ in the setting of Proposition \ref{pro:StrongStab}, then  
\[
        H(\overline\mu_0|\overline\mu_\varepsilon) \leq C_\mathrm{stab} \varepsilon.
    \]
\end{corollary}

We finally address weak stability when $\nu$, $\F$ and the $\Psi_t$ are perturbed at the same time. 
Since interesting results already exist in the setting of equality constraints \cite{ghosal2022stability,eckstein2022quantitative,nutz2023stability,chiarini2023gradient,divol2024tight}, we only focus on inequality constraints;
therefore, we assume that $\S = \emptyset$.
Let $(\nu_k)_{k \geq 1}$ and $\nu$ be in $\ps_\phi(E)$, and $(\F_k)_{k \geq 1}$ and $\F$ be functions $\ps_\phi(E) \to (-\infty,+\infty]$.
Let $\mathcal{I}_k : \mu \mapsto H ( \mu \vert \nu_k ) + \F_k ( \mu )$ be the related rate function with domain $\mathcal{D}_{\mathcal{I}_k}$.  
For each $t \in \T$, let $(\Psi_{t,k})_{k \geq 1}$ be functions $\ps_\phi(E) \to \R$.

\begin{assumptiona}[Setting for stability] \label{ass:Stab}
$\phantom{a}$
\begin{enumerate}[label=(\roman*),ref=\roman*]
\item\label{it:StabRate} The functions $\mathcal{I}$, $( \mathcal{I}_k )_{k \geq 1}$, $ ( \Psi_t )_{t \in \T}$ and $ ( \Psi_{t,k} )_{t \in \T, k \geq 1}$ are lower semi-continuous. 
Moreover, for every $k \geq 1$, $\mathcal{I}_k$ has compact level sets in $\ps_\phi (E)$. 
\item\label{it:StabConv} The functions $\F$, $( \F_k )_{k \geq 1}$, $ ( \Psi_t )_{t \in \T}$ and $ ( \Psi_{t,k} )_{t \in \T, k \geq 1}$ are convex.
\item\label{it:StabDensB} For every $k \geq 1$, any measure that has a bounded density w.r.t. $\nu_k$ belongs to $\mathcal{D}_{\mathcal{I}_k}$.
\item\label{it:StabIntExp} $\nu_k \to \nu$ in $\ps_\phi(E)$, and $\sup_{k \geq 1} \langle \nu, e^{\alpha \phi} \rangle + \langle \nu_k, e^{\alpha \phi} \rangle$ is finite, for every $\alpha >0$.
\item\label{it:StabSemCont} For any $\mu \in \ps_\phi(E)$, $\lim_{k \to +\infty} \sup_{t \in \T} \Psi_{t,k}(\mu) \leq \sup_{t \in \T} \Psi_t(\mu)$.
\item\label{it:StabCont} If $( \mu_l )_{l \geq 1}$ converges to $\mu$ in $\ps_\phi(E)$, then $\sup_{k \geq 1} \vert \sup_{t \in \T} \Psi_{t,k}(\mu_l) - \sup_{t \in \T} \Psi_{t,k}(\mu) \vert \to 0$,
and $\vert \sup_{t \in \T} \Psi_{t}(\mu_l) - \sup_{t \in \T} \Psi_{t}(\mu) \vert \to 0$.
\item\label{it:StabQualif} There exists $\tilde\mu \in \AdF$ such that $\sup_{t \in \T} \Psi_t(\tilde\mu) < 0$.
\end{enumerate}
\end{assumptiona}

\newcommand{\Adineqk}{\mathcal{A}_{\Psi,k}}
\newcommand{\AdFk}{\mathcal{D}_{\mathcal{I}_k}}

We notice that \ref{ass:Stab}-(\ref{it:StabRate},\ref{it:StabIntExp}) make $\mathcal{I}_k$, $\nu_k$ and $( \Psi_{t,k} )_{t \in \T}$ enter the framework of \ref{ass:FiniteF}-\ref{ass:nonlinear-constr}.
We define the sets $(\Adineqk)_{k \geq 1}$ accordingly.
As previously, \ref{ass:Stab}-\eqref{it:StabCont} can be verified using Proposition \ref{prop:ass-ldp-diff}.
Similarly, a sufficient condition for the compactness of level sets in \ref{ass:Stab}-\eqref{it:StabRate} is given by Lemma~\ref{lem:Igood}.
From \ref{ass:Stab}-\eqref{it:StabRate}, the $(\mathcal{I}_k)_{k \geq 1}$ are good rate functions.
Therefore, using the the convexity assumption \ref{ass:Stab}-\eqref{it:StabConv} (which corresponds to ~\ref{ass:diff-glob-conv}), the minimisation problem
\[
    \overline{\mathcal{I}}_{\Psi,k} := \inf_{\mu \in \AdFk \cap \Adineqk} \mathcal{I}_k(\mu), \] 
has a unique minimiser denoted by $\overline\mu_k$, for every $k \geq 1$.
Let us now verify the qualification condition in Lemma \ref{lem:EquivQualif}. 

\begin{lemma}[Qualification] \label{lem:StabQualif}
Under \ref{ass:Stab}, there exist $( \tilde\mu_k )_{k \geq 1}$ and $k_0 \geq 1$ such that $\tilde\mu_k \in \mathcal{D}_{\mathcal{I}_k}$, $( H(\tilde\mu_k|\nu_k) )_{ k \geq 1}$ is bounded, and
\[ \sup_{k \geq k_0} \sup_{t \in \T} \Psi_{t,k} (\tilde\mu_k) < 0. \]
\end{lemma}

To obtain Gibbs measures and show their stability, we need differentiability assumptions.

\begin{assumptiona}[Local regularity] \label{ass:StabReg}
$\phantom{a}$
\begin{enumerate}[label=(\roman*),ref=\roman*]
\item\label{it:StabRegalif} For every large enough $k$, the regularity assumptions~\ref{ass:DiffF-barmu} and \ref{ass:regPsi-barmu}-(\ref{ass:regPsi-barmu:1},\ref{ass:regPsi-barmu:2},\ref{ass:regPsi-barmu:3},\ref{ass:regPsi-barmu:4}) hold at $\overline\mu_k$ for $(\mathcal{I}_k, (\Psi_{t,k})_{t \in \T})$.
\item\label{it:StabTech} For every $k \geq k_0$, \ref{ass:Fboundedens-barmu} holds at $\overline\mu_k$ and $\tilde{\mu}_k$, given by Lemma \ref{lem:StabQualif}, for $(\mathcal{I}_k, (\Psi_{t,k})_{t \in \T})$.
\item\label{itStabDom} There exists $D^{\F,\Psi} \in [0,\infty)$ such that
\[ \forall x \in E, \qquad \sup_{k \geq 1} \left\{\left|\frac{\delta\F_k}{\delta\mu}(\overline\mu_k,x)\right| + \sup_{t \in \T} \left|\frac{\delta\Psi_{t,k}}{\delta\mu}(\overline\mu_k,x)\right|\right\} \leq D^{\F,\Psi}[1+\phi(x)]. \]
\item\label{it:StabEqui}
For any sub-sequence $(\overline\mu_{l_k})_{k \geq 1}$ that converges towards some $\overline\mu_\infty$ in $\ps_\phi (E)$, $\overline\mu_\infty$ belongs to $\AdF$, \ref{ass:DiffF-barmu} and \ref{ass:regPsi-barmu}-(\ref{ass:regPsi-barmu:1},\ref{ass:regPsi-barmu:2},\ref{ass:regPsi-barmu:3},\ref{ass:regPsi-barmu:4}) hold at $\overline\mu_\infty$ for $(\mathcal{I},(\Psi_t)_{t \in \T})$, and for every $x \in E$,
\[ \bigg\vert \frac{\delta \F_{l_k}}{\delta\mu}(\overline\mu_{l_k},x)-\frac{\delta \F}{\delta\mu}(\overline\mu_\infty,x) \bigg\vert \xrightarrow[k \rightarrow +\infty]{} 0, \] 
\[ \sup_{t \in \T} \big \vert \Psi_{t,l_k}(\overline\mu_{l_k} )-\Psi_t (\overline\mu_\infty ) \big\vert +
\bigg\vert \frac{\delta \Psi_{t,l_k}}{\delta\mu}(\overline\mu_{l_k},x)-\frac{\delta \Psi_t}{\delta\mu}(\overline\mu_\infty,x) \bigg\vert \xrightarrow[k \rightarrow +\infty]{} 0. \]
Moreover, $( x \mapsto \tfrac{\delta \F_{l_k}}{\delta\mu} (\overline\mu_{l_k}, x ) )_{k \geq 1}$, $( x \mapsto \tfrac{\Psi_{t,l_k}}{\delta\mu} ( \overline\mu_{l_k}, x) )_{t \in \T, k \geq 1}$ are equi-continuous on $E$.
\end{enumerate}
\end{assumptiona}

The regularity assumptions in \ref{ass:StabReg}-\eqref{it:StabEqui} are similar to the ones in Proposition \ref{prop:ass-ldp-diff}.

\begin{theorem}[Weak stability] \label{thm:WeakStab} 
$\phantom{a}$
\begin{enumerate}
    \item\label{it:StabComp} Under \ref{ass:Stab} and \ref{ass:StabReg}-(\ref{it:StabRegalif},\ref{it:StabTech},\ref{itStabDom}), for every large enough $k$, $\overline\mu_k$ writes as Gibbs measure~\eqref{eq:Gibbsmeasure}-\eqref{eq:compSlackness} for some $\overline\lambda_k \in \M_+(\T)$. Moreover, the sequence $(\overline\lambda_k)_{k}$ is weakly precompact in $\M_+(\T)$.
    \item\label{it:StabWeakCV} Under \ref{ass:Stab}-\ref{ass:StabReg}, $\mathcal{I}$ has a unique minimiser $\overline{\mu}$ in $\AdF \cap \Adineq$, which is the limit of $( \overline\mu_k )_{k \geq 1}$ in $\ps_\phi(E)$. Moreover, \eqref{eq:Gibbsmeasure}-\eqref{eq:compSlackness} hold at $\overline{\mu}$ for any limit point $\overline\lambda_\infty$ of $(\overline\lambda_k)_{k}$.
\end{enumerate}
\end{theorem}
Since we are working with $\mathcal{S}=\emptyset$, following Remark~\ref{rem:equivalence}, in the statement of Theorem \ref{thm:WeakStab}-\ref{it:StabComp}, there is no need to for introducing neither the set $\overline S_k$ nor the function $\overline\zeta_k$.

Theorem~\ref{thm:convex-case}, Proposition~\ref{pro:StrongStab}, Lemma~\ref{lem:StabQualif} and Theorem~\ref{thm:WeakStab} are proved in Section~\ref{ss:proofs-convex}.

\section{Application to stochastic processes}\label{s:processes}

In this section we fix $T>0$ and $d \geq 1$.
We consider the case where $E$ is the set of continuous trajectories $C([0,T],\R^d)$, endowed with the sup norm. 
In this setting, we denote by $x_{[0,T]} = (x_t)_{t \in [0,T]}$ typical elements of $C([0,T],\R^d)$, and by $\mu_{[0,T]}$ typical elements of $\ps(C([0,T],\R^d))$. If $\mathsf{X}_{[0,T]}$ is a random variable in $C([0,T],\R^d)$ with law $\mu_{[0,T]}$, then the marginal distribution of $\mathsf{X}_t$ is denoted by $\mu_t$. 
We focus on inequality constraints of the form 
\begin{equation}\label{eq:constr-time-marg}
    \forall t \in \T = [0,T], \qquad \Psi_t(\mu_{[0,T]}) = \Psi(\mu_t),
\end{equation}
for some function $\Psi : \ps(\R) \to \R$.
We moreover focus on the case $\F \equiv 0$, so that minimisers of~\eqref{eq:min-pb} describe the asymptotic behaviour of large systems of independent and identically distributed continuous stochastic processes, conditionally on the constraint~\eqref{eq:constr-time-marg} on their empirical distribution at all times $t \in [0,T]$.

In Section~\ref{ss:Psi-Psit}, we provide conditions on the function $\Psi$ and on the reference measure $\nu_{[0,T]}$ ensuring that the assumptions of the main theorems of Section~\ref{s:abstract-results} are satisfied. We next present two examples of applications of these theorems: in Section~\ref{ss:schro}, we construct and study a constrained version of the Schrödinger bridge, and in Section~\ref{sec:Gibbsdif}, we state a Gibbs principle for i.i.d. diffusion processes under constraints of the form~\eqref{eq:constr-time-marg}. In particular, we provide a detailed description of the law of the minimiser $\overline \mu_{[0,T]}$ as the law of a diffusion process with tilted initial condition and modified drift coefficient with respect to the reference measure $\nu_{[0,T]}$. To illustrate the latter result, we present simple examples in Section~\ref{ss:brown-expl}.

\subsection{Constraints on time marginal laws}\label{ss:Psi-Psit}

\subsubsection{Global setting}

Throughout Section~\ref{s:processes}, we work with $E = C([0,T],\R^d)$ and set $\phi(x_{[0,T]}) := \sup_{0 \leq t \leq T} \vert x_t \vert^p$ for some $p \in [1,+\infty)$, so that the topology on $\ps_\phi (E) = \ps_p ( C ( [0,T], \R^d) )$ corresponds to the one induced by the $p$-Wasserstein distance $W_p$. 
In this context, we have the following standard properties, which follow from elementary coupling arguments.

\begin{lemma}[Metric properties of time marginal distributions]\label{lem:mutW1} $\phantom{a}$
    \begin{enumerate}[ref=\roman*,label=(\roman*)]
        \item\label{it:mutW1:1} For any $t \in [0,T]$, the map 
        \[
            \begin{array}{ccc}
                \left(\ps_p(C ( [0,T], \R^d) ), W_p\right) & \to & \left(\ps_p( \R^d ), W_p\right)\\
                \mu_{[0,T]} & \mapsto & \mu_t
            \end{array}
        \]
        is $1$-Lipschitz continuous.
        \item\label{it:mutW1:2} For any $\mu_{[0,T]} \in \ps_p(C ( [0,T], \R^d) )$, the map $t \mapsto \mu_t$ is continuous.
    \end{enumerate}
\end{lemma}

We shall always take $\F \equiv 0$, so that given a reference probability measure $\nu_{[0,T]}$ on $C ( [0,T], \R^d)$, Assumption~\ref{ass:FiniteF} holds as soon as
\begin{equation}\label{eq:exp-integ-proc}
    \forall \alpha > 0, \qquad \E_{\nu_{[0,T]}} \left[e^{\alpha \sup_{0 \leq t \leq T} \vert {\sf X}_t \vert^p}\right] < \infty,
\end{equation}
where the notation $\E_{\nu_{[0,T]}}[\cdots]$ means that the process $\mathsf{X}_{[0,T]}$ has distribution $\nu_{[0,T]}$. Moreover, the set $\AdF$ rewrites
\[
    \AdF = \left\{\mu_{[0,T]} \in \ps_p(C([0,T],\R^d)), H(\mu_{[0,T]}|\nu_{[0,T]}) < \infty\right\}.
\]

Equality constraints are represented by a family $( \zeta_s )_{s \in \S}$ of functions $C([0,T],\R^d) \rightarrow \R$ which satisfy~\ref{ass:linear-constr}. Last, we fix a function $\Psi : \ps_p(\R^d) \to \R$ and define the functions $\Psi_t$, $t \in [0,T]$, by~\eqref{eq:constr-time-marg}. Then~\ref{ass:nonlinear-constr} holds as soon as $\Psi$ is lower semi-continuous on $\ps_p(\R^d)$. Overall, we deduce that in this setting, \ref{ass:FiniteF}-\ref{ass:linear-constr}-\ref{ass:nonlinear-constr} hold under the following condition.

\begin{assumptionb}[Condition for~\ref{ass:FiniteF}-\ref{ass:linear-constr}-\ref{ass:nonlinear-constr}]\label{ass-b:global}
    $\F \equiv 0$, \eqref{eq:exp-integ-proc} holds, $( \zeta_s )_{s \in \S}$ satisfies~\ref{ass:linear-constr} and $\Psi : \ps_p(\R^d) \to \R$ is lower semi-continuous.
\end{assumptionb}

Notice that under~\ref{ass-b:global}, $\mathcal{I}(\mu_{[0,T]}) = H(\mu_{[0,T]}|\nu_{[0,T]})$ is always a good rate function thanks to Lemma~\ref{lem:Igood}, so minimisers to~\eqref{eq:min-pb} exist as soon as $\AdF \cap \Ad$ is nonempty.

\subsubsection{Differentiability of the constraints}

In the assumptions of Theorem~\ref{thm:AbsLDP} (via Proposition~\ref{prop:ass-ldp-diff}) and Theorem~\ref{thm:abstractGibbs}, the linear functional derivative of the function $\Psi_t$ plays a central role. The next lemma, whose proof is straightforward, links the linear functional derivatives of $\Psi_t$ (defined on $\ps_p(C([0,T],\R^d))$) to that of $\Psi$ (defined on $\ps_p(\R^d)$).

\begin{lemma}[Linear functional derivative of $\Psi_t$]\label{lem:diff-Psi-Psit}
    Let $\mu \in \ps_p(\R^d)$ be such that $\Psi$ is differentiable at $\mu$ w.r.t. the set of directions $\ps_p(\R^d)$, and let $t \in [0,T]$. For any $\mu_{[0,T]} \in \ps_p(C([0,T],\R^d))$ such that $\mu_t=\mu$, $\Psi_t$ is differentiable at $\mu_{[0,T]}$ w.r.t. the set of directions $\ps_p(C([0,T],\R^d))$, and for any $x_{[0,T]} \in C([0,T],\R^d)$,
    \[
        \frac{\delta \Psi_t}{\delta\mu_{[0,T]}}\left(\mu_{[0,T]},x_{[0,T]}\right) = \frac{\delta \Psi}{\delta\mu}(\mu_t,x_t).
    \]
\end{lemma}

As a consequence of Lemma~\ref{lem:diff-Psi-Psit}, we get that the assumptions of Proposition~\ref{prop:ass-ldp-diff} are satisfied under the following condition on $\Psi$.

\begin{assumptionb}[Global differentiability of $\Psi$]\label{ass-b:diff-Psi-glob}
    For any $\mu \in \ps_p(\R^d)$, $\Psi$ is differentiable at $\mu$ w.r.t. the set of directions $\ps_p(\R^d)$. Moreover,
    \begin{enumerate}[label=(\roman*),ref=\roman*]
        \item\label{it:diff-Psi-glob:1} for any compact set $K \subset \ps_p(\R^d)$, there exists $D_K^\Psi \in [0,\infty)$ such that
        \[
            \forall x \in \R^d, \qquad \sup_{\mu \in K} \left|\frac{\delta \Psi}{\delta\mu}(\mu,x)\right| \leq D_K^\Psi[1+|x|^p];
        \]
        \item\label{it:diff-Psi-glob:2} for any compact set $K \subset \ps_p(\R^d)$, the family of functions $(x \mapsto \frac{\delta \Psi}{\delta \mu}(\mu,x))_{\mu \in K}$ is equi-continuous;
        \item\label{it:diff-Psi-glob:3} for any compact set $\Xi \subset \R^d$, the family of functions $(\mu \mapsto \frac{\delta \Psi}{\delta \mu}(\mu,x))_{x \in \Xi}$ is uniformly equi-continuous on $\ps_p(\R^d)$.
    \end{enumerate}
\end{assumptionb}

The fact that the assumptions of Proposition~\ref{prop:ass-ldp-diff} are satisfied under~\ref{ass-b:diff-Psi-glob} essentially follows from Lemma~\ref{lem:diff-Psi-Psit}, together with the observation that if $K_{[0,T]}$ is a compact subset of $\ps_p(C([0,T],\R^d))$, then $K := \{\mu_t: t \in [0,T], \mu_{[0,T]} \in K_{[0,T]}\}$ is a compact subset of $\ps_p(\R^d)$. The latter statement is an easy consequence of Lemma~\ref{lem:mutW1}~\eqref{it:mutW1:1} and~\eqref{it:mutW1:2}.

Hence, since we are working with $\F\equiv 0$, Theorem~\ref{thm:AbsLDP} applies as soon as~\ref{ass-b:global}-\ref{ass-b:diff-Psi-glob} hold (with $\S=\emptyset)$, together with the constraint qualification condition that for any $\mu_{[0,T]} \in \AdF \cap \Adineq$, there exists $\tilde{\mu}_{[0,T]} \in \AdF$ such that, for any $\varepsilon > 0$ small enough, $\Psi(\mu_t + \varepsilon(\tilde{\mu}_t-\mu_t))<0$ for all $t \in [0,T]$.

\medskip

We now turn our attention to the assumptions of Theorem~\ref{thm:abstractGibbs}. Since $\F\equiv 0$, \ref{ass:DiffF-barmu} and~\ref{ass:Fboundedens-barmu} necessarily hold true, at any $\overline\mu_{[0,T]} \in \AdF \cap \Ad$. On the other hand, Lemmata~\ref{lem:mutW1} and~\ref{lem:diff-Psi-Psit} again show that given $\overline\mu_{[0,T]}$, \ref{ass:regPsi-barmu} is satisfied under the following conditions on $\Psi$. 

\begin{assumptionb}[Regularity of $\Psi$ and constraint qualification at $\overline{\mu}_{[0,T]}$]\label{ass-b:regPsi-barmu}
$\phantom{a}$
    \begin{enumerate}[label=(\roman*),ref=\roman*]
        \item\label{ass-b:regPsi-barmu:1} The function $\Psi$ is continuous on $\ps_p(\R^d)$.
        \item\label{ass-b:regPsi-barmu:2} For any $t \in \T$, $\Psi$ is differentiable at $\overline\mu_t$ w.r.t. the set of directions $\ps_p(\R^d)$, uniformly in $t$, in the sense that for any $\mu_{[0,T]} \in \ps_p(C([0,T],\R^d))$,
        \begin{equation*}
            \lim_{\varepsilon \to 0} \sup_{t \in [0,T]} \left|\frac{\Psi((1-\varepsilon)\overline\mu_t + \varepsilon\mu_t)-\Psi(\overline\mu_t)}{\varepsilon} - \left\langle \mu_t-\overline\mu_t, \frac{\delta \Psi}{\delta\mu}(\overline\mu_t)\right\rangle\right|=0.
        \end{equation*}
        \item\label{ass-b:regPsi-barmu:3} The function $(t,x) \mapsto \frac{\delta \Psi}{\delta \mu}(\overline\mu_t,x)$ is continuous on $[0,T] \times \R^d$.
        \item\label{ass-b:regPsi-barmu:4} There exists $D^\Psi \in [0,\infty)$ such that
        \begin{equation*}
            \forall x \in \R^d, \qquad \sup_{t \in [0,T]} \left|\frac{\delta \Psi}{\delta\mu}(\overline\mu_t,x)\right| \leq D^\Psi [1+|x|^p].
        \end{equation*}
        \item\label{ass-b:regPsi-barmu:5} There exist $\tilde\mu_{[0,T]}$ in $\AdF \cap \Adeq$ and $\tilde\varepsilon > 0$ such that
        \begin{equation*}
            \forall \t \in [0,T], \quad \Psi(\overline\mu_t) + \tilde\varepsilon \bigg\langle \tilde\mu_t - \overline\mu_t, \frac{\delta\Psi}{\delta\mu}(\overline\mu_t) \bigg\rangle < 0.
        \end{equation*}
    \end{enumerate}
\end{assumptionb}

This condition allows to apply Theorem~\ref{thm:abstractGibbs} to characterise minimisers of~\eqref{eq:min-pb}. Let us finally mention that, as soon as $\Psi$ is convex, then~\ref{ass:diff-glob-conv} holds and the minimiser is therefore unique. Moreover, in this case, Lemma~\ref{lem:EquivQualif} implies that~\ref{ass-b:regPsi-barmu}-\eqref{ass-b:regPsi-barmu:5} holds (with $\tilde\varepsilon=1$) as soon as there exists $\tilde\mu_{[0,T]}$ in $\AdF \cap \Adeq$ such that $\Psi(\tilde\mu_t)<0$ for any $t \in [0,T]$.

\subsubsection{Example of a nonlinear inequality constraint}

A typical example of a nonlinear function $\Psi$ is given as follows. Let $W : \R^d \rightarrow \R$ be a measurable function, bounded from below, which satisfies $W(x) \leq C(1+|x|^p)$. For any $\mu \in \ps_p(\R^d)$, set
\begin{equation}\label{eq:Psi-W}
    \Psi(\mu) = \langle \mu , W \star \mu \rangle = \int_{\R^d \times \R^d} W(x-y)\mu(\d x)\mu(\d y).
\end{equation}
Then $\Psi$ is well-defined, and lower semi-continuous, on $\ps_p(\R^d)$. Moreover, it is differentiable on $\ps_p(\R^d)$ w.r.t. the set of directions $\ps_p(\R^d)$, and its linear functional derivative writes
\[ \frac{\delta\Psi}{\delta\mu}(\mu,x) = (W+W_-) \star \mu (x) - 2\Psi(\mu), \qquad W_-(x) := W(-x). \]
However, in general, it is not convex. For example, if $p \geq 2$, then one may take $W(x) = \frac{1}{2}|x|^2 - M$, for some $M \geq 0$, so the condition that $\Psi(\mu) \leq 0$ means that the variance of $\mu$ (understood, for $d \geq 2$, as the trace of the covariance matrix) is bounded from above by $M$. Then $\Psi$ is concave.

\subsection{Constrained Schrödinger bridge}\label{ss:schro}

Given a reference probability measure $\nu_{[0,T]} \in \ps_p(C([0,T],\R^d))$, for example the law of the $d$-dimensional Brownian motion on $[0,T]$, and two probability measures $\mu^{\mathrm{ini}}, \mu^{\mathrm{fin}} \in \ps_p(\R^d)$, such that $\mu^{\mathrm{ini}} \ll \nu_0$ and $\mu^{\mathrm{fin}} \ll \nu_T$, the standard Schrödinger bridge is the continuous process on $[0,T]$ whose law $\overline\mu_{[0,T]}$ is the minimiser of the problem
\[
    \inf_{\substack{\mu_{[0,T]} \in \ps_p(C([0,T],\R^d)), \\ \mu_0 = \mu^{\mathrm{ini}}, \, \mu_T = \mu^{\mathrm{fin}}}} H( \mu_{[0,T]} \vert \nu_{[0,T]} ).
\]
In this section, we address the \emph{dynamically constrained} version of this problem, namely,  for a lower semi-continuous function $\Psi : \ps_p(\R^d) \to \R$, we consider the minimisation problem
\begin{equation} \label{eq:SchrödProb}
\inf_{\substack{\mu_{[0,T]} \in \ps_p(C([0,T],\R^d)), \\ \mu_0 = \mu^{\mathrm{ini}}, \, \mu_T = \mu^{\mathrm{fin}}, \\
\forall t \in [0,T], \; \Psi(\mu_t) \leq 0}} H( \mu_{[0,T]} \vert \nu_{[0,T]} ).
\end{equation}
We call minimisers of this problem \emph{constrained Schrödinger bridges}. 
Our main result for this constrained problem reads as follows.

\begin{theorem}[Constrained Schrödinger bridge] \label{thm:ConsSchröd}
    Let $\nu_{[0,T]} \in \ps_p(C([0,T],\R^d))$ satisfy~\eqref{eq:exp-integ-proc}, and $\Psi : \ps_p(\R^d) \to \R$ be a lower semi-continuous function. Let $\mu^{\mathrm{ini}}, \mu^{\mathrm{fin}} \in \ps_p(\R^d)$ be such that $\nu_{0,T} \sim \mu^{\mathrm{ini}} \otimes \mu^{\mathrm{fin}}$, where $\nu_{0,T}$ denotes the marginal distribution of the pair $(\mathsf{X}_0,\mathsf{X}_T)$ under $\nu_{[0,T]}$.
    \begin{enumerate}
        \item If the infimum in~\eqref{eq:SchrödProb} is finite, then it is reached at some $\overline\mu_{[0,T]}$ which satisfies $\overline\mu_0 = \mu^{\mathrm{ini}}$, $\overline\mu_T = \mu^{\mathrm{fin}}$, and $\Psi(\mu_t) \leq 0$ for any $t \in [0,T]$. Moreover, if $\Psi$ is convex, then this minimiser is unique.
        \item If $\overline\mu_{[0,T]}$ is a minimiser at which~\ref{ass-b:regPsi-barmu} holds, then there exist $\overline{\lambda} \in \M_+([0,T])$ and measurable $\overline\zeta_0, \, \overline\zeta_T : \R^d \rightarrow \R$ such that
        \begin{equation} \label{eq:GibbsSchrod}
            \frac{\d \overline\mu_{[0,T]}}{\d \nu_{[0,T]}} ( x_{[0,T]} ) = \overline{Z}^{-1} \exp \biggl[ - \overline\zeta_0 (x_0) - \overline\zeta_T (x_T) - \int_{[0,T]} \frac{\delta\Psi}{\delta \mu}(\overline\mu_t,x_t) \overline{\lambda}( \d \t) \biggr],
        \end{equation}
        where $\overline{Z} \in (0, \infty)$ is a normalising constant, and the slackness condition holds
        \begin{equation*} 
            \Psi (\overline{\mu}_t) = 0 \quad \text{for } \overline{\lambda}\text{-a.e. } \t\in [0,T].
        \end{equation*}
    \end{enumerate}
\end{theorem}

The proof of Theorem~\ref{thm:ConsSchröd} follows from the application of Theorem~\ref{thm:abstractGibbs} to the case where the equality constraints are defined as follows: let $( \phi^k )_{ k \geq 1}$ be a countable family of bounded continuous functions $\R^d \rightarrow \R$ which separates measures. 
For $k \geq 1$, we define $\phi^k_0, \phi^k_T : C([0,T],\R^d) \rightarrow \R$ by 
\[ \phi^k_0 ( x_{[0,T]} ) := \phi^k ( x_0 ) - \int_{\R^d} \phi^k \d \mu^{\mathrm{ini}}, \qquad \phi^k_T ( x_{[0,T]} ) := \phi^k ( x_T ) - \int_{\R^d} \phi^k \d \mu^{\mathrm{fin}}. \] 
We finally set $(\zeta_s )_{s \in \S} := (\phi^k_0)_{k\geq 1} \cup (\phi^k_T)_{k\geq 1}$. The details of the proof are postponed to Section~\ref{ss:pf-schr}.

Compared to Theorem \ref{thm:abstractGibbs}, a major improvement of Theorem~\ref{thm:ConsSchröd} is that $\overline{\mu}_{[0,T]} \sim \nu_{[0,T]}$, see Remark \ref{rem:equivalence}. 

\begin{rem}[Multi-marginal constraints and support constraints]
We could easily adapt the above result
to impose marginal laws at arbitrarily many instants.
We could also consider time-dependent constraints $\Psi ( t, \mu_t )$.
\end{rem}

\subsection{Gibbs principle for diffusion processes} \label{sec:Gibbsdif}

In this section, we focus on the Gibbs principle for i.i.d. diffusion processes under constraints of the form~\eqref{eq:constr-time-marg} on the empirical distribution.

\begin{assumptionc}[On the reference SDE]\label{ass-c:sde} 
The functions $b : [0,T] \times \R^d \rightarrow \R^d$, $\sigma : [0,T] \times \R^d \rightarrow \R^{d \times d}$ are measurable and locally bounded, $\nu_0 \in \ps ( \R^d )$, and in addition: 
    \begin{enumerate}[label=(\roman*),ref=\roman*]
        \item\label{it:ass-c-sde:1} uniformly in $t \in [0,T]$, $x \mapsto b_t(x)$ and $x \mapsto \sigma_t(x)$ are Lipschitz continuous;
        \item\label{it:ass-c-sde:2} there exists $M_\sigma \geq 0$ such that $|\sigma_t(x)| \leq M_\sigma$ for all $(t,x) \in [0,T] \times \R^d$, and $t \mapsto \sigma_t (x)$ is locally Hölder-continuous;
        \item\label{it:ass-c-sde:3} for all $\alpha > 0$, $\int_{\R^d} e^{\alpha |x|^p}\nu_0(\d x) < \infty$.
    \end{enumerate}
\end{assumptionc}
Under~\ref{ass-c:sde}, for a given $d$-dimensional Brownian motion $(B_t)_{t \in [0,T]}$, it is standard that the SDE
\begin{equation}\label{eq:sde}
    \d X_t = b_t ( X_t ) \d t + \sigma_t (X_t) \d B_t, \qquad X_0 \sim \nu_0,
\end{equation}
has a pathwise unique strong solution $X_{[0,T]}$, whose law in $C([0,T],\R^d)$ is denoted by $\nu_{[0,T]}$. Moreover, if $p<2$ then this measure satisfies~\eqref{eq:exp-integ-proc}. 

As a consequence, the law $\Pi^N$ of the empirical distribution of $N$ iid copies of the SDE~\eqref{eq:sde} satisfies a LDP on $\ps_p(C([0,T],\R^d))$, with good rate function given by $H(\mu_{[0,T]}|\nu_{[0,T]})$~\cite{wang2010sanov}. Thus, if $\Psi$ is such that the functions $(\Psi_t)_{t \in [0,T]}$ satisfy the assumptions of Theorem~\ref{thm:AbsLDP}, the conditional law $\Pi^N(\cdot|\Adineq)$ satisfies a LDP with good rate function given by 
\[
    H(\mu_{[0,T]}|\nu_{[0,T]}) - \inf_{\mu'_{[0,T]}\in \AdF \cap \Adineq} H(\mu'_{[0,T]}|\nu_{[0,T]})
\] 
on $\AdF \cap \Adineq$, and as soon as this good rate function has a unique minimiser $\overline\mu_{[0,T]}$, $\Pi^N(\cdot|\Adineq)$ converges weakly to $\delta_{\overline\mu_{[0,T]}}$. For such a minimiser, if $\Psi$ satisfies~\ref{ass-b:regPsi-barmu}, Theorem~\ref{thm:abstractGibbs} yields the existence of $\overline \lambda \in \M_+([0,T])$ such that
\[
    \frac{\d \overline\mu_{[0,T]}}{\d \nu_{[0,T]}} \bigl( x_{[0,T]} \bigr) = \frac{1}{\overline{Z}} 
        \exp \biggl[ - \int_{[0,T]} \frac{\delta\Psi}{\delta \mu}(\overline\mu_t,x_t) \overline{\lambda}( \d t) \biggr], \qquad \overline Z \in (0,+\infty).
\]
Starting from this expression, the goal of the present section is to express $\overline\mu_{[0,T]}$ as the law of the solution $\overline X_{[0,T]}$ to some SDE of the form
\begin{equation}\label{eq:sde-bar}
    \d \overline X_t = \overline b_t (\overline X_t ) \d t + \sigma_t (\overline X_t) \d B_t, \qquad \overline X_0 \sim \overline \mu_0,
\end{equation}
and to express the drift coefficient $\overline b_t$ and the probability measure $\overline \mu_0$ in terms of the original SDE~\eqref{eq:sde}. We shall actually work in the more general setting in which there exist measurable functions $c : [0,T] \times \R^d \to \R$ and $\psi : [0,T] \times \R^d \to \R$, and a positive Radon measure $\lambda \in \M_+([0,T])$, such that the measure $\overline\mu_{[0,T]}$ satisfies
\begin{equation}\label{eq:overline-mu-proc}
    \frac{\d \overline\mu_{[0,T]}}{\d \nu_{[0,T]}} \bigl( x_{[0,T]} \bigr) = \frac{1}{Z} 
        \exp \biggl[ - \int_0^T c_t(x_t)\d t - \int_{[0,T]} \psi_t(x_t) \lambda( \d t) \biggr], \qquad Z \in (0,+\infty),
\end{equation}
which would thus allow us to take a nonzero interaction functional $\F$ into account. We postpone the detail of our assumptions on $c$ and $\psi$ to~\ref{ass-c:hjb} below, and first provide an informal sketch of our argument. 

Let us assume that $\lambda$ 
has a density $(\lambda_t)_{0 \leq t \leq T}$ w.r.t. to the Lebesgue measure.
Under very weak regularity assumptions,
\cite[Theorem 5.24]{leonard2022feynman} shows that $\overline{\mu}_{[0,T]}$ is the law of the solution to~\eqref{eq:sde-bar}, with
\[
    \overline b_t(x) = b_t(x) + a_t(x) \alpha_t(x), 
\]
for $a_t := \sigma_t \sigma^\top_t$ and a measurable $\alpha : [0,T] \times \R^d \rightarrow \R^d$.
Moreover, 
\[ \alpha_t = - \nabla \varphi_t, \qquad \frac{\d \overline{\mu}_0}{\d \nu_0} (x) = Z^{-1} e^{-\varphi_0(x)}, \]
where $\varphi$ is the solution to the Hamilton-Jacobi equation (HJB) 
\begin{equation} \label{eq:LeoHJB}
\partial_t \varphi_t + b_t \cdot \nabla \varphi_t - \frac{1}{2} \vert \sigma_t^\top \nabla \varphi_t \vert^2 + \frac{1}{2} \mathrm{Tr}[a_t \nabla^2 \varphi_t ] = -c_t - \psi_t \lambda_t,
\qquad \varphi_T = 0, 
\end{equation}
in a generalised sense.
Formally, the Feynman-Kac formula tells that the solution of \eqref{eq:LeoHJB} can be represented as \begin{equation*} 
\varphi_t (x) := - \log \E \exp \bigg[ \int_t^T  c_s (Z^{t,x}_s) \d s + \int_{[t,T]} \psi_s (Z^{t,x}_s) \lambda (\d s) \bigg], 
\end{equation*} 
where $( Z^{t,x}_s )_{t \leq s \leq T}$ is the solution of $\d Z^{t,x}_s = b_s ( Z^{t,x}_s ) \d s + \sigma_s ( Z^{t,x}_s ) \d B_s $, $t \leq s \leq T$,
with $Z^{t,x}_t=x$.
We now prove an analogous of this result in a more classical sense, and without assuming existence of a density for $\lambda$.

\begin{assumptionc}[Regularity of the coefficients of~\eqref{eq:LeoHJB}]\label{ass-c:hjb}
    Assumption~\ref{ass-c:sde} holds with $p=1$, $\psi$, $c : [0,T] \times \R^d \rightarrow \R$ are measurable and locally bounded, and in addition:
    \begin{enumerate}[label=(\roman*),ref=\roman*]
        \item\label{it:ass-c-hjb:1} the function $(t,x) \mapsto \psi_t (x)$ is continuous;
        \item\label{it:ass-c-hjb:2} uniformly in $t \in [0,T]$, $c_t$ and $\psi_t$ are Lipschitz-continuous, and $\sigma_t^{-1}$ is bounded;
        \item\label{it:ass-c-hjb:3} the families $( x \mapsto \nabla b_t (x) )_{0 \leq t \leq T}$, $( x \mapsto \nabla \sigma_t (x) )_{0 \leq t \leq T}$, $( x \mapsto \nabla c_t (x) )_{0 \leq t \leq T}$, and $( x \mapsto \nabla \psi_t (x) )_{0 \leq t \leq T}$ are equi-continuous. 
    \end{enumerate}
\end{assumptionc}

We now fix a positive Radon measure $\lambda \in \M_+ ([0,T])$.
To make sense of the analogous of \eqref{eq:LeoHJB} when $\lambda$ may not have a density, we rely on the integrated version 
\begin{equation} \label{eq:HJBlim} 
- \varphi_t + \int_t^T \left(b_s \cdot \nabla \varphi_s - \frac{1}{2} \vert \sigma_s^\top \nabla \varphi_s \vert^2 + \frac{1}{2} \mathrm{Tr}[a_s \nabla^2 \varphi_s] + c_s\right) \d s
+ \int_{[t,T]} \psi_s \lambda ( \d s) = 0.
\end{equation}
If $\lambda(\{t \}) \neq 0$, an arbitrary choice has been made when considering integrals over $[t,T]$ rather than $(t,T]$.
However, the set of atoms of $\lambda$ is at most countable; hence the choice of the interval does not matter if we only require equality Lebesgue-a.e. 
For approximation and stability purposes, we introduce a specific notion of solution, which relies on the following construction.

Under the assumptions made on $\sigma$ in~\ref{ass-c:sde}-\ref{ass-c:hjb}, a consequence of \cite[Theorem 2.1]{rubio2011existence} is that for any $0 < s \leq T$ and any continuous $\varphi : \R^d \rightarrow \R$ with linear growth, the parabolic equation
\[
\begin{cases}
\partial_t \varphi_t + \tfrac{1}{2} \mathrm{Tr} \big[ a_t \nabla^2 \varphi_t \big] = 0, \quad 0 \leq t \leq s, \\
\varphi_s = \varphi,
\end{cases}
\]
has a unique solution $\varphi \in C([0,s] \times \R^d) \cap C^{1,2}((0,s) \times \R^d)$ with linear growth. 
From this, we define the evolution system $(S_{t,s})_{0\leq t\leq s\leq T}$ by
\[ S_{t,s} [ \varphi ] (x) = \varphi_t (x), \]
for any continuous $\varphi : \R^d \rightarrow \R$ with linear growth, $S_{t,s} [ \varphi ]$ being a $C^2$ function with linear growth as soon as $t < s$.

\begin{definition} \label{def:MildForm}
We say that a measurable $\varphi : [0,T] \times \R^d \rightarrow \R$ is a \emph{mild} solution of \eqref{eq:HJBlim}
if for Lebesgue-a.e. $t \in [0,T]$, $x \mapsto \varphi_t (x)$ is $C^1$, $(t,x) \mapsto \nabla \varphi_t (x)$ is bounded measurable, and for a.e. $t \in [0,T]$,
\begin{equation} \label{eq:MildForm}
\varphi_t = \int_t^T S_{t,s} \big[ b_s \cdot \nabla \varphi_s - \tfrac{1}{2} \lvert \sigma_s^\top \nabla \varphi_s \rvert^2 + c_s \big] \d s
+ \int_{[t,T]} S_{t,s} [ \psi_s ] \lambda (\d s).
\end{equation} 
This implies that $x \mapsto \varphi_t (x)$ is $C^2$ for Lebesgue-a.e. $t$.
\end{definition}

The above definition is a natural extension of the Duhamel formula for perturbations of a linear PDE, which was already used in \cite{daudin2023optimal}.
Lebesgue-almost sure uniqueness always holds for \eqref{eq:HJBlim} in the sense of Definition \ref{def:MildForm} because the difference of two solutions solves a classical linear parabolic equation without source term.

\begin{theorem} \label{thm:LinkPath}
Under~\ref{ass-c:hjb}, let $\overline\mu_{[0,T]}$ be given by~\eqref{eq:overline-mu-proc}.
Then, $\overline{\mu}_{[0,T]}$ is the path-law of a solution to the SDE
\[ \d \overline{X}_t = b_t ( \overline{X}_t ) \d t - a_t(\overline{X}_t) \nabla \varphi_t (\overline{X}_t) \d t + \sigma_t ( \overline{X}_t ) \d B_t, \quad \overline{X}_0 \sim {Z}^{-1} e^{-\varphi_0(x)} \nu_0 ( \d x), \]
where $\varphi_t$ is given by
\begin{equation} \label{eq:RepFK} \varphi_t (x) := - \log \E \exp \bigg[ \int_t^T  c_s (Z^{t,x}_s) \d s + \int_{[t,T]} \psi_s (Z^{t,x}_s) \lambda (\d s) \bigg], 
\end{equation} 
where $( Z^{t,x}_s )_{t \leq s \leq T}$ is the path-wise unique solution to
\begin{equation} \label{eq:Ztx}
\d Z^{t,x}_s = b_s ( Z^{t,x}_s ) \d s + \sigma_s ( Z^{t,x}_s ) \d B_s, \quad t \leq s \leq T,  
\end{equation} 
with $Z^{t,x}_t=x$.
Moreover, $\varphi$ is the mild solution of \eqref{eq:HJBlim} in the sense of Definition \ref{def:MildForm}. 
\end{theorem}

Theorem \ref{thm:LinkPath} is proved in Section~\ref{ss:pf-LinkPath}. 

\begin{rem}
If one wants to apply Theorem~\ref{thm:LinkPath} with $\psi_t(x) = \tfrac{\delta \Psi}{\delta \mu}(\overline\mu_t,x)$ in order to describe the measure $\overline\mu_{[0,T]}$ obtained by Theorem~\ref{thm:abstractGibbs}, with a function $\Psi$ that satisfies~\ref{ass-b:regPsi-barmu}, then the global Lipschitz assumption made on $\psi_t$ in~\ref{ass-c:hjb} implies that $\tfrac{\delta \Psi}{\delta \mu}(\overline\mu_t,x)$ must grow at most linearly in $x$. In the case of a linear constraint, where $\Psi(\mu) = \langle \mu, \psi\rangle$ for some function $\psi : \R^d \to \R$, this implies that $\psi$ must grow at most linearly in $x$. In this case, and if $\psi$ is continuous, then~\ref{ass-b:regPsi-barmu} is satisfied already for $p=1$. Therefore, there is no need to require~\ref{ass-c:sde} to hold with some $p>1$ in the statement of~\ref{ass-c:hjb}.
\end{rem}

\subsection{Gaussian examples with inequality constraint on the expectation}\label{ss:brown-expl}

In this section, we provide examples of one-dimensional diffusion processes $X_{[0,T]}$, for which the results of Theorems~\ref{thm:AbsLDP} and~\ref{thm:abstractGibbs} hold, with the linear inequality constraint $\Psi(\mu_{[0,T]}) = \langle \mu_t, \psi\rangle$, where we fix $\psi(x)=x$. In other words, we consider cases where the conditional distribution of $X^{1,N}_{[0,T]}$ given the event
\[
    \forall t \in [0,T], \qquad \frac{1}{N}\sum_{i=1}^N X^{i,N}_t \leq 0
\]
converges as $N \rightarrow +\infty$ towards the measure $\overline\mu_{[0,T]}$ with density w.r.t. to $\nu_{[0,T]} := \L (X_{[0,T]})$ given by
\[
    \frac{\dd \overline\mu_{[0,T]}}{\dd \nu_{[0,T]}}(x_{[0,T]}) = \frac{1}{\overline{Z}} \exp\left[-\int_{[0,T]} x_t\overline\lambda(\d t)\right],
\]
for some Lagrange multiplier $\overline\lambda \in \M_+([0,T])$. 
Then by Theorem~\ref{thm:LinkPath}, $\overline\mu_{[0,T]}$ is the law of a \emph{corrected} diffusion process $\overline X_{[0,T]}$ with modified drift and tilted initial condition. 
On our examples, we are able to give \emph{explicit} formulas for these corrections, as well as for $\overline\lambda$.

\subsubsection{Time-inhomogeneous drifted Brownian motion} \label{sssec:diftedInhomo}

Let $m : [0,T] \rightarrow \R$ be a $C^2$ function such that $m(0) =0$, $\dot{m} \geq 0$ and $\ddot{m} \leq 0$, where $\dot{m} := \tfrac{\d}{\d t} m$. 
Let $\nu_{[0,T]}$ be the path-law of the Gaussian process given by
\[ \d X_t = \dot{m}(t) \d t + \d B_t, \quad X_0 \sim \mathcal{N}(x_0,\sigma^2). \] 
To avoid situations in which $\nu_{[0,T]}$ already satisfies the constraints, we assume that $x_0 + m(T) > 0$ and $\sigma >0$.
The corrected process given by Theorems \ref{thm:abstractGibbs}-\ref{thm:LinkPath} shall be the solution to
\begin{equation} \label{eq:correctDrift}
\d \overline{X}_t = \dot{m}(t) \d t - \nabla \varphi_t (\overline{X}_t) \d t + \d B_t, \quad \overline{X}_0 \sim Z^{-1} e^{-\varphi_0 (x)} \nu_0 ( \d x), 
\end{equation}
where $\varphi$ solves the HJB equation \eqref{eq:HJBlim}.
In the current simple case, we can look for solutions of \eqref{eq:HJBlim} with $\nabla \varphi_t$ independent of $x$.
Several options appear, depending on the function
\[ \overline{m} : t \mapsto x_0 + m(t) - [ \sigma^2 +t ] \dot{m}(t). \]
Our assumptions on $m$ make $\overline{m}$ non-decreasing.

\begin{enumerate}
    \item \underline{If $\overline{m}(T) < 0$:} we can verify that
    \[ \overline{\lambda}(\d t) = \frac{x_0 + m ( T )}{\sigma^2 + T} \delta_T ( \d t), \quad \nabla \varphi_t (x) = \frac{x_0 + m ( T )}{\sigma^2 + T}, \]
    satisfy the HJB equation \eqref{eq:HJBlim}, and that the process $\overline{X}_{[0,T]}$ given by \eqref{eq:correctDrift} satisfies $\overline{X}_0 \sim \mathcal{N}(\tfrac{T}{T+\sigma^2}(x_0-\sigma^2 \dot{m}(T)),\sigma^2)$, $\E[ \overline{X}_t ] < 0$  for $t \in [0,T)$, and $\E [ \overline{X}_T ] = 0$. 
    From Theorem \ref{thm:LinkPath} and the sufficient condition given by Theorem \ref{thm:convex-case}, this shows that $\L ( \overline{X}_{[0,T]} )$ is indeed the corrected law $\overline{\mu}_{[0,T]}$.
    Both the velocity and the initial condition are tilted by the effect of the multiplier at the final time.
    \item \underline{If $\overline{m}(0) \leq 0 \leq \overline{m}(T)$:} let $\overline\tau$ be the first time at which $\overline{m}$ vanishes. 
    We similarly obtain a solution of \eqref{eq:HJBlim} by considering
    \[ \overline{\lambda}(\d t) = - \mathbbm{1}_{[\overline\tau,T]} \ddot{m}(t) \d t + \dot{m}(T) \delta_T ( \d t), \quad \nabla \varphi_t (x) = \begin{cases}
    \dot{m}(\overline\tau), \quad \text{if  } t \leq \overline\tau,   \\
    \dot{m}(t), \quad \text{if  } t > \overline\tau.
    \end{cases} \]
    The related process $\overline{X}_{[0,T]}$ satisfies $\overline{X}_0 \sim \mathcal{N}(x_0 - \sigma^2 \dot{m}(\tau),\sigma^2)$, $\E[\overline{X}_t] < 0$ for $t \in [0,\overline\tau)$, and $\E[\overline{X}_t] = 0$ for $t \in [\overline\tau,T]$.
    As previously, this shows that $\L ( \overline{X}_{[0,T]} ) = \overline{\mu}_{[0,T]}$.
    Both the velocity and the initial condition are tilted.
    The continuous part of the multiplier activates at the first time where $\E [ \overline{X}_t ]$ vanishes.
    Its effect is to decrease the correction $\nabla \varphi_t (x)$ of the velocity, which is generated by the atom at the terminal time.
    The process is corrected at the second order in time (the acceleration is modified), as it is customary in control theory. 
    \item \underline{If $\overline{m}(0)> 0$:} similarly, we verify that the corrected process is given by \eqref{eq:correctDrift} with
    \[ \overline{\lambda}(\d t) =  \frac{x_0 - \sigma^2 \dot{m}(0)}{\sigma^2} \delta_0 ( \d t) - \ddot{m}(t) \d t + \dot{m}(T) \delta_T ( \d t), \quad \nabla \varphi_t (x) = 
    \begin{cases}
    \tfrac{x_0}{\sigma^2}, \quad \text{at  } t =0, \\
    \dot{m}(t), \quad \text{if  } t > 0,
    \end{cases} \]
    together with $\overline{X}_0 \sim \mathcal{N}(0,\sigma^2)$.
    The corrected process now satisfies $\E [ \overline{X}_t ] = 0$ at each time.
    The atom at time $0$ generates a discontinuity that pushes the initial law on the constraint before starting the dynamics.
\end{enumerate}
The transition between the different cases is continuous. 
Interestingly, $\overline\lambda$ having a density on $(0,T)$ is directly related to the (second order) regularity of the input.
If we decrease the $C^2$ regularity of $m$, other atoms may appear within $\overline\lambda$.

\subsubsection{Ornstein-Uhlenbeck process} \label{sssec:ornsetin}

We now assume that $\nu_{[0,T]}$ is the path-law of the Gaussian process given by
\[ \d X_t = ( 1 - X_t ) \d t + \d B_t, \quad X_0 \sim \mathcal{N}(x_0,\sigma^2). \] 
To avoid situations in which $\nu_{[0,T]}$ already satisfies the constraints, we assume that $1 + (x_0 -1) e^{-T} > 0$ and $\sigma >0$. 
The corrected process shall still be the solution of
\eqref{eq:correctDrift} with $\varphi$ solving the HJB equation \eqref{eq:HJBlim}.
In this simple case, we can still look for affine solutions of \eqref{eq:HJBlim} with $\nabla \varphi_t$ independent of $x$.
As in Section \ref{sssec:diftedInhomo}, several options appear, now depending on the functions
\[ m_{\tau,\lambda} : t \in \R_{\geq 0} \mapsto 1 + (x_0 - 1 - \sigma^2 \lambda e^{-\tau} ) e^{-t} - \lambda e^{-\tau} \sinh ( t ), \]
for $\tau \geq 0$ and $\lambda \in [0,1]$.
Since $e^{2t} m_{t,1} (t)$ is a second-order polynomial in $e^t$, we easily get that $x_0 \leq \sigma^2$ implies the existence of a unique $\overline\tau > 0$ such that $m_{\overline{\tau},1} (\overline\tau) = 0$.
Using that $\cosh ( \overline\tau ) \geq 0$, this implies that $x_0 - 1 - \sigma^2 e^{-\overline\tau} \leq 0$. 
We can then check that $t \mapsto m_{\overline{\tau},1} (t)$ is increasing on $[0,\overline{\tau}]$.

\begin{enumerate}
    \item \underline{If $x_0 \leq \sigma^2$ and $T < \overline{\tau}$:} from $m_{\overline{\tau},1} (\overline\tau) = 0$, we get $m_{\overline{\tau},1} (T) < 0$. 
    Since $T < \overline{\tau}$, this implies that $m_{T,1} ( T ) < 0$. 
    Since we assumed that $m_{T,0} (T) > 0$, there exists a unique $\lambda_T \in (0,1)$ such that $m_{T,\lambda_T} ( T ) = 0$.
    Since $e^{t} m_{T,\lambda_T} (t)$ is a second-order polynomial in $e^t$, we easily deduce that $m_{T,\lambda_T} (t) < 0$ for every $t \in [0,T)$.  
    At this stage, we can verify that 
    \[ \overline{\lambda}(\d t) =  \lambda_T \delta_T ( \d t), \quad \nabla \varphi_t (x) = \lambda_T e^{t-T}, \]
    satisfy the HJB equation \eqref{eq:HJBlim}. 
    The process $\overline{X}_{[0,T]}$ given by \eqref{eq:correctDrift} then satisfies $\overline{X}_0 \sim \mathcal{N}(x_0 - \lambda_T \sigma^2 e^{-T},\sigma^2)$, and $\E [ \overline{X}_t ] = m_{T,\lambda_T} (t)$ for every $t \in [0,T]$.
    From Theorem \ref{thm:LinkPath} and the sufficient condition given by Theorem \ref{thm:convex-case}, this shows that $\L ( \overline{X}_{[0,T]} )$ is indeed the corrected law $\overline{\mu}_{[0,T]}$.
    This situation is similar to the first case in Section \ref{sssec:diftedInhomo}.
    \item \underline{$x_0 \leq \sigma^2$ and $\overline{\tau} \leq T$:} we obtain a solution of \eqref{eq:HJBlim} by considering
    \[ \overline{\lambda}(\d t) = \mathbbm{1}_{t \in [\tau,T]} \d t + \delta_T ( \d t), \quad \nabla \varphi_t (x) = \begin{cases}
    e^{t-\tau}, \quad \text{if  } t \leq \tau,   \\
    1, \quad \text{otherwise.}
    \end{cases} \]
    The related process $\overline{X}_{[0,T]}$ satisfies $\overline{X}_0 \sim \mathcal{N}(x_0 - \sigma^2 e^{-\tau},\sigma^2)$, $\E [ \overline{X}_t ] = m_{\overline{\tau},1} (t)$ for $t \in [0,\overline{\tau}]$, and $\E [ \overline{X}_t ] = 0$ for $t \in [\overline{\tau},T]$.
    As previously, this shows that $\L ( \overline{X}_{[0,T]} ) = \overline{\mu}_{[0,T]}$.
    The continuous part of the multiplier activates at the first time where $\E [ \overline{X}_t ]$ vanishes.
    Its effect is to decrease the correction $\nabla \varphi_t (x)$ of the velocity, which is generated by the atom at the terminal time.
    \item \underline{If $x_0 > \sigma^2$:} similarly, we verify that the corrected process is given by \eqref{eq:correctDrift} with
    \[ \overline{\lambda}(\d t) = \frac{x_0 - \sigma^2}{\sigma^2} \delta_0 ( \d t) - \d t + \delta_T ( \d t), \quad \nabla \varphi_t \equiv 1 \text{  for } t >0, \quad  \overline{X}_0 \sim \mathcal{N}(x_0,\sigma^2). \]
    As previously, the atom at time $0$ generates a discontinuity that pushes the initial law on the constraint before starting the dynamics.
    Then, the correction compensates the $+1$ component of the velocity, so that $\E [ \overline{X}_t ] = 0$ for every $t \in [0,T]$. 
\end{enumerate}
The transition between the different cases is continuous. 
Past the first vanishing of $\E [ \overline{X}_t ]$, the corrected process corresponds to a shifted Ornstein-Uhlenbeck process, whose new stationary distribution is $\mathcal{N}(0,\sigma^2)$.
Interestingly, Sections \ref{sssec:diftedInhomo}-\ref{sssec:ornsetin} show that quite different behaviours can happen with specific transitions.

\section{Proofs of the results of Section~\ref{s:abstract-results}} \label{s:proofAbs}
\subsection{Proofs of the results of Section~\ref{ss:LDP}}\label{ss:proofs-LDP}

The proof of Theorem~\ref{thm:AbsLDP} relies on the following preliminary lemma.

\begin{lemma}[Continuity sets] \label{lem:ContSet}
Under the assumptions of Theorem~\ref{thm:AbsLDP}, for any open set $U$ in $\ps_\phi(E)$,
\[ \inf_{\mu \in U \cap \mathring{\Adineq}} \mathcal{I}(\mu) = \inf_{\mu \in U \cap \Adineq} \mathcal{I}(\mu). \]    
\end{lemma}

\begin{proof}
Since $U \cap \mathring{\Adineq} \subset U \cap \Adineq$, it is sufficient to show that the l.h.s. is lower than the r.h.s.
Let us fix $\mu$ in $U \cap \Adineq$. 
There is no loss of generality in assuming that $\mu \in \AdF$.
Let $\tilde\mu$ be given by the constraint qualification~\ref{ass:constr-qual-ldp}, and let $\mu_\varepsilon := (1-\varepsilon) \mu + \varepsilon \tilde\mu$.
From~\ref{ass:constr-qual-ldp}-\eqref{ass:constr-qual-ldp:1}, 
\[ \forall \t \in \T, \quad \Psi_t ( \mu_\varepsilon ) < 0, \]
for every small enough $\varepsilon >0$. 
By~\ref{ass:psi-ldp}, the mapping $\mu \in \ps_\phi(E) \mapsto \sup_{\t \in \T} \Psi_t(\mu)$ is continuous and $\mu_\varepsilon$ belongs to the pre-image of the open set $(-\infty,0)$, which is contained in $\Adineq$. 
Therefore, for every $\varepsilon$ small enough, $\mu_\varepsilon \in \mathring{\Adineq}$.
Since $\mu_\varepsilon$ converges towards $\mu$ as $\varepsilon \rightarrow 0$ and $U$ is open, we eventually get that $\mu_\varepsilon$ belongs to $U \cap \mathring{\Adineq}$ for every small enough $\varepsilon$.
Moreover, by convexity of $H$,
\[ \mathcal{I}(\mu_\varepsilon) = H( \mu_\varepsilon \vert \nu ) + \F ( \mu_\varepsilon ) \leq (1-\varepsilon) H(\mu\vert \nu) + \varepsilon H( \tilde{\mu} \vert \nu ) + \F ( \mu_\varepsilon ), \]
and~\ref{ass:constr-qual-ldp}-\eqref{ass:constr-qual-ldp:2} implies that 
\[ \limsup_{\varepsilon \rightarrow 0} H( \mu_\varepsilon \vert \nu ) + \F ( \mu_\varepsilon ) \leq H( \mu \vert \nu ) + \F ( \mu ) = \mathcal{I}(\mu), \] 
completing the proof.
\end{proof}

We may now present the proof of Theorem~\ref{thm:AbsLDP}.

\begin{proof}[Proof of Theorem \ref{thm:AbsLDP}]
For any measurable subset $A$ of $\ps_\phi(E)$, 
\[ \log \Pi^N(A \vert \Adineq) = \log \Pi^N( A \cap \Adineq ) - \log \Pi^N( \Adineq). \] 
If $A$ is closed, we write
\[ \log \Pi^N(A \vert \Adineq) \leq \log \Pi^N( A \cap \Adineq ) - \log \Pi^N( \mathring{\Adineq} ). \] 
Since, by Lemma~\ref{lem:AdClosed}, $A \cap \Adineq$ is closed and $\mathring{\Adineq}$ is open, we get that
\[ \limsup_{N \rightarrow +\infty} N^{-1} \log \Pi^N(A \vert \Adineq) \leq - \inf_{\mu \in A \cap \Adineq} \mathcal{I} ( \mu ) + \inf_{\mu \in \mathring{\Adineq}} \mathcal{I} ( \mu ), \] 
using the LDP satisfied by $( \Pi^N )_{N \geq 1}$.
Similarly, if $A$ is open,
\[ \log \Pi^N (A \vert \Adineq) \geq \log \Pi^N( A \cap \mathring{\Adineq} ) - \log \Pi^N( \Adineq ), \]
and the LDP satisfied by $( \Pi^N )_{N \geq 1}$ yields  
\[ \liminf_{N \rightarrow +\infty} N^{-1} \log \Pi^N(A \vert \mathring{\Adineq}) \geq -\inf_{\mu \in A \cap \mathring{\Adineq}} \mathcal{I} ( \mu ) + \overline{\mathcal{I}}_\Psi. \] 
We get the desired large deviation upper bound by using Lemma~\ref{lem:ContSet} with $U = \ps_\phi(E)$, and the lower bound using Lemma~\ref{lem:ContSet} with $U = A$.
\end{proof}

We now turn to the proof of Proposition~\ref{prop:ass-ldp-diff}. We start with an elementary measure theoretical result.

\begin{lemma} \label{lem:CVUnif}
Let $( \mu_k )_{k \geq 1}$ be a sequence that weakly converges towards $\mu$ in $\ps(E)$.
Let $( f^k_t)_{t \in \T, \, k \geq 1}$ be an equi-continuous family in $C(E,\R)$, such that
\[ \forall x \in E, \quad \sup_{t \in \T} \vert f^k_t (x) - f_t (x) \vert \xrightarrow[k \rightarrow +\infty]{} 0, \]
for some (necessarily equi-continuous) family $(f_t)_{t \in \T}$ in $C(E,\R)$.
If the following uniform integrability criterion holds
\[ \sup_{t \in \T, \, k \geq 1} \; \int_E \vert f_t \rvert \1_{\vert f_t \rvert > M} \d \mu + \int_E \vert f^k_t \rvert \1_{\vert f^k_t \rvert > M} \d \mu_k \xrightarrow[M \rightarrow +\infty]{} 0,   \]
then 
\[ \sup_{t \in \T} \bigg\lvert \int_E f^k_t \d \mu_k - \int_E f_t \d \mu \bigg\rvert \xrightarrow[k \rightarrow +\infty]{} 0. \]
\end{lemma}

\begin{proof}
Since $E$ is separable, the Skorokhod representation theorem provides a probability space $(\Omega,\F,\P)$ together with a sequence $(X_k)_{k \geq 1}$ of random variables on it with $X_k \sim \mu_k$ that $\P$-a.s. converges towards some $X \sim \mu$.
As a consequence:
\begin{equation} \label{eq:ContProjDom}
\sup_{t \in \T} \bigg\lvert \int_E f^k_t \d \mu_k - \int_E f_t \d \mu \bigg\rvert \leq \E \big[ \sup_{t \in \T} \lvert f^k_t (X^k) - f_t ( X) \rvert \big].
\end{equation}
Writing
\[
    \sup_{t \in \T} \lvert f^k_t (X^k) - f_t ( X) \rvert \leq \sup_{t \in \T} \lvert f^k_t (X^k) - f^k_t ( X) \rvert + \sup_{t \in \T} \lvert f^k_t (X) - f_t ( X) \rvert,
\]
we have by equi-continuity that, a.s.,
\[
    \sup_{t \in \T} \lvert f^k_t (X^k) - f^k_t ( X) \rvert \leq \sup_{t \in \T, k' \geq 1} \lvert f^{k'}_t (X^k) - f^{k'}_t ( X) \rvert \xrightarrow[k \rightarrow +\infty]{} 0,
\]
while by assumption, a.s.,
\[
    \sup_{t \in \T} \lvert f^k_t (X) - f_t ( X) \rvert \xrightarrow[k \rightarrow +\infty]{} 0.
\]
As a consequence, defining
\begin{equation*} 
M_k := \sup_{t \in \T} \lvert f^k_t ( X^k) \rvert \vee \lvert f_t ( X ) \rvert,
\end{equation*}
we get, by dominated convergence, for any $M >0$,
\[ \E \big[ \sup_{t \in \T} \lvert f^k_t (X^k) - f_t ( X) \rvert \mathbbm{1}_{M_k \leq M}\big] \xrightarrow[k \rightarrow +\infty]{} 0. \]
On the other hand, the uniform integrability assumption implies that
\[ \sup_{k \geq 1} \E \big[ \sup_{t \in \T} \lvert f^k_t (X^k) - f_t ( X) \rvert \mathbbm{1}_{M_k > M}\big] \xrightarrow[M \rightarrow +\infty]{} 0, \]
concluding the proof.
\end{proof}

We may now prove Proposition~\ref{prop:ass-ldp-diff}.

\begin{proof}[Proof of Proposition~\ref{prop:ass-ldp-diff}]
Let $(\mu_k)_{k \geq 1}$ be a sequence that weakly converges towards $\mu$ in $\ps_\phi(E)$. 
The condition~\eqref{it:ass-ldp-diff:1} in Proposition~\ref{prop:ass-ldp-diff} allows us to apply Lemma~\ref{lem:integ-diff} to get, for any $t \in \T$:
\[ \Psi_t(\mu_k) - \Psi_t(\mu) = \int_0^1 \bigg\langle \mu_k - \mu, \frac{\delta\Psi_t}{\delta\mu}((1-r) \mu + r \mu_k) \bigg\rangle \d r =: \int_E f^k_t \d \mu_k - \int_E f_t \d \mu,  \]
with
\begin{align*}
    f^k_t (x) &:= \int_0^1 \bigg[ \frac{\delta\Psi_t}{\delta\mu}((1-r) \mu + r \mu_k,x) - \bigg\langle \mu, \frac{\delta\Psi_t}{\delta\mu}((1-r) \mu + r \mu_k) \bigg\rangle \bigg] \d r,\\
    f_t(x) &:= \frac{\delta\Psi_t}{\delta\mu}(\mu,x) - \bigg\langle \mu, \frac{\delta\Psi_t}{\delta\mu}(\mu) \bigg\rangle.
\end{align*}

We shall check that the functions $f_t^k$ and $f_t$ satisfy the assumptions of Lemma~\ref{lem:CVUnif}, which thus yields
\[
    \sup_{t \in \T} \bigg\lvert \Psi_t(\mu_k) - \Psi_t(\mu) \bigg\rvert \xrightarrow[k \rightarrow +\infty]{} 0
\]
and completes the proof. We first show that the family $( f^k_t)_{t \in \T, \, k \geq 1}$ is equi-continuous. Let $(x_l)_{l \geq 1}$ be a sequence of elements of $E$ converging to some $x \in E$. For any $t \in \T$ and $k \geq 1$,
\[
    |f^k_t(x_l)- f^k_t(x)| = \left|\int_0^1 \bigg[ \frac{\delta\Psi_t}{\delta\mu}((1-r) \mu + r \mu_k,x_l)-\frac{\delta\Psi_t}{\delta\mu}((1-r) \mu + r \mu_k,x)\bigg] \d r\right|,
\]
so that, setting $K = \cup_{k \geq 1} \{(1-r)\mu + r\mu_k, r \in [0,1]\}$ which is easily seen to be compact,
\[
    \sup_{t \in \T, k \geq 1} |f^k_t(x_l)- f^k_t(x)| \leq \sup_{t \in \T, \rho \in K} \left|\frac{\delta\Psi_t}{\delta\mu}(\rho,x_l)-\frac{\delta\Psi_t}{\delta\mu}(\rho,x)\right| \xrightarrow[l \rightarrow +\infty]{} 0,
\]
thanks to the condition~\eqref{it:ass-ldp-diff:2} in Proposition~\ref{prop:ass-ldp-diff}. Let us now show that $f^k_t(x) \to f_t(x)$, uniformly in $t \in \T$. For any $x \in E$, we have
\[
    \sup_{t \in \T} \left|f^k_t(x) - f_t(x)\right| \leq \int_0^1 [ \varrho_{k,r}(x) + \langle \mu, \varrho_{k,r}\rangle ] \d r,
\]
with
\[
    \varrho_{k,r}(x) := \sup_{t \in \T} \left|\frac{\delta \Psi_t}{\delta\mu}((1-r)\mu+r\mu_k,x)-\frac{\delta \Psi_t}{\delta\mu}(\mu,x)\right|.
\]
The condition~\eqref{it:ass-ldp-diff:3} in Proposition~\ref{prop:ass-ldp-diff} implies that, for any $r \in [0,1]$ and $x \in E$,
\[
    \varrho_{k,r}(x) \xrightarrow[k \rightarrow +\infty]{} 0.
\]
Since, in addition, the condition~\eqref{it:ass-ldp-diff:1} in Proposition~\ref{prop:ass-ldp-diff} yields $\varrho_{k,r}(x) \leq 2D_K^\Psi[1+\phi(x)]$ with the same compact set $K$ as above, we deduce from the dominated convergence theorem that
\[
    \int_0^1 [ \varrho_{k,r}(x) + \langle \mu, \varrho_{k,r}\rangle ] \d r \xrightarrow[k \rightarrow +\infty]{} 0.
\]
It remains to check the uniform integrability condition of Lemma~\ref{lem:CVUnif}. The latter easily follows from the condition~\eqref{it:ass-ldp-diff:1} in Proposition~\ref{prop:ass-ldp-diff} and the fact that, from \cite[Theorem 4.5.6]{bogachev2007measure}, the weak convergence of $\mu_k$ in $\ps_\phi(E)$ implies that 
\[
    \sup_{k \geq 1} \int_E \phi \1_{\phi > M}\d \mu_k \xrightarrow[M \rightarrow +\infty]{} 0,
\]
so the proof is completed.
\end{proof}

\subsection{Proofs of the results of Section~\ref{ss:gibbs-density}}\label{ss:proofs-gibbs-density}

As a preliminary step for the proof of Theorem~\ref{thm:abstractGibbs}, we first show that any minimiser for~\eqref{eq:min-pb} is also a minimiser for a linearised version of this problem. 

\begin{lemma}[Linearisation of $\F$ and $\Psi_t$] \label{lem:lin}
Let $\overline\mu$ be a minimiser for \eqref{eq:min-pb} at which the regularity condition \ref{ass:DiffF-barmu} and the constraint qualification \ref{ass:regPsi-barmu} hold. Let $\tilde\varepsilon > 0$ be given by~\ref{ass:regPsi-barmu}-\eqref{ass:regPsi-barmu:5}.
\begin{enumerate}
    \item\label{it:lemLinInt} For any $\mu \in \AdF \cap \Adeq$ that satisfies 
    \begin{equation} \label{eq:ConsLin}
        \forall \t \in \T, \quad \Psi_{\t}(\overline\mu) + \tilde\varepsilon \bigg\langle \mu - \overline\mu, \frac{\delta\Psi_{\t}}{\delta\mu}(\overline\mu) \bigg\rangle \leq 0,
    \end{equation}
    we have
    \[
        \left\langle \mu, \left|\log \frac{\d\overline\mu}{\d\nu}\right|\right\rangle < +\infty.
    \]
    \item\label{it:lemLinMin} The measure $\overline\mu$ is a minimiser for
    \[ \inf_{\mu \in \AdF \cap \Adeq \text{ s.t.~\eqref{eq:ConsLin} holds} }  \bigg\langle \mu, \log \frac{\d\overline\mu}{\d\nu} + \frac{\delta\F}{\delta\mu}(\overline\mu) \bigg\rangle, \]
    and the minimum value is $H(\overline\mu|\nu)$.
\end{enumerate}
\end{lemma}

\begin{proof}
Let $\tilde\varepsilon$ be given by~\ref{ass:regPsi-barmu}-\eqref{ass:regPsi-barmu:5}. 
Let $\mu$ be any measure in $\AdF \cap \Adeq$ that satisfies~\eqref{eq:ConsLin} with a strict inequality, namely 
\begin{equation} \label{eq:ConsQuali2}
\forall \t \in \T, \quad \Psi_{\t}(\overline\mu) + \tilde\varepsilon \bigg\langle \mu - \overline\mu, \frac{\delta\Psi_{\t}}{\delta\mu}(\overline\mu) \bigg\rangle < 0.
\end{equation}
For $\varepsilon$ in $(0,\tilde\varepsilon]$, we define the probability measure
$\overline\mu_\varepsilon := (1-\varepsilon)\overline{\mu} + \varepsilon {\mu}$, 
which belongs to $\AdF \cap \Adeq$ by \ref{ass:DiffF-barmu}. 
By Remark~\ref{rem:continuity-new} and the choice of $\mu$, we have
\begin{equation*}
    -\eta := \sup_{t \in \T} \left\{\Psi_{\t}(\overline\mu) + \tilde\varepsilon \bigg\langle \mu - \overline\mu, \frac{\delta\Psi_{\t}}{\delta\mu}(\overline\mu) \bigg\rangle\right\} < 0.
\end{equation*}
Moreover, by~\ref{ass:regPsi-barmu}-\eqref{ass:regPsi-barmu:2}, we have, for every small enough $\varepsilon$ and every $t \in \T$,
\begin{align*}
    \Psi_t(\overline\mu_\varepsilon) &\leq \Psi_t(\overline\mu) + \varepsilon \left(\left\langle \mu - \overline\mu, \frac{\delta\Psi_{\t}}{\delta\mu}(\overline\mu) \right\rangle + \frac{\eta}{2\tilde\varepsilon}\right) \\
    &\leq \frac{\varepsilon}{\tilde\varepsilon}\left(\Psi_t(\overline\mu) + \tilde\varepsilon \left\langle \mu - \overline\mu, \frac{\delta\Psi_{\t}}{\delta\mu}(\overline\mu) \right\rangle + \frac{\eta}{2}\right) \leq - \frac{\varepsilon\eta}{2\tilde\varepsilon},
\end{align*}
where we have used the fact that $\varepsilon \leq \tilde\varepsilon$ and $\Psi_t(\overline\mu) \leq 0$, for any $t \in \T$, at the second line. This shows that $\overline\mu_\varepsilon \in \Adineq$ for every small enough $\varepsilon$.

\medskip

\noindent\emph{\textbf{Step 1.} Linearisation.} For small enough $\varepsilon$, $\overline{\mu}_\varepsilon$ is admissible for \eqref{eq:min-pb}, hence 
\[  H(\overline{\mu} \vert \nu) + \F(\overline{\mu}) \leq H(\overline{\mu}_\varepsilon \vert \nu) + \F(\overline{\mu}_\varepsilon) < +\infty, \]
by optimality of $\overline{\mu}$.
We then divide by $\varepsilon$ and we send it to $0$.
Using~\ref{ass:DiffF-barmu},
\begin{equation*} 
\varepsilon^{-1} [ \F(\overline\mu_\varepsilon) - \F(\overline\mu) ] \xrightarrow[\varepsilon \rightarrow 0]{} \biggl\langle \mu - \overline{\mu}, \frac{\delta \F}{\delta\mu}(\overline{\mu}) \biggr\rangle,
\end{equation*} 
and we deduce from Lemma \ref{lem:HDiff} that 
\[ 0 \leq \int_E \log \frac{\d \overline{\mu}}{\d \nu} \d {\mu} - H(\overline{\mu}\vert \nu) + \biggl\langle \mu - \overline{\mu}, \frac{\delta \F}{\delta\mu}(\overline{\mu}) \biggr\rangle, \]
the integral belonging to $\R \cup \{-\infty\}$. However $H(\overline{\mu}\vert \nu) < +\infty$, hence the sign condition gives the finiteness of the integral: we have proved that the first part of Lemma \ref{lem:lin} holds under the condition~\eqref{eq:ConsQuali2}. Moreover, recalling that $\langle \overline\mu, \tfrac{\delta\F}{\delta\mu}(\overline\mu) \rangle =0$, we conclude that
\[ H(\overline\mu \vert \nu) \leq \inf_{\mu \in \AdF \cap \Adeq \text{ s.t.~\eqref{eq:ConsQuali2} holds}} \biggl\langle \mu, \log \frac{\d \overline{\mu}}{\d \nu} + \frac{\delta \F}{\delta\mu}(\overline{\mu}) \biggr\rangle. \]

\noindent\emph{\textbf{Step 2.} Closure.}
Let now $\mu$ be any measure in $\AdF \cap \Adeq$ which satisfies~\eqref{eq:ConsLin}. Let $\tilde{\mu}$ be given by~\ref{ass:regPsi-barmu}-\eqref{ass:regPsi-barmu:5}.
For $\varepsilon$ in $(0,1]$, $\mu_\varepsilon := (1-\varepsilon)\mu + \varepsilon\tilde{\mu}$ belongs to $\AdF \cap \Adeq$ from~\ref{ass:DiffF-barmu},
and since, for any $t \in \T$,
\begin{align*}
    &\Psi_{\t}(\overline\mu) + \tilde\varepsilon \bigg\langle \mu_\varepsilon - \overline\mu, \frac{\delta\Psi_{\t}}{\delta\mu}(\overline\mu) \bigg\rangle\\
    &= (1-\varepsilon)\left(\Psi_{\t}(\overline\mu) + \tilde\varepsilon \bigg\langle \mu - \overline\mu, \frac{\delta\Psi_{\t}}{\delta\mu}(\overline\mu) \bigg\rangle\right) + \varepsilon\left(\Psi_{\t}(\overline\mu) + \tilde\varepsilon \bigg\langle \tilde\mu - \overline\mu, \frac{\delta\Psi_{\t}}{\delta\mu}(\overline\mu) \bigg\rangle\right),
\end{align*}
then $\mu_\varepsilon$ satisfies \eqref{eq:ConsQuali2}. Since $\tilde\mu$ also satisfies \eqref{eq:ConsQuali2}, we deduce from \emph{\textbf{Step 1.}} that $\langle \mu, |\log \frac{\d\overline\mu}{\d\nu}|\rangle < +\infty$, which is the first part of Lemma \ref{lem:lin}. Moreover, using the conclusion of \emph{\textbf{Step 1.}} applied to $\mu_\varepsilon$, we get
\begin{align*} 
H(\overline\mu \vert \nu) &\leq \biggl\langle \mu_\varepsilon, \log \frac{\d \overline{\mu}}{\d \nu} + \frac{\delta \F}{\delta\mu}(\overline{\mu}) \biggr\rangle \\
&= (1-\varepsilon) \biggl\langle \mu, \log \frac{\d \overline{\mu}}{\d \nu} + \frac{\delta \F}{\delta\mu}(\overline{\mu}) \biggr\rangle + \varepsilon \biggl\langle \tilde\mu, \log \frac{\d \overline{\mu}}{\d \nu} + \frac{\delta \F}{\delta\mu}(\overline{\mu}) \biggr\rangle.
\end{align*}
Sending $\varepsilon$ to $0$ eventually yields 
\[
    \bigg\langle \overline\mu, \log \frac{\d\overline\mu}{\d\nu} + \frac{\delta\F}{\delta\mu}(\overline\mu) \bigg\rangle = H(\overline\mu \vert \nu) \leq \bigg\langle \mu, \log \frac{\d\overline\mu}{\d\nu} + \frac{\delta\F}{\delta\mu}(\overline\mu) \bigg\rangle,
\]
which completes the proof of the second part of Lemma \ref{lem:lin}.
\end{proof}

We may now present the proof of Theorem~\ref{thm:abstractGibbs}. 

\begin{proof}[Proof of Theorem \ref{thm:abstractGibbs}]
Since, by definition, $\overline\mu \in \AdF$, $\overline\mu$ is absolutely continuous w.r.t. $\nu$. We fix a (measurable) density $\frac{\d\overline\mu}{\d\nu}$ and define
\[
    \overline S := \left\{x \in E, \; \frac{\d \overline{\mu}}{\d\nu}(x) > 0\right\}
\]
which by construction is measurable and satisfies $\overline\mu(\overline S)=1$.

\medskip

\noindent\emph{\textbf{Step 1.} Absolute continuity.}
Let $\mu$ be any measure in $\AdF \cap \Adeq$. From Remark \ref{rem:continuity-new}, $t \mapsto \Psi_t(\overline\mu) + \tilde\varepsilon\langle \mu - \overline{\mu}, \tfrac{\delta\Psi_t}{\delta\mu}(\overline\mu) \rangle$ is continuous on the compact set $\T$, hence bounded.
As a consequence, $\tilde\mu_\varepsilon := (1-\varepsilon) \tilde\mu + \varepsilon \mu$ satisfies \eqref{eq:ConsQuali2} for $\varepsilon > 0$ small enough, where $\tilde\mu$ is given by the constraint qualification~\ref{ass:regPsi-barmu}-\eqref{ass:regPsi-barmu:5}. Lemma~\ref{lem:lin}-\ref{it:lemLinInt} applied to $\tilde\mu_\varepsilon$ then implies that $\int_E \log \tfrac{\d \overline\mu}{\d\nu} \, \d\tilde\mu_\varepsilon > -\infty$, hence $\tfrac{\d \overline\mu}{\d\nu}$ is $\tilde\mu_\varepsilon$-a.s. positive and then $\mu$-a.s. positive. Any $\mu \in \AdF \cap \Adeq$ is thus absolutely continuous w.r.t. $\overline\mu$, or equivalently $\mu ( \overline{S} ) =1$, proving Theorem \ref{thm:abstractGibbs}-\ref{it:thmAbsAbsCont}.
This holds in particular for $\tilde\mu$. 

\medskip

\noindent\emph{\textbf{Step 2.} Computation of the density.}
Let us consider the map
\[ \Phi :
\begin{cases}
E \rightarrow \R, \\
x \mapsto \log \frac{\d \overline{\mu}}{\d \nu}(x) + \frac{\delta \F}{\delta\mu}(\overline{\mu},x) - H( \overline\mu \vert \nu)
\end{cases}
\]
which belongs to $L^1 ( E , \d \overline{\mu} )$.
We notice that $\Phi$ is finite on $\overline{S}$.
From Lemma~\ref{lem:lin}-\ref{it:lemLinInt}, $\Phi$ also belongs to $L^1 ( E , \d \tilde{\mu} )$. 
We define $\C$ as the convex cone generated by the set
\[ \{ \pm \zeta_s , \, s \in \S \} \cup \bigg\{ - \Psi_\t (\overline{\mu}) - \tilde\varepsilon \frac{\delta\Psi_\t}{\delta \mu}(\overline\mu), \, t \in \T \bigg\}, \]
i.e. the elements of $\C$ are the (finite) linear combinations with non-negative coefficients of elements from the above set.
We then set $\C^+ := \C + C_b (E,\R_+)$.
Let us show that $\Phi$ belongs to the closure of $\C^+$ in $L^1 ( E , \d \overline{\mu} ) \cap L^1 ( E , \d \tilde{\mu} )$.
If it were not the case, using the geometric Hahn-Banach theorem \cite[Corollary 4.5]{barbu2012convexity}, $\Phi$ could be strictly separated from the closure of $\C^+$, which is a closed convex set, by an affine hyperplane.
Since $L^1 ( E , \d \overline{\mu} ) \cap L^1 ( E , \d \tilde{\mu} )$ is dense in $L^1 ( E , \d \overline{\mu} )$ and $L^1 ( E , \d \tilde{\mu} )$, the topological dual of $L^1 ( E , \d \overline{\mu} ) \cap L^1 ( E , \d \tilde{\mu} )$ is $L^\infty ( E , \d \overline{\mu} ) + L^\infty ( E , \d \tilde{\mu} )$ from \cite[Theorem 2.7.1]{bergh2012interpolation}.
The strict separation thus provides $(h_1,h_2) \in L^\infty ( E , \d \overline{\mu} ) \times L^\infty ( E , \d \tilde{\mu} )$ such that
\begin{equation} \label{eq:Separation}
\int_E \Phi h_1 \, \d \overline\mu + \int_E \Phi h_2 \, \d \tilde\mu < \inf_{\varphi \in \C^+} \int_E \varphi h_1 \, \d \overline\mu + \int_E \varphi h_2 \, \d \tilde\mu. 
\end{equation} 
The infimum on the r.h.s. is non-positive because the function $\varphi \equiv 0$ belongs to $\C^+$.
If this infimum were negative, it would be $-\infty$ because, for every $\varphi \in \C^+$, $\alpha \varphi \in \C^+$ for every $\alpha >0$.
Since the l.h.s. of \eqref{eq:Separation} is finite, the infimum on the r.h.s. thus equals $0$.

From \emph{\textbf{Step 1.}}, $\tilde\mu$ is absolutely continuous w.r.t. $\overline{\mu}$, so that we can define
\[ h := h_1 + h_2 \frac{\d \tilde\mu}{\d\overline\mu} \in L^1(E, \d \overline\mu). \]
Since $\C^+$ contains $C_b(E,\R_+)$ and the r.h.s. of \eqref{eq:Separation} is non-negative (it equals $0$), $h$ is $\overline\mu$-a.s. non-negative.
Up to modifying the bounded functions $h_1$ and $h_2$, we can then assume that $h_1$ and $h_2$ are $\overline\mu$-a.s. non-negative. 
Similarly, \eqref{eq:Separation} can be divided by any positive constant; hence we can assume that $\int_E h \, \d \overline\mu = 1$. 

Let $\mu_h$ denote the measure with density $h$ w.r.t. $\overline\mu$.
\textcolor{red}{les h en idnices ici}
Let $\mu_h^1$ (resp. $\mu^2_h$) denote the probability measure with density w.r.t. $\overline{\mu}$ (resp. $\tilde\mu$) proportional to $h_1$ (resp. $h_2$).
From the definition of $h$, $\mu_h$ is a convex combination of $\mu_h^1$ and $\mu_h^2$.
Since $\mu^1_h$ (resp. $\mu^h_2$) has a bounded density w.r.t. to $\overline{\mu}$ (resp. $\tilde\mu$) and $H ( \overline{\mu} \vert \nu )$ (resp. $H ( \tilde{\mu} \vert \nu )$) is finite, this implies that $H( \mu_h \vert \nu )$ is finite.
Using the same argument and \ref{ass:DiffF-barmu}-\ref{ass:Fboundedens-barmu}, we further get that $\mu_h$ belongs to $\AdF$.
For every $s \in \S$, we now use that the r.h.s. of \eqref{eq:Separation} is non-negative when considering the test functions $\varphi = \zeta_s$ and $\varphi = -\zeta_s$, which belong to $\C^+$.
This shows that $\langle \mu_h, \zeta_s \rangle$ = 0 for every $s \in \S$, so that $\mu_h \in \Adeq$.
We can thus apply Lemma \ref{lem:lin}-\ref{it:lemLinMin} to $\mu_h$, yielding
\begin{equation*} 
0 \leq \int_{E} \log \frac{\d \overline{\mu}}{\d\nu} \, \d \mu_h + \left\langle \mu_h, \frac{\delta\F}{\delta\mu}(\overline{\mu})\right\rangle - H(\overline{\mu}|\nu) = \int_E h \Phi \, \d \overline\mu.
\end{equation*}
Since the r.h.s. of \eqref{eq:Separation} is $0$, this gives the desired contradiction.

\medskip

\noindent\emph{\textbf{Step 3.} Lagrange multiplier.}
From \emph{\textbf{Step 2.}}, there exists a sequence $((f_k,g_k))_{k \geq 1}$ in $\C \times C_b(E,\R_+)$ such that 
\[ f_k + g_k \xrightarrow[k \rightarrow + \infty]{L^1(\overline\mu) \cap L^1(\tilde\mu)} \Phi. \]
By definition of $\C$, $\int_E f_k \, \d \overline{\mu} \geq 0$ because  $\overline{\mu} \in \Ad$. 
Since $\int_E \Phi \, \d \overline{\mu} =0$, this yields
\[ \lVert g_k \rVert_{L^1( \overline\mu)} = \int_E g_k \, \d \overline{\mu} \xrightarrow[k \rightarrow + \infty]{} 0, \]
so that $\Phi$ belongs to the closure of $\C$ in $L^1(E,\d \overline\mu)$.
By definition of $\C$, each $f_k$ can be decomposed as
\begin{equation} \label{eq:preMult}
f_k = -\xi_k - \sum_{i = 1}^{n_k} \lambda_{i,k} \bigg[ \Psi_{\t_{i,k}} (\overline{\mu}) + \tilde\varepsilon \frac{\delta\Psi_{\t_{i,k}}}{\delta \mu}(\overline\mu) \bigg] , 
\end{equation}
where $\xi_k \in \mathrm{Span}( \zeta_s , \, s \in \S)$, $\{ t_{i,k}, 1 \leq i \leq n_k \}$ is a finite subset of $\T$ and the $\lambda_{i,k}$ are non-negative real numbers. 
We then define
\[ \lambda_k := \sum_{i = 1}^{n_k} \lambda_{i,k} \delta_{t_{i,k}} \; \in \; \M_+ (\T). \]
By Remark~\ref{rem:continuity-new}, we have
\begin{equation*}
-\eta := \sup_{t \in \T} \bigg\{\Psi_{\t}(\overline\mu) + \tilde\varepsilon \bigg\langle \tilde\mu - \overline\mu, \frac{\delta\Psi_{\t}}{\delta\mu}(\overline\mu) \bigg\rangle\bigg\} < 0.
\end{equation*}
Hence, integrating \eqref{eq:preMult} against $\tilde\mu$ and using that $g_k \geq 0$,
\[ \int_E [ f_k + g_k ] \, \d \tilde\mu = \int_E g_k \d \tilde\mu - \int_\T \Psi_t (\overline{\mu}) + \tilde\varepsilon \bigg\langle \tilde\mu-\overline\mu, \frac{\delta\Psi_t}{\delta\mu} (\overline{\mu}) \bigg\rangle \lambda_k ( \d t ) \geq \eta \lambda_k ( \T ). \]
Since $\int_E [ f_k + g_k] \, \d \tilde\mu$ converges, this shows that $(\lambda_k(\T))_{k \geq 1}$ is bounded.  
Up to extracting a sub-sequence, we can assume that $(\lambda_k(\T) )_{k \geq 1}$ converges towards some $\tilde\lambda(\T) \geq 0$.
If $\tilde\lambda(\T) = 0$, we set $\tilde\lambda := 0$.
Otherwise, $( \tfrac{\lambda_k}{\lambda_k(\I)} )_{k \in \N}$ is a sequence of probability measures over the compact set $\I$; hence, it is relatively compact by the Prokhorov theorem \cite[Theorem 5.1]{billingsley2013convergence}.
Consequently, we can always assume that $( \lambda_k )_{k \in \N}$ weakly converges towards some $\tilde\lambda \in \M_+(\T)$ with mass $\tilde\lambda ( \T)$. Thus, for any $x \in E$, by the continuity assumptions~\ref{ass:regPsi-barmu}-(\ref{ass:regPsi-barmu:1},\ref{ass:regPsi-barmu:3}),
\[
    \int_\T \bigg[ \Psi_t (\overline{\mu}) + \tilde\varepsilon \frac{\delta\Psi_t}{\delta \mu}(\overline\mu,x) \bigg] \lambda_k(\d t) \xrightarrow[k \rightarrow +\infty]{} \int_\T \bigg[ \Psi_t (\overline{\mu}) + \tilde\varepsilon \frac{\delta\Psi_t}{\delta \mu}(\overline\mu,x) \bigg] \tilde\lambda(\d t),
\]
and by the domination~\ref{ass:regPsi-barmu}-\eqref{ass:regPsi-barmu:4}, this convergence holds in $L^1(E,\d\overline\mu)$.
Since $f_k$ converges in $L^1(E,\d \overline\mu)$, \eqref{eq:preMult} now shows that $\xi_k$ converges in $L^1(E,\d \overline\mu)$ towards some $\overline\zeta$ in the closure of $\mathrm{Span}( \zeta_s , \, s \in \S)$. As a consequence, we may choose $\overline\zeta$ such that, for all $x \in \overline S$,
\begin{equation} \label{eq:PreDebsity}
    \Phi(x) = -\overline\zeta(x) - \int_\T \bigg[ \Psi_t (\overline{\mu}) + \tilde\varepsilon \frac{\delta\Psi_t}{\delta \mu}(\overline\mu,x) \bigg] \tilde\lambda(\d t).
\end{equation}
Since $\xi_k \in \mathrm{Span}( \zeta_s , \, s \in \S)$ and $\overline\mu \in \Adeq$, we have $\int_E \xi_k \, \d \overline\mu = 0$; hence we get $\int_E \overline\zeta \, \d \overline\mu = 0$, by $L^1(E,\d \overline\mu)$-convergence of $ \xi_k$ to $\overline\zeta$.
Moreover, $\int_E \Phi \, \d \overline\mu = \int_E \tfrac{\delta \F}{\delta\mu}(\overline\mu) \, \d\overline\mu = 0$ and $\int_E \tfrac{\delta \Psi_t}{\delta\mu}(\overline\mu) \, \d\overline\mu = 0$, so that integrating \eqref{eq:PreDebsity} against $\overline{\mu}$ yields
\[ \int_\T \Psi_t (\overline\mu) \, \tilde\lambda (\d t) = 0. \]
Since $\Psi_t (\overline\mu) \leq 0$ for every $t \in \T$, we get $\Psi_t (\overline\mu) = 0$ for $\tilde\lambda$-a.e. $t \in \T$.
Setting $\overline\lambda := \tilde\varepsilon \tilde\lambda$, this proves Theorem \ref{thm:abstractGibbs}-\ref{it:thmAbsSlack}. 
We now take the exponential of \eqref{eq:PreDebsity} to obtain that
\[ \forall x \in \overline S, \quad \frac{\d \overline{\mu}}{\d\nu} ( x ) = e^{ H( \overline{\mu} \vert \nu ) } \exp \bigg[ - \frac{\delta\F}{\delta\mu} ( \overline{\mu}, x ) - \overline{\zeta} ( x) - \int_\T \frac{\delta\Psi_t}{\delta\mu} ( \overline{\mu}, x ) \overline\lambda ( \d t ) \bigg], \]
while $\frac{\d \overline{\mu}}{\d\nu} ( x )=0$ for $x \not\in \overline S$.
Since 
$\int_E \tfrac{\d \overline{\mu}}{\d\nu} \d\nu = 1$, we get that
\[ \overline{Z} := \int_E \1_{\overline{S}}(x) \exp \biggl[ - \frac{\delta \F}{\delta \mu}(\overline\mu,x) - \overline\zeta (x)- \int_\T \frac{\delta\Psi_{\t}}{\delta \mu}(\overline\mu,x) \overline{\lambda}( \d \t) \biggr] \nu(\d x) = e^{- H( \overline{\mu} \vert \nu )} \in (0,\infty). \]
This implies that $\overline{\mu}$ is the Gibbs measure~\eqref{eq:Gibbsmeasure}, completing the proof of Theorem \ref{thm:abstractGibbs}-\ref{it:thmAbsGibbs}.
\end{proof}

\begin{rem}[Bound on the Lagrange multiplier] \label{rem:BoundLambda} 
The computations in \emph{\textbf{Step 3.}} show that 
\begin{equation*}
\overline{\lambda}(\I) \leq \tilde\varepsilon \eta^{-1} \biggl\langle \tilde\mu, \log \frac{\d \overline{\mu}}{\d \nu} + \frac{\delta \F}{\delta\mu}(\overline{\mu}) \biggr\rangle - \tilde\varepsilon \eta^{-1} H(\overline\mu \vert \nu),
\end{equation*}
where $\eta := -\sup_{t \in \T} [ \Psi_{\t}(\overline\mu) + \tilde\varepsilon \langle \tilde\mu - \overline\mu, \frac{\delta\Psi_{\t}}{\delta\mu}(\overline\mu) \rangle ]$.
This gives a bound on the mass of $\overline{\lambda}$.
\end{rem}

\subsection{Proofs of the results of Section~\ref{ss:convex}}\label{ss:proofs-convex}

\begin{proof}[Proof of Theorem~\ref{thm:convex-case}] 

Let $\mu \in \AdF \cap \Ad$. Since $H(\mu|\nu) < \infty$ and $\nu(\overline{S})=1$, \cite[Lemma 1.4-(b)]{nutz2021introduction} ensures that the identity
\begin{equation*}
    H(\mu|\nu) - H(\mu|\overline\mu) = \int_E \bigg[ -\log \overline{Z} - \frac{\delta\F}{\delta\mu}(\overline\mu,x) -\overline\zeta(x) - \int_\T \frac{\delta\Psi_t}{\delta\mu}(\overline\mu,x)\overline\lambda(\d t) \bigg] \d\mu(x)
\end{equation*}
holds in $[-\infty,+\infty)$. 
The differentiability assumption \eqref{it:SuffDiff} on $\F$ and $\Psi_t$, together with domination \eqref{it:SuffDom} and the assumption \eqref{it:SuffSupp} on $\overline\zeta$ imply that 
\begin{equation*}
    \left\langle \mu, \left|\frac{\delta\F}{\delta\mu}(\overline\mu)\right|\right\rangle < \infty, \qquad \left\langle \mu, |\overline\zeta| \right\rangle < \infty, \qquad  \int_\T\left\langle \mu, \left|\frac{\delta\Psi_t}{\delta\mu}(\overline\mu)\right|\right\rangle\overline\lambda(\d t) < \infty,
\end{equation*}
which shows that $H(\mu|\overline\mu)<\infty$ and that
\begin{equation*}
    H(\mu|\nu) - H(\mu|\overline\mu) = -\log \overline{Z} - \left\langle \mu, \frac{\delta\F}{\delta\mu}(\overline\mu)\right\rangle - \int_\T\left\langle \mu, \frac{\delta\Psi_t}{\delta\mu}(\overline\mu)\right\rangle\overline\lambda(\d t),
\end{equation*}
since $\langle \mu, \overline\zeta \rangle = 0$ from \eqref{it:SuffSupp}. On the other hand, using that $\overline{\mu}$ is a Gibbs measure from \eqref{it:SuffEq},
\begin{equation*}
    H(\overline\mu|\nu) = -\log \overline{Z} - \left\langle \overline\mu, \frac{\delta\F}{\delta\mu}(\overline\mu)\right\rangle - \int_\T \left\langle \overline\mu,\frac{\delta\Psi_t}{\delta\mu}(\overline\mu)\right\rangle\overline\lambda(\d t).
\end{equation*}
We deduce that
\begin{equation*}
    H(\mu|\nu) = H(\mu|\overline\mu) + H(\overline\mu|\nu) + \left\langle \overline\mu-\mu, \frac{\delta\F}{\delta\mu}(\overline\mu)\right\rangle + \int_\T \left\langle \overline\mu-\mu, \frac{\delta\Psi_t}{\delta\mu}(\overline\mu)\right\rangle\overline\lambda(\d t).
\end{equation*}
By the convexity assumption~\ref{ass:diff-glob-conv}, the differentiability \eqref{it:SuffDiff} of $\F$ and $\Psi_t$ at $\overline\mu$, and~\eqref{eq:convex-diff},
\begin{equation*}
    \F(\mu) \geq \F(\overline{\mu}) + \left\langle \mu-\overline\mu, \frac{\delta\F}{\delta\mu}(\overline{\mu})\right\rangle, \qquad \Psi_t(\mu) \geq \Psi_t(\overline{\mu}) + \left\langle \mu-\overline\mu, \frac{\delta\Psi_t}{\delta\mu}(\overline{\mu})\right\rangle.
\end{equation*}
Since $\mu$ is admissible and the slackness condition~\eqref{eq:compSlackness} holds, we have
\begin{equation*}
    \left\langle \mu-\overline\mu, \frac{\delta\Psi_t}{\delta\mu}(\overline{\mu})\right\rangle \leq \Psi_t(\mu) - \Psi_t(\overline{\mu}) \leq 0, \quad \text{   for $\overline\lambda(\d t)$-a.e.   } t \in \T,
\end{equation*}
so that
\begin{equation*}
    \int_\T \left\langle \mu-\overline\mu, \frac{\delta\Psi_t}{\delta\mu}(\overline{\mu})\right\rangle \overline\lambda(\d t) \leq 0
\end{equation*}
and therefore
\begin{equation*}
    H(\mu|\nu) + \F(\mu) \geq H(\mu|\overline\mu) + H(\overline\mu|\nu) + \F(\overline{\mu}),
\end{equation*}
which concludes because $H(\mu|\overline\mu) \geq 0$
\end{proof}

\begin{proof}[Proof of Proposition~\ref{pro:StrongStab}]
    For any $\varepsilon>0$, the minimisation problem~\eqref{eq:min-pb-eps} is the same as~\eqref{eq:min-pb}, with $\Psi_t$ replaced by $\Psi_t-\varepsilon$. 
    Therefore, since $\F$ is bounded from below, Lemmata~\ref{lem:Igood} and~\ref{lem:AdClosed} for~\eqref{eq:min-pb-eps} show that this problem admits a minimiser $\overline\mu_\varepsilon$, which is unique by convexity of $\F$ and $\Psi_t-\varepsilon$. Since $\Adeps \subset \Adepsstar$ for $\varepsilon \in [0,\varepsilon_*]$, the regularity assumptions~\ref{ass:DiffF-barmu}, \ref{ass:regPsi-barmu}-(\ref{ass:regPsi-barmu:1},\ref{ass:regPsi-barmu:2},\ref{ass:regPsi-barmu:3},\ref{ass:regPsi-barmu:4}) and~\ref{ass:Fboundedens-barmu} are satisfied at $\overline\mu_\varepsilon$ for $\Psi_t-\varepsilon$. 
    Last, Lemma~\ref{lem:EquivQualif} trivially implies that the qualification condition~\ref{ass:regPsi-barmu}-\eqref{ass:regPsi-barmu:5} holds for $\Psi_t-\varepsilon$, therefore Theorem~\ref{thm:abstractGibbs} applies to $\overline\mu_\varepsilon$.
    This proves Theorem \ref{thm:convex-case}-\ref{it:StabMin}.
    
    For $\varepsilon \in [0,\varepsilon_*]$ and $(\mu,\lambda) \in \Adeq \times \M_+ (\T)$, we define
    \[ \mathcal{L}^\varepsilon ( \mu, \lambda) := \mathcal{I}(\mu) + \int_\T [ \Psi_t ( \mu ) - \varepsilon ] \lambda ( \d t ). \]
    By the convexity assumption~\ref{ass:diff-glob-conv} and~\eqref{eq:convex-diff}, for any $(\mu,\lambda) \in \Adeq \times \M_+ (\T)$,
    \[
        \Psi_t(\mu) \geq \Psi_t(\overline\mu_\varepsilon) + \left\langle \mu-\overline\mu_\varepsilon, \frac{\delta\Psi_t}{\delta\mu}(\overline\mu_\varepsilon)\right\rangle,
    \]
    and by the domination assumption~\ref{ass:regPsi-barmu}-\eqref{ass:regPsi-barmu:4} for $\overline\mu_\varepsilon$, there is $D^\Psi_\varepsilon \in [0,\infty)$ such that
    \[
        \int_\T \left|\left\langle \mu-\overline\mu_\varepsilon, \frac{\delta\Psi_t}{\delta\mu}(\overline\mu_\varepsilon)\right\rangle\right|\lambda(\d t)  \leq \lambda(\T) D^\Psi_\varepsilon \langle \mu, 1+\phi\rangle.
    \]
    We deduce that
    \[
        \mathcal{L}^\varepsilon ( \mu, \lambda) \geq H(\mu|\nu) + \inf \F + \int_\T [\Psi_t(\overline\mu_\varepsilon)-\varepsilon]\lambda(\d t) - \lambda(\T) D^\Psi_\varepsilon \langle \mu, 1+\phi\rangle.
    \]
    By the continuity assumption~\ref{ass:regPsi-barmu}-\eqref{ass:regPsi-barmu:1} for $\overline\mu_\varepsilon$, the integral on the r.h.s. is finite, and by the dual representation for entropy~\eqref{eq:DualEntrop}, $H(\mu|\nu)-\lambda(\T) D^\Psi_\varepsilon \langle \mu, 1+\phi\rangle$ is bounded from below, uniformly in $\mu \in \ps_\phi(E)$. We conclude that, for any $\lambda \in \M_+ (\T)$,
    \[
        \inf_{\mu \in \Adeq} \mathcal{L}^\varepsilon ( \mu, \lambda) > -\infty.
    \]
    Moreover, it follows from the semi-continuity assumption~\ref{ass:nonlinear-constr}, the convexity \ref{ass:diff-glob-conv}, \eqref{eq:convex-diff} and Fatou's lemma that $\mu \mapsto \int_\T [\Psi_t(\mu)-\varepsilon]\lambda(\d t)$ is lower semi-continuous on $\ps_\phi(E)$. Therefore, since $\F$ is bounded from below, by Lemma~\ref{lem:Igood}, the level sets of $\mu \mapsto \mathcal{L}^\varepsilon ( \mu, \lambda)$ are compact in $\ps_\phi(E)$. Since $\Adeq$ is closed by Lemma~\ref{lem:AdClosed}, we deduce that there is a minimiser for $\inf_{\mu \in \Adeq} \mathcal{L}^\varepsilon ( \mu, \lambda)$, and this minimiser does not depend on $\varepsilon$ because of the shape $\mathcal{L}^\varepsilon$. 
    Moreover,
    $\varepsilon \mapsto \inf_{\mu \in \Adeq} \mathcal{L}^\varepsilon ( \mu, \lambda)$ is continuously differentiable with constant derivative $-\lambda (\T)$.

    The assumptions on $\overline\zeta_\varepsilon$ and $\overline{S}_\varepsilon$ now allow us to use the same arguments as in the proof of Theorem~\ref{thm:convex-case} to get that
    \[
        \inf_{\mu \in \Adeq} \mathcal{L}^\varepsilon ( \mu, \overline\lambda_\varepsilon) = \mathcal{L}^\varepsilon (\overline\mu_\varepsilon, \overline\lambda_\varepsilon),
    \]
    which then yields that $\sup_\lambda \inf_\mu \mathcal{L}^\varepsilon(\mu,\lambda)$ is reached at $\overline\lambda_\varepsilon$ and takes the value $\mathcal{L}^\varepsilon (\overline\mu_\varepsilon, \overline\lambda_\varepsilon)=\overline{\mathcal{I}}^\zeta_{\Psi,\varepsilon}$. 
    The envelope theorem \cite[Theorem 2]{milgrom2002envelope} now states that $\varepsilon \mapsto \overline{\mathcal{I}}^\zeta_{\Psi,\varepsilon}$ is absolutely continuous with derivative a.e. equal to $-\overline\lambda^\varepsilon (\T)$. 

    To show that $\overline\lambda^\varepsilon (\T)$ is bounded from above uniformly in $\varepsilon$, we recall that, using Lemma~\ref{lem:EquivQualif} and Remark \ref{rem:BoundLambda}, 
    \[
        \overline\lambda_\varepsilon(\T) \leq \eta^{-1} \biggl\langle \tilde\mu, \log \frac{\d \overline{\mu}_\varepsilon}{\d \nu} + \frac{\delta \F}{\delta\mu}(\overline{\mu}_\varepsilon) \biggr\rangle,
    \]
    where
    \[ \eta := -\sup_{t \in \T} \Psi_t ( \overline{\mu} ) + \bigg\langle \tilde\mu - \overline{\mu}, \tfrac{\delta\Psi_t}{\delta\mu} ( \overline{\mu} ) \bigg\rangle. \]
    By convexity of $\Psi$, the bound still holds with $\eta := -\sup_{t \in \T} \Psi_t ( \tilde\mu )$, which is positive from \ref{ass:regPsi-barmu}-(\ref{ass:regPsi-barmu:1}) and does not depend on $\varepsilon$. 
    Since $H ( \tilde\mu \vert \nu )$ is finite, \cite[Lemma 1.4-(b)]{nutz2021introduction} yields
    \[
        \biggl\langle \tilde\mu, \log \frac{\d \overline{\mu}_\varepsilon}{\d \nu}\biggr\rangle \leq H(\tilde\mu|\nu) < +\infty,
    \]
    while using the convexity~\ref{ass:diff-glob-conv} and~\eqref{eq:convex-diff},
    \[
        \biggl\langle \tilde\mu, \frac{\delta \F}{\delta\mu}(\overline{\mu}_\varepsilon) \biggr\rangle \leq \F(\tilde\mu) - \F(\overline{\mu}_\varepsilon) \leq \F(\tilde\mu) - \inf\F < +\infty.
    \]
    We therefore deduce that there exists $C_\mathrm{stab} \in [0,\infty)$ such that for any $\varepsilon \in [0,\varepsilon_*]$, $\overline\lambda_\varepsilon(\T) \leq C_\mathrm{stab}$. This yields the upper bound in~\eqref{eq:quant-stab-I}, while the lower bound simply follows from the fact that $\Adineq \subset \Adineqeps$.
\end{proof}

\begin{rem}[Towards improved bounds on $\lambda^\varepsilon$] 
The mass of $\overline\lambda^\varepsilon$ should be very small in regions where $\Psi_t (\overline\mu_\varepsilon) = \varepsilon$ and $\Psi_t (\overline\mu_0) < 0$, because the $\varepsilon$-relaxed problem should not correct much when the non-relaxed problem needs no correction. 
This can been seen by writing the Lagrangian duality: we have shown in the proof of Proposition \ref{pro:StrongStab} that $\overline\mu_\varepsilon$ is a minimiser for
\[ \inf_{\mu \in \Adeq} \mathcal{I} (\mu ) + \int_\T [\Psi_t ( \mu ) - \varepsilon] \overline\lambda^\varepsilon (\d t). \] 
Let us use $\overline\mu_0$ as a competitor for this problem: using the complementary slackness~\eqref{eq:compSlackness} for $\overline\mu_\varepsilon$, we get
\[ - \int_\T \Psi_t ( \overline\mu_0 ) \overline\lambda^\varepsilon (\d t) \leq \mathcal{I}( \overline\mu_0) - \mathcal{I}( \overline\mu_\varepsilon) - \varepsilon \overline\lambda^\varepsilon (\T), \]
and the r.h.s. is negligible in front of $\varepsilon$ because $\tfrac{\d}{\d \varepsilon} \mathcal{I}( \overline\mu_\varepsilon) = - \overline\lambda^\varepsilon (\T)$. However, deducing local bounds on $\overline\lambda^\varepsilon$ from this seems non-obvious.
\end{rem}

\begin{proof}[Proof of Lemma \ref{lem:StabQualif}]
From \ref{ass:Stab}-\eqref{it:StabQualif}, there exists $\tilde\mu \in \AdF$ such that $\sup_{t \in \T} \Psi_t(\tilde\mu) < 0$.
For $M > 1$, let us define the probability measure $\tilde\mu^M$ by
\[
    \frac{\d\tilde\mu^M}{\d\nu}(x) := \frac{1}{Z_M} \bigg[ \frac{\d\tilde\mu}{\d\nu}(x) \wedge M \bigg], \quad \text{where} \quad Z_M := \int_E \bigg[ \frac{\d\tilde\mu}{\d\nu}(x) \wedge M \bigg] \d\nu(x) \in (0,\infty).
\]
By dominated convergence, $\tilde\mu^M$ converges towards $\tilde\mu$ in $\ps_\phi (E)$ as $M \rightarrow +\infty$. 
Using the continuity assumption \ref{ass:Stab}-\eqref{it:StabCont}, this implies that $\sup_{t \in \T} \Psi_t(\tilde\mu^M) < 0$, for every large enough $M$.
For every small enough $\varepsilon >0$, the Lusin theorem for Borel measures (see e.g. \cite[Theorem 7.1.13]{bogachev2007measure}) provides a closed set $C_{M,\varepsilon}$ such that $\tfrac{\d \tilde\mu^M}{\d\nu} \vert_{C_{M,\varepsilon}}$ is continuous and $\nu(E \setminus C_{M,\varepsilon}) < \varepsilon$.
The Tietze-Urysohn extension theorem 
\cite[Theorem 35.1]{munkres2018elements} then provides a continuous $\rho_{M,\varepsilon} : E \to [0, M/Z_M]$ that coincides with $\tfrac{\d \tilde\mu^M}{\d\nu}$ on $C_{M,\varepsilon}$.
In particular, $\rho_{M,\varepsilon}$ converges to $\tfrac{\d \tilde\mu^M}{\d\nu}$, $\nu$-a.e., as $\varepsilon \to 0$. We then define the probability measure $\tilde\mu^{M,\varepsilon}$ by
\[ \frac{\d\tilde\mu^{M,\varepsilon}}{\d\nu}(x) := \frac{1}{Z_{M,\varepsilon}}[ \varepsilon+\rho_{M,\varepsilon}(x) ], \quad \text{where} \quad Z_{M,\varepsilon} := \int_E [ \varepsilon+\rho_{M,\varepsilon}(x) ] \d\nu(x) \in (0,+\infty).
\]
By dominated convergence, $\tilde\mu^{M,\varepsilon}$ converges to $\tilde\mu^M$ in $\ps_\phi(E)$, as $\varepsilon \to 0$. 
Using~\ref{ass:Stab}-\eqref{it:StabCont} again, we deduce that $\sup_{t \in \T} \Psi_t(\tilde\mu^{M,\varepsilon}) < 0$, for every small enough $\varepsilon$. 
From now on, we fix $(M,\varepsilon)$ such that this property holds.
Let $\tilde\mu^\star$ denote the related measure $\tilde\mu^{M,\varepsilon}$, and let 
$\varrho^\star$ denote its density w.r.t. $\nu$, which is positive, continuous and bounded.

For any $k \geq 1$, we define the probability measure $\tilde\mu_k$ by
\[
    \frac{\d\tilde\mu_k}{\d\nu_k} := \frac{1}{Z_k}\varrho^\star (x), \quad \text{where} \quad Z_k := \int_E \varrho^\star(x)\d\nu_k(x) \in (0,+\infty).
\]
Since $\nu_k$ converges to $\nu$ in $\ps_\phi(E)$ from \ref{ass:Stab}-\eqref{it:StabIntExp}, and $\varrho^\star$ is continuous and bounded, $\tilde\mu_k$ converges to $\tilde\mu^\star$ in $\ps_\phi(E)$. 
Similarly, $H(\tilde\mu_k|\nu_k)$ converges towards $H(\tilde\mu^\star |\nu)$; hence $( H(\tilde\mu_k|\nu_k) )_{ k \geq 1}$ is bounded.
Since $\varrho^\star$ is bounded, \ref{ass:Stab}-\eqref{it:StabDensB} gives that $\tilde\mu_k \in \AdFk$.
We now write
\[
    \sup_{t \in \T} \Psi_{t,k}(\tilde\mu_k) \leq \left|\sup_{t \in \T}\Psi_{t,k}(\tilde\mu_k) - \sup_{t \in \T}\Psi_{t,k}(\tilde\mu^\star)\right| + \sup_{t \in \T} \Psi_{t,k}(\tilde\mu^\star).
\]
From \ref{ass:Stab}-\eqref{it:StabCont}, the first term on the r.h.s. vanishes as $k \to +\infty$. 
From \ref{ass:Stab}-\eqref{it:StabSemCont}, the limsup of the second term is lower than $\sup_{t \in \T} \Psi_t(\tilde\mu^\star)$.
Since this last quantity is negative by construction of $\tilde\mu^\star$, this completes the proof. 
\end{proof}

\begin{proof}[Proof of Theorem \ref{thm:WeakStab}-\ref{it:StabComp}]
From \ref{ass:Stab}, Lemma \ref{lem:StabQualif} and the regularity assumptions \ref{ass:StabReg}-(\ref{it:StabRegalif},\ref{it:StabTech}), $\overline{\mu}_k$ satisfies the assumptions of Theorem~\ref{thm:abstractGibbs}, for every large enough $k$.
As a consequence, $\overline{\mu}_k$ has the Gibbs form~\eqref{eq:Gibbsmeasure}-\eqref{eq:compSlackness} for some $\overline\lambda_k \in \M_+(\T)$. 
It remains to show that $(\overline\lambda_k)_{k \geq 1}$ is weakly pre-compact.
Using the Prokhorov theorem as in the proof of Theorem \ref{thm:abstractGibbs}, it is sufficient to show that $(\overline\lambda_k ( \T ) )_{k \geq 1}$ is bounded.
Let $( \tilde{\mu}_k )_{k \geq k_0}$ be given by Lemma \ref{lem:StabQualif}.
Using Remark~\ref{rem:BoundLambda},
\begin{equation*}
    \overline\lambda_k(\T) \leq \eta_k^{-1}\left\langle \tilde\mu_k, \log\frac{\d\overline\mu_k}{\d\nu_k} + \frac{\delta \F_k}{\delta\mu}(\overline\mu_k)\right\rangle,
\end{equation*}
where $\eta_k := - \sup_{t \in \T} \Psi_{t,k} ( \overline\mu_k ) + \langle \tilde\mu_k - \overline{\mu}_k, \tfrac{\delta \Psi_{t,k}}{\delta\mu} (\overline\mu_k) \rangle$.
By convexity of $\Psi_{t,k}$ and~\eqref{eq:convex-diff}, the bound still holds with $\eta := - \sup_{k \geq k_0} \sup_{t \in \T} \Psi_{t,k} ( \tilde\mu_k )$ instead of $\eta_k$, which is positive by Lemma \ref{lem:StabQualif} and independent of $k$.
By optimality of $\overline\mu_k$, \cite[Lemma 1.4-(b)]{nutz2021introduction} gives the estimate
\[
    \left\langle \tilde\mu_k, \log\frac{\d\overline\mu_k}{\d\nu_k} \right\rangle \leq H(\tilde\mu_k|\nu_k),
\]
and $( H(\tilde\mu_k|\nu_k) )_{ k \geq k_0}$ is bounded from Lemma \ref{lem:StabQualif}.
The bound on $( H(\tilde\mu_k|\nu_k) )_{ k \geq k_0}$ together with the dual representation for entropy~\eqref{eq:DualEntrop}, the domination \ref{ass:StabReg}-\eqref{itStabDom} on $\tfrac{\delta \F_k}{\delta\mu}(\overline\mu_k)$, and the exponential integrability \ref{ass:Stab}-\eqref{it:StabIntExp} then give that $\langle \tilde\mu_k, \tfrac{\delta \F_k}{\delta\mu}(\overline\mu_k) \rangle$ is bounded uniformly in $k$.  
This shows that $(\overline\lambda_k ( \T ) )_{k \geq 1}$ is indeed bounded.
\end{proof}

\begin{proof}[Proof of Theorem \ref{thm:WeakStab}-\ref{it:StabWeakCV}]
For $k \geq 1$, we set
\[
    P_k(x) := \exp\bigg[ -\frac{\delta\F_k}{\delta\mu}(\overline\mu_k,x) - \int_\T \frac{\delta\Psi_{t,k}}{\delta\mu}(\overline\mu_k,x)\overline\lambda_k(\d t)\bigg], \qquad \overline Z_k := \int_E P_k \d\nu_k,
\]
so that $\tfrac{\d \overline{\mu}_k}{\d \nu_k} = \overline{Z}^{-1}_k P_k$.
Using the domination \ref{ass:StabReg}-\eqref{itStabDom} and the bound on $( \overline\lambda_k(\T) )_{k \geq 1}$ given by Theorem \ref{thm:WeakStab}-\ref{it:StabComp}, there exist $D^P, \beta > 0$ such that
\begin{equation} \label{eq:Pkbound}
\forall x \in E, \; \forall k \geq 1, \qquad ( D^P )^{-1} \exp [ - \beta \phi ( x) ] \leq  P_k ( x ) \leq D^P \exp [ \beta \phi ( x) ]. 
\end{equation} 
Since $( \nu_k )_{k \geq 1}$ weakly converges to $\nu$, the lower bound in \eqref{eq:Pkbound} 
shows that $( \overline{Z}^{-1}_k )_{k \geq 1}$ is bounded. 
Since $( \nu_k )_{ k \geq 1}$ converges in $\ps_\phi (E)$, it is a tight sequence. 
Using the upper bound in \eqref{eq:Pkbound} and the exponential integrability \ref{ass:Stab}-\eqref{it:StabIntExp}, we deduce that $( \overline\mu_k )_{k \geq 1}$ is also tight, and thus pre-compact in $\ps(E)$. 
The same argument yields
\[ \lim_{M \to \infty} \sup_{k \geq 1} \int_{\phi \geq M} \phi \, \d\overline\mu_k = 0, \]
which eventually gives that $( \overline\mu_k )_{k \geq 1}$ is pre-compact in $\ps_\phi (E)$.
Up to re-indexing, we can assume that $( \overline\mu_k )_{k \geq 1}$ converges towards some $\overline\mu_\infty$ in $\ps_\phi (E)$, and that $(\overline \lambda_k )_{k \geq 1}$ weakly converges towards some $\overline\lambda_\infty \in \M_+ ( \T )$.
From \ref{ass:StabReg}-\eqref{it:StabEqui}, $\overline{\mu}_\infty \in \AdF$. 

We now show that, for all $x \in E$, $P_k(x)$ converges towards
\[ P_\infty(x) := \exp\bigg[-\frac{\delta\F}{\delta\mu}(\overline\mu_\infty,x) - \int_\T \frac{\delta\Psi_t}{\delta\mu}(\overline\mu_\infty,x)\overline\lambda_\infty(\d t)\bigg]. \]
The convergence of $\frac{\delta\F_k}{\delta\mu}(\overline\mu_k,x)$ to $\frac{\delta\F}{\delta\mu}(\overline\mu_\infty,x)$ is given by~\ref{ass:StabReg}-\eqref{it:StabEqui}. Besides, for every $x \in E$,
\begin{align*}
    & \bigg\vert \int_\T \frac{\delta\Psi_{t,k}}{\delta\mu}(\overline\mu_k,x)\overline\lambda_k (\d t) - \int_\T \frac{\delta\Psi_t}{\delta\mu}(\overline\mu_\infty,x)\overline\lambda_\infty(\d t) \bigg\vert\\
    &\leq \overline{\lambda}_k ( \T ) \sup_{t \in \T} \bigg\vert \frac{\delta\Psi_{t,k}}{\delta\mu}(\overline\mu_k,x) - \frac{\delta\Psi_{t}}{\delta\mu}(\overline\mu_\infty,x) \bigg\vert + \bigg\vert \int_\T \frac{\delta\Psi_t}{\delta\mu}(\overline\mu_\infty,x) [ \overline\lambda_k(\d t) - \overline\lambda_\infty(\d t) ] \bigg\vert,
\end{align*}
and the second term goes to $0$ using the weak convergence of $( \lambda_k )_{k \geq 1}$, because $t \mapsto \tfrac{\delta\Psi_{t}}{\delta\mu} (\overline\mu_\infty,x)$ is continuous using \ref{ass:StabReg}-\eqref{it:StabEqui}.
Since $( \overline\lambda_k(\T) )_{k \geq 1}$ is bounded, the point-wise convergence to $0$ of the first term (along the considered sub-sequence) then follows from \ref{ass:StabReg}-\eqref{it:StabEqui}. 

Let us show that for any bounded continuous $\varphi : E \rightarrow \R$,
\begin{equation} \label{eq:LastEqui}
\int_E \varphi {P_k} \d\nu_k \xrightarrow[k \rightarrow +\infty]{} \int_E \varphi P_\infty \d\nu. 
\end{equation} 
To do so, we set $f^k := \varphi P_k$. 
The upper bound in \eqref{eq:Pkbound} and the equi-continuity assumption \ref{ass:StabReg}-\eqref{it:StabEqui} w.r.t. $x$ give that $( f^k )_{k \geq 1}$ is equi-continuous. 
Moreover, the upper bound in \eqref{eq:Pkbound} and the integrability assumption \ref{ass:Stab}-\eqref{it:StabIntExp} ensure that
\[ \sup_{k \geq 1} \int_E \1_{|f_k| \geq M} |f_k| \d \nu_k \xrightarrow[M \rightarrow + \infty]{} 0. \]
Lemma \ref{lem:CVUnif} (with no dependence on $t$) then provides \eqref{eq:LastEqui}.
Taking $\varphi = 1$ further shows that $\overline Z_k$ converges to $\overline Z_\infty := \int_E P_\infty \d\nu$. 

We now write
\[ \sup_{t \in \T} \Psi_t(\overline\mu_\infty) = \big[ \sup_{t \in \T} \Psi_t(\overline\mu_\infty) - \sup_{t \in \T} \Psi_{t,k}(\overline\mu_k) \big] + \sup_{t \in \T} \Psi_{t,k}(\overline\mu_k). \]
By \ref{ass:StabReg}-\eqref{it:StabEqui}, the first term on the r.h.s. vanishes as $k \rightarrow +\infty$.  
By definition of $\overline\mu_k$, the second term is non-positive. 
Therefore, $\sup_{t \in \T} \Psi_t(\overline\mu_\infty) \leq 0$.
Using \ref{ass:StabReg}-\eqref{it:StabEqui} and the bound on $( \lambda_k ( \T ) )_{k \geq 1}$, 
\[  0 = \int_\T \Psi_{t,k} ( \overline\mu_k ) \overline\lambda_k ( \d t ) \xrightarrow[k \rightarrow +\infty]{} \int_\T \Psi_t ( \overline\mu_\infty ) \overline{\lambda}_\infty ( \d t), \]
where we also used that $t \mapsto \Psi_t ( \overline\mu_\infty ) $ is continuous from \ref{ass:StabReg}-\eqref{it:StabEqui}.
This shows that $( \overline{\mu}_\infty, \overline\lambda_\infty )$ satisfies \eqref{eq:Gibbsmeasure}-\eqref{eq:compSlackness}. 
Since $\overline{\mu}_\infty \in \AdF$ from \ref{ass:StabReg}-\eqref{it:StabEqui}, Theorem~\ref{thm:convex-case} now proves that $\overline{\mu}_\infty$ is the unique minimiser $\overline\mu$ of $\mathcal{I}$ in $\AdF \cap \Adineq$.
In particular, $\overline{\mu}$ satisfies \eqref{eq:Gibbsmeasure}-\eqref{eq:compSlackness} for any limit point $\overline\lambda_\infty$ of $(\overline\lambda_k)_{k \geq 1}$.
Since $\overline{\mu}_\infty$ was any limit point of $(\overline\mu_k)_{k \geq 1}$, this finally proves the convergence of the full sequence $(\overline\mu_k)_{k \geq 1}$ towards $\overline{\mu}$.
\end{proof}

\section{Proofs of the results of Section~\ref{s:processes}} \label{sec:Difproofs}

\subsection{Proof of Theorem~\ref{thm:ConsSchröd}}\label{ss:pf-schr}

With the family of functions $(\zeta_s)_{s \in \S}$ introduced after the statement of Theorem~\ref{thm:ConsSchröd}, the minimisation problem~\eqref{eq:SchrödProb} is exactly~\eqref{eq:min-pb} in the setting of the condition~\ref{ass-b:global}. As a consequence, the first part of Theorem~\ref{thm:ConsSchröd} is straightforward.

Let us now fix a minimiser $\overline\mu_{[0,T]}$ for~\eqref{eq:SchrödProb}, at which~\ref{ass-b:regPsi-barmu} holds. Let $\overline\zeta$, $\overline\lambda$ and $\overline S$ be given by Theorem~\ref{thm:abstractGibbs}. To complete the proof of the second part of Theorem~\ref{thm:ConsSchröd}, we need to prove the following two statements.

\begin{lemma}[Completion of the proof of Theorem~\ref{thm:ConsSchröd}] $\phantom{a}$
    \begin{enumerate}[label=(\roman*),ref=\roman*]
        \item $\nu_{[0,T]}(\overline S)=1$;
        \item there exist $(\overline\zeta_0,\overline\zeta_T) \in L^1(\R^d,\d \mu^{\mathrm{ini}}) \times L^1(\R^d,\d \mu^{\mathrm{fin}})$ such that $\overline \zeta(x_{[0,T]}) = \overline\zeta_0(x_0) + \overline\zeta_T(x_T)$, for $\nu_{[0,T]}$-a.e. $x_{[0,T]}$.
    \end{enumerate}
\end{lemma}
\begin{proof}
To prove the first claim, we consider the static Schrödinger problem
\[ \inf_{\substack{\mu \in \ps_p ( \R^d \times \R^d ) \\ \mu_1 = \mu^{\mathrm{ini}}, \, \mu_2 = \mu^{\mathrm{fin}}}} H ( \mu \vert \nu_{0,T} ). \]
The above infimum is finite, because the law of $({\sf X}_0, {\sf X}_T)$ when $\mathsf{X}_{[0,T]} \sim \overline\mu_{[0,T]}$ has finite entropy w.r.t. $\nu_{0,T}$, using the additivity property of entropy \cite[Theorem D.13]{dembo2009large}.
Since $\nu_{0,T} \sim \mu^{\mathrm{ini}} \otimes \mu^{\mathrm{fin}}$, \cite[Theorem 2.1]{nutz2021introduction} provides an optimal measure $\pi^\star \sim \nu_{0,T}$ for this static problem.
From e.g. \cite[Theorem D.3]{dembo2009large}, the disintegration $\nu^{x,y}_{[0,T]}$ of $\nu_{[0,T]}$ given $({\sf X}_0, {\sf X}_T) =(x,y)$ is well-defined. 
The path-measure $\nu^{x_0,x_T}_{[0,T]} ( \d x_{[0,T]} ) \pi^\star ( \d x_0, \d x_T )$ is then an admissible measure for \eqref{eq:SchrödProb} that is equivalent to $\nu_{[0,T]}$.
From Theorem \ref{thm:AbsLDP}-\ref{it:thmAbsAbsCont}, this shows that $\overline{\mu}_{[0,T]}$ is equivalent to $\nu_{[0,T]}$, and thus that $\nu_{[0,T]}(\overline S)=1$.

We now prove the existence of $(\overline{\zeta_0},\overline{\zeta}_T)$.
Let $( \xi^k )_{k \geq 1}$ be a sequence in $\mathrm{Span} ( (\phi^k_0)_{k\geq 1} \cup (\phi^k_T)_{k\geq 1} )$ that converges in $L^1 ( C ( [0,T], \R^d ),\d \overline\mu_{[0,T]} )$ towards $\overline{\zeta}$. 
Up to extracting a sub-sequence, we can assume that this convergence holds $\overline\mu_{[0,T]}$-a.s.
Since $\overline\mu_{[0,T]} \sim \nu_{[0,T]}$, this convergence also holds $\nu_{[0,T]}$-a.s.
By definition, $\xi^k$ only depends on $({\sf X}_0, {\sf X}_T)$; hence, the convergence holds $\nu_{0,T}$-a.s.
Since $\nu_{0,T} \sim \mu^{\mathrm{ini}} \otimes \mu^{\mathrm{fin}}$, the convergence eventually holds $\mu^{\mathrm{ini}} \otimes \mu^{\mathrm{fin}}$-a.s.
As a consequence,
\cite[Corollary 2.12]{nutz2021introduction} gives that $\overline{\zeta} = \overline\zeta_0 + \overline\zeta_T$ as desired, completing the proof.
\end{proof}

\subsection{Proof of Theorem~\ref{thm:LinkPath}}\label{ss:pf-LinkPath}

The main obstacle in the proof of Theorem \ref{thm:LinkPath} being the time-regularity of the coefficients, we are going to smooth them to produce solutions $(\varphi^k)_{k \geq 1}$ of approximated equations.
For every $t\in [0,T]$, we will then show local compactness estimates for $( \nabla \varphi^k_t )_{k \geq 1}$ in $C^1(\R^d,\R^d)$.
To identify the limit as a function of $t$, the representation formula \eqref{eq:RepFK} will be needed.
We first prove a general stability result.

\begin{proposition}[Limit HJB equation] \label{pro:wpHJB}
Under \ref{ass-c:hjb}-(\ref{it:ass-c-hjb:1}),
let
$(b^k)_{k \geq 1}$, $(\sigma^k)_{k \geq 1}$, $( c^k )_{k \geq 1}$, $( \psi^k )_{k \geq 1}$ be measurable functions 
$[0,T] \times \R^d \rightarrow \R^d$, $[0,T] \times \R^d \rightarrow \R^{d \times d}$, $[0,T] \times \R^d \rightarrow \R$, $[0,T] \times \R^d \rightarrow \R$, and $( \lambda^k )_{k \geq 1}$ be a sequence in $\M_+ ( [0,T] )$, satisfying the following conditions.
\begin{enumerate}[label=(\roman*),ref=\roman*]
\item\label{it:EDStabLoc} $t \mapsto b^k_t (0)$, $t \mapsto \sigma^k_t (0)$, $t \mapsto c^k_t (0)$, $t \mapsto \psi^k_t (0 )$ are bounded uniformly in $k \geq 1$.
\item\label{it:EDStabLip} $b^k$, $\sigma^k $, $ c^k $, $ \psi^k $ satisfy~\ref{ass-c:sde}-(\ref{it:ass-c-sde:1},\ref{it:ass-c-sde:2}) and~\ref{ass-c:hjb}-(\ref{it:ass-c-hjb:2}) uniformly in $k \geq 1$.
\item\label{it:EDStabL1}
For Lebesgue-a.e. $t \in [0,T]$, for every $x \in \R^d$, $(b^k_t(x),\sigma^k_t(x),c^k_t(x))$ converges towards $(b_t(x),\sigma_t(x),c_t(x))$ as $k \rightarrow +\infty$. 
\item\label{it:EDStabPoint} $(\psi^k)_{k \geq 1}$ converges towards $\psi$ uniformly on every compact set of $[0,T] \times \R^d$.
\item\label{it:EDStabLamb} $(\lambda^k)_{k \geq 1}$ weakly converges towards $\lambda$ in $\M_+ ( [0,T] )$.
\end{enumerate}
Let $\varphi^k$ be given by
\begin{equation} \label{eq:RepFKk}
\varphi^k_t (x) := - \log \E \exp \bigg[ \int_t^T  c^k_s (Z^{t,x,k}_s) \d s + \int_{[t,T]} \psi^k_s (Z^{t,x,k}_s) \lambda^k(\d s) \bigg], 
\end{equation} 
where $( Z^{t,x,k}_s )_{t \leq s \leq T}$ is the path-wise unique solution to
\[ \d Z^{t,x,k}_s = b_s^k ( Z^{t,x,k}_s ) \d s + \sigma_s^k ( Z^{t,x,k}_s ) \d B_s, \quad t \leq s \leq T, \qquad Z^{t,x,k}_t = x.  \]
For every $k \geq 1$, we assume that $\varphi^k$ is the mild solution of 
\begin{equation} \label{eq:approxmild}
- \varphi^k_t + \int_t^T \left(b^k_s \cdot \nabla \varphi^k_s - \frac{1}{2} \vert ( \sigma^k_s )^\top \nabla \varphi^k_s \vert^2 + \frac{1}{2} \mathrm{Tr}[ \sigma^k_s (\sigma^k_s)^\top \nabla^2 \varphi^k_s] + c^k_s\right) \d s
+ \int_{[t,T]} \psi^k_s \lambda^k ( \d s) = 0,
\end{equation}
in the sense of Definition~\ref{def:MildForm}, and that $( \nabla \varphi^k )_{k \geq 1}$ is uniformly bounded.
For every $t \in [0,T]$ and every compact set $K \subset \R^d$, we assume that $( \nabla \varphi^k_t )_{k \geq 1}$ is pre-compact in $C(K,\R^d)$. 
Then, for $t = 0$ and every $t \in (0,T]$ with $\lambda ( \{ t \} ) = 0$, $(\varphi^k_t)_{k \geq 1}$ point-wise converges towards $\varphi_t$ given by \eqref{eq:RepFK}, and $\varphi$ is the solution of \eqref{eq:HJBlim} in the sense of Definition \ref{def:MildForm}.
\end{proposition}

\begin{proof}
Using (\ref{it:EDStabLoc},\ref{it:EDStabLip}), the coefficients have linear growth uniformly in $k$, and $(\sigma^k )_{k \geq 1}$ is uniformly bounded. 
For every $(t,x) \in [0,T] \times \R^d$, this implies that
\[ \forall \alpha >0, \qquad \sup_{k \geq 1} \E \exp \bigg[ \alpha \sup_{t \leq s \leq T} \vert Z^{t,x,k}_s \vert \bigg] < +\infty. \]
From (\ref{it:EDStabLip},\ref{it:EDStabL1}), a standard coupling argument (a variation of the proof of \cite[Chapter 5, Theorem 5.2]{friedman1975stochastic}) implies that
\[ \E \big[ \sup_{t \leq s \leq T} \vert Z^{t,x,k}_s - Z^{t,x}_s \vert \big] \xrightarrow[k \rightarrow +\infty]{} 0, \]
where $( Z^{t,x}_s)_{t \leq s \leq T}$ is given by \eqref{eq:Ztx}.
Let us fix $t =0$ or any $t \in (0,T]$ that is not an atom of $\lambda$. 
Since the number of atom is always countable, this includes Lebesgue a.e. $t \in [0,T]$.
For such a $t$, \eqref{it:EDStabLamb} and \cite[Theorem 2.7]{billingsley2013convergence} imply that
\begin{equation*} 
\forall \phi \in C([0,T],\R), \quad \int_{[t,T]} \phi(s) \lambda^k( \d s) \xrightarrow[k \rightarrow +\infty]{} \int_{[t,T]} \phi(s) \lambda (\d s).
\end{equation*}
As a consequence, using (\ref{it:EDStabLip},\ref{it:EDStabL1},\ref{it:EDStabPoint}) and the continuity~\ref{ass-c:hjb}-\eqref{it:ass-c-hjb:1} of $\psi$, \eqref{eq:RepFKk} shows that $( \varphi^k_t )_{k \geq 1}$ point-wise converges towards $\varphi_t$ given by \eqref{eq:RepFK}.
The pre-compactness of $( x \mapsto \nabla \varphi^k_t (x) )_{k \geq 1}$ on every compact set further shows that $\varphi_t$ is $C^1$.
We then deduce that $( \nabla \varphi^k_t )_{k \geq 1}$ converges towards $\nabla \varphi_t$, uniformly on every compact set.
The uniform bound on $\nabla \varphi_t$ is inherited from the one on $( \nabla \varphi^k_t )_{k \geq 1}$.

From Definition \ref{def:MildForm}, for Lebesgue-a.e. $t \in [0,T]$,
\begin{equation} \label{eq:TempmfHJBbarApprox}
-\varphi^k_t + \int_t^T S^k_{t,s} \big[ b^k_s \cdot \nabla \varphi^k_s - \frac{1}{2} \lvert ( \sigma^k_s )^\top \nabla \varphi^k_s \rvert^2 + c^k_s \big] \d s + \int_{[t,T]} S^k_{t,s} \big[ \psi^k_s \big] \lambda^k( \d s) = 0.  
\end{equation}
The Feynman-Kac representation formula (see e.g. \cite[Chapter 6.5]{friedman1975stochastic}) yields
\[ \forall x \in \R^d, \qquad S^k_{t,s}[\phi](x) = \E [ \phi ( \xi^{t,x,k}_s ) ], \]
for every continuous $\phi : \R^d \rightarrow \R$ with linear growth, where $\xi^{t,x,k}$ is the solution of $\d \xi_s^{t,x,k} = \sigma^k_s ( \xi_s^{t,x,k}) \d B_s$, $\xi_t^{t,x,k} = x$. As for $Z^{t,x,k}$, we get that
\[ \E \big[ \sup_{t \leq s \leq T} \vert \xi^{t,x,k}_s - \xi^{t,x}_s \vert \big] \xrightarrow[k \rightarrow +\infty]{} 0,  \]
where the definition of $\xi^{t,x}$ involves $\sigma$ instead of $\sigma^k$.
Setting $\phi^k_s := b^k_s \cdot \nabla \varphi^k_s - \tfrac{1}{2} \lvert ( \sigma^k_s )^\top \nabla \varphi^k_s \rvert^2 + c^k_s$, the pre-compactness assumption, the bound on $( \nabla \varphi^k )_{k \geq 1}$ and \eqref{it:EDStabLip} show that $(x \mapsto \phi^k_s (x) )_{k \geq 1}$ is equi-continuous.
By dominated convergence, this provides the point-wise convergence of $S^k_{t,s} [ \phi^k_s ]$ as $k \rightarrow +\infty$, and then of the middle term in \eqref{eq:TempmfHJBbarApprox}.
Similarly, we split
\begin{multline*}
\int_{[t,T]} S^k_{t,s} \big[ \psi^k_s \big] (x) \lambda^k( \d s) - \int_{[t,T]} S_{t,s} \big[ \psi_s \big] (x) \lambda ( \d s ) =\\
\int_{[t,T]} \E \big[ \psi^k_s ( \xi^{t,x,k}_s ) - \psi_s ( \xi^{t,x}_s ) \big] \lambda^k( \d s) + \int_{[t,T]} \E \big[ \psi_s ( \xi^{t,x}_s ) \big] [ \lambda^k - \lambda ] ( \d s).
\end{multline*} 
Since $(\lambda^k ([0,T]) )_{k \geq 1}$ is bounded, the first term on the r.h.s. goes to $0$ using the path-wise convergence of $\xi^{t,x,k}$ and the uniform convergence \eqref{it:EDStabPoint} of $\psi^k$ on compact sets.
From the continuity~\ref{ass-c:hjb}-\eqref{it:ass-c-hjb:1} of $\psi$ and the weak convergence of $(\lambda^k)_{k \geq 1}$, the second term goes to $0$ too as $k \rightarrow +\infty$.
Taking the limit in \eqref{eq:TempmfHJBbarApprox} now shows that $\varphi$ is the solution of \eqref{eq:HJBlim} in the sense of Definition \ref{def:MildForm}.
\end{proof}

\begin{lemma} \label{lem:thmIFCV}
Under \ref{ass-c:hjb}-(\ref{it:ass-c-hjb:1}) and (\ref{it:EDStabLoc},\ref{it:EDStabLip},\ref{it:EDStabL1},\ref{it:EDStabPoint},\ref{it:EDStabLamb}) in Proposition \ref{pro:wpHJB}, let us assume that $t \mapsto b^k_t (x)$, $\sigma^k_t (x)$, $c^k_t (x)$, $\psi^k_t (x)$ are $C^1$, and that $\lambda^k (\d t) = \lambda^k_t \d t + \overline\lambda_T \delta_T (\d t)$ for some $t \mapsto \lambda^k_t$ in $C^1([0,T],\R_+)$.
Then, the function $\varphi^k$ given by \eqref{eq:RepFKk} belongs to $C ( [0,T] \times \R^d) \cap C^{1,2} ( (0,T) \times \R^d )$ and is the solution of \eqref{eq:approxmild} in the sense of Definition \ref{def:MildForm}.
If $( \nabla \varphi^k_t )_{k \geq 1}$ is furthermore pre-compact in $C(K, \R^d)$, for every compact $K \subset \R^d$ and $t \in [0,T]$, then the conclusion of Theorem \ref{thm:LinkPath} holds.
\end{lemma}

\begin{proof}
From \cite[Theorem 2.2]{chaintron2023existence}, the Hamilton-Jacobi-Bellman (HJB) equation
\begin{equation} \label{eq:ApproxHJB}
\partial_t u^k_t + b^k_t \cdot \nabla u^k_t + \frac{1}{2} \mathrm{Tr}[ a^k_t \nabla^2 u^k_t ] - \frac{1}{2} \lvert ( \sigma^k_t )^\top \nabla u^k_t \rvert^2 = -c^k_t -\lambda^k_t \psi^k_t, \quad
u^k_T = \overline{\lambda}_T \psi^k_T,
\end{equation}
where $a^k_t := \sigma^k_t ( \sigma^k_t )^\top$,
has a unique solution $u^k$ in $C ( [0,T] \times \R^d) \cap C^{1,2} ( (0,T) \times \R^d )$, and
\[ \sup_{(t,x) \in [0,T] \times \R^d} \lvert \nabla u^k_t (x) \rvert \leq C, \]
for a constant $C$ that does not depend on $k$, because $( \lambda^k ( [0,T] ) )_{k \geq 1}$ is bounded.
Using the Cole-Hopf transform $v^k := e^{-u^k}$, $v^k$ is a $C^{1,2}$ solution of the linear parabolic equation
\[ \partial_t v^k_t + b^k_t \cdot \nabla v^k_t + \frac{1}{2} \mathrm{Tr}[ a^k_t \nabla^2 v^k_t ] = c^k_t v^k_t + \lambda^k_t \psi^k_t v^k_t. \]
From the Feynman-Kac representation formula \cite[Chapter 6.5]{friedman1975stochastic}, for $(t,x) \in [0,T] \times \R^d$,
\begin{equation} \label{eq:ApproxFK}
v^k_t(x) = \E \exp \bigg[ \overline\lambda_T \psi^k_T ( Z^{t,x,k}_T ) + \int_t^T  [ c^k_s (Z^{t,x,k}_s) + \psi^k_s (Z^{t,x,k}_s) \lambda^k_s ] \d s \bigg].
\end{equation} 
From \eqref{eq:RepFKk}, this yields $u^k = \varphi^k$.
Classically, the solution $u^k$ of \eqref{eq:ApproxHJB} is the mild solution of \eqref{eq:approxmild} in the sense of Definition \ref{def:MildForm}.
Under the pre-compactness assumption, $( \varphi^k )_{k \geq 1}$ thus satisfies the assumptions of Proposition \ref{pro:wpHJB}, yielding that $\varphi$ given by \eqref{eq:RepFK} is the solution of \eqref{eq:HJBlim} in the sense of Definition \ref{def:MildForm}.
Let us now finish the proof of Theorem \ref{thm:LinkPath} by identifying the process with law $\overline{\mu}_{[0,T]}$.

On the canonical space $\Omega = C([0,T],\R^d)$, the canonical process satisfies 
\[ \d {\sf X}_t = b_t ( {\sf X}_t ) \d t + \sigma_t ( {\sf X}_t ) \d {\sf B}_t, \quad \nu_{[0,T]}\text{-a.s.} \]
where $( {\sf B}_t )_{0 \leq t \leq T}$ is a Brownian motion under $\nu_{[0,T]}$. Let $\tilde\mu_{[0,T]}$ be the measure over $C([0,T],\R^d)$ defined by
\begin{equation} \label{eq:mfGirsanovTransform}
\frac{\d \tilde\mu_{[0,T]}}{\d \nu_{[0,T]}} ({\sf X}_{[0,T]}) := \exp \bigg[ -\int_0^T \nabla \varphi_t ({\sf X}_t) \cdot \sigma_t ({\sf X}_t) \d {\sf B}_t - \frac{1}{2} \int_0^T \lvert \sigma_t^\top({\sf X}_t) \nabla \varphi_t ({\sf X}_t) \rvert^2 \d t \bigg].
\end{equation} 
The bounds on $\sigma$ and $\nabla {\varphi}$, inherited from the one $(\nabla \varphi^k )_{k \geq 1}$, guarantee that the Novikov condition holds
\[ \E_{\nu_{[0,T]}} \, \exp \bigg[ \frac{1}{2} \int_0^T \lvert \sigma^\top_t ({\sf X}_t) \nabla \varphi_t ({\sf X}_t) \rvert^2 \d t \bigg] < +\infty, \]
so that the Girsanov transform (see e.g. \cite[Theorem 2.3]{leonard2012girsanov}) tells us that $\tilde\mu_{[0,T]}$ is a probability measure such that
\[ \d {\sf X}_t = b_t ({\sf X}_t) \d t - a_t ( {\sf X}_t ) \nabla \varphi_t ({\sf X}_t) \d t + \sigma_t ( {\sf X}_t ) \d {\sf \tilde{B}}_t, \quad \tilde\mu_{[0,T]}\text{-a.s.}, \]
the process $({\sf \tilde{B}}_t)_{0\leq t\leq T}$ being a Brownian motion under $\tilde\mu_{[0,T]}$.
Moreover, uniqueness in law holds for the above SDE.

Under $\nu_{[0,T]}$, let now $X^k$ denote the strong solution of the SDE 
\[ \d X^k_t = b^k_t ( X^k_t ) \d t + \sigma^k_t ( X^k_t ) \d {\sf B}_t, \qquad X^k_0 = {\sf X}_0. \]
Using \eqref{eq:ApproxHJB}, Ito's formula applied to $\varphi^k = u^k$ yields
\begin{multline*}
{\varphi}^k_0 (X^k_0) - \int_0^T [ {c}^k_t (X^k_t) + \psi^k_t ( X^k_t ) \lambda^k_t ] \d t - \overline{\lambda}_T \psi^k_T ( X^k_T ) = \\
-\int_0^T \nabla \varphi^k_t ( X^k_t) \cdot \sigma^k_t ( X^k_t) \d {\sf B}_t - \frac{1}{2} \int_0^T \lvert ( \sigma^k_t )^\top ( X^k_t) \nabla \varphi^k_t ( X^k_t) \rvert^2 \d t. 
\end{multline*}
As previously, we can show that $\E [ \sup_{0 \leq t \leq T} \vert X^k_t - {\sf X}_t \vert ] \to 0$. 
We then take the $k \rightarrow + \infty$ limit (in probability) in the above equality, as we did in the proof of Proposition \ref{pro:wpHJB},
to get
\begin{equation*}
\frac{\d \tilde\mu_{[0,T]}}{\d \nu_{[0,T]}} ({\sf X}_{[0,T]}) = \exp \bigg[ {\varphi}_0 ({\sf X}_0) - \int_0^T c_t ({\sf X}_t) \d t - \int_{[0,T]} \psi_t ({\sf X}_t ) \lambda ( \d t ) \bigg]. 
\end{equation*}
Noticing that $\overline\mu ( \d x_{[0,T]}) = Z^{-1} e^{-\varphi_0 (x_0)} \tilde\mu ( \d x_{[0,T]})$, this concludes.
\end{proof}

To complete the proof of Theorem \ref{thm:LinkPath}, it remains to build approximations of the coefficients that are $C^1$ in time and satisfy the assumptions of Lemma \ref{lem:thmIFCV}.

\begin{lemma}
By convolution with a non-negative approximation of unity with compact support, let $( \lambda^k )_{k \geq 1}$ be a sequence of functions in $C^1([0,T],\R_+)$ that weakly converges towards $\lambda$.   
Similarly, let $t \mapsto c_t(x)$, $t \mapsto b_t(x)$, $t \mapsto \psi_t(x)$ be regularisations by convolution of $t \mapsto c^k_t (x)$, $t \mapsto b^k_t (x)$, $t \mapsto \psi^k_t (x)$, which are now $C^1$ in time.
Then
(\ref{it:EDStabLoc},\ref{it:EDStabLip},\ref{it:EDStabL1},\ref{it:EDStabPoint},\ref{it:EDStabLamb}) in Proposition \ref{pro:wpHJB} are satisfied. 
\end{lemma}

\begin{proof}
The convolution in time does not change the $x$-dependence: $x \mapsto c^k_t (x)$, $x \mapsto b^k_t (x)$ and $x \mapsto \psi^k_t (x)$ are globally Lipschitz uniformly in $(k,t)$.
From \ref{ass-c:sde}-\ref{ass-c:hjb}, the coefficients were locally bounded; hence, for every $x \in \R^d$, for Lebesgue-a.e. $t \in [0,T]$, $(b^k_t (x),\sigma^k_t(x),c^k_t(x))$ converges to $(b_t (x),\sigma_t(x),c_t(x))$. 
Using the uniform Lipschitz regularity, this implies that for Lebesgue-a.e. $t \in [0,T]$ and every $x \in \R^d$, $(b^k_t (x),\sigma^k_t(x),c^k_t(x))$ converges to $(b_t (x),\sigma_t(x),c_t(x))$, giving \eqref{it:EDStabL1} in Proposition \ref{pro:wpHJB}.
From \ref{ass-c:hjb}-\eqref{it:ass-c-hjb:1}, $(t,x) \mapsto \psi_t (x)$ is uniformly continuous and bounded on every compact set; hence, $\psi^k$ converges towards $\psi$ uniformly on every compact set.
The uniform bounds on $\sigma^k_t$ are similarly obtained.
We thus verified items (\ref{it:EDStabLip},\ref{it:EDStabL1},\ref{it:EDStabPoint},\ref{it:EDStabLamb}) in Proposition \ref{pro:wpHJB}.
Item (\ref{it:EDStabLoc}) further holds because the coefficients are locally bounded.
\end{proof}

To conclude the proof of Theorem \ref{thm:LinkPath} using Lemma \ref{lem:thmIFCV}, it remains to show pre-compactness.

\begin{lemma}[Equi-continuity for the gradients] \label{lem:mfRegGrad}
Under \ref{ass-c:hjb}, for every $t$ in $[0,T]$ and any compact $K \subset \R^d$, $( x \mapsto \nabla \varphi^k_t (x) )_{k \geq 1}$ is pre-compact in $C( K, \R)$, where $\varphi^k$ is defined by \eqref{eq:RepFKk}.   
\end{lemma}

\begin{proof}
From the proof of Lemma \ref{lem:thmIFCV}, $(\nabla \varphi^k_t )_{k \geq 1}$ is uniformly bounded.
From the Arzelà-Ascoli theorem, pre-compactness will follow if we show that $(x \mapsto \nabla \varphi^k_t (x) )_{k \geq 1}$ is equi-continuous.

Let us fix $x \in \R^d$.
From \ref{ass-c:hjb}-(\ref{it:ass-c-hjb:3}), the coefficients $b^k$ and $\sigma^k$ have derivatives w.r.t. $x$, which are equi-continuous w.r.t. $k$.
From \cite[Chapter 5,Theorem 5.3]{friedman1975stochastic}, $x \mapsto Z^{t,x,k}_s$ is a.s. differentiable and the gradient process $\nabla Z^{t,x,k}$ is the solution of the linear SDE
\begin{equation} \label{eq:linearSDE}
\begin{cases}
\d ( \nabla Z^{t,x,k}_s ) = \nabla b^k_s ( Z^{t,x,k}_s ) \nabla Z^{t,x,k}_s \d s + \nabla \sigma^k_s ( Z^{t,x,k}_s ) \nabla Z^{t,x,k}_s \d B_s, \\
Z^{t,x,k}_t = \mathrm{Id}.
\end{cases}
\end{equation} 
We now fix $(t,x)$ in $[0,T] \times \R^d$.
From the Lipschitz assumptions, the derivatives of the coefficients are globally bounded. 
We then define 
\[ \chi^k_t(x) := \int_t^T \big[ c^k_s (Z^{t,x,k}_s) + \lambda^k_s \psi^k_s (Z^{t,x,k}_s) \big] \d s,  \]
whose gradient w.r.t. $x$ writes
\[ \nabla \chi^k_t(x) := \int_t^T \big[ \nabla Z^{t,x,k}_s \nabla c^k_s (Z^{t,x,k}_s) + \lambda^k_s \nabla Z^{t,x,k}_s \nabla \psi^k_s (Z^{t,x,k}_s) \big] \d s.  \]
Setting $v^k :=\E[e^{\chi^k_t(x)}]$, we can write that
\begin{equation} \label{eq:GradApproxFK}
\nabla \varphi^k_t (x) = - \frac{\nabla v^k_t (x)}{v^k_t (x)} = - \frac{\E \big[ e^{\chi^k_t(x)} \nabla \chi^k_t(x) \big]}{v^k_t (x)}.
\end{equation}
To prove the equi-continuity of $\nabla \varphi^k_t$ at $x$, let us consider a sequence $( x_l )_{l \geq 1}$ that converges in $\R^d$ towards $x$ as $l \rightarrow + \infty$. 
We want to show that $\nabla \varphi^k_t (x_l)$ converges towards $\nabla \varphi^k_t (x)$ uniformly in $k$.
From \eqref{eq:ApproxFK}, we obtain as previously that  $( v^k_t (0) )_{k \geq 1}$ converges towards a positive number, so that $( \varphi^k_t (0) )_{k \geq 1}$ is bounded; the uniform bound on $\nabla \varphi^k$ then gives that $\varphi^k_t$ has linear growth uniformly in $k$. 
As a consequence, $v^k_t (x_l)$ is bounded from above, and bounded from below by a positive constant, uniformly in $(k,l)$.
Similarly, $\nabla v^k_t (x_l)$ is uniformly bounded, and $v^k_t (x_l)$ converges towards $v^k_t (x)$ uniformly in $k$.
The convergence of $\nabla v^k_t ( x_l)$ remains to be studied. 
First,
\begin{align*}
\big\lvert \nabla v^k_t (x_l) - \nabla v^k_t (x) \big\lvert &= \big\vert \E \big[ ( e^{\chi^k_t(x_l)} - e^{\chi^k_t(x)} ) \nabla \chi^k_t(x) + e^{\chi^k_t(x_l)} ( \nabla \chi^k_t(x_l) - \nabla \chi^k_t(x) ) \big] \big\vert, \\
&\leq \E \big[ \max ( e^{\chi^k_t(x_l)} , e^{\chi^k_t(x)} ) \lvert \chi^k_t(x_l) - \chi^k_t(x) \vert \lvert \nabla \chi^k_t(x) \rvert \big] \\
&\phantom{abcdefghijklmnopqrstuvwx.}+ \E \big[ \lvert \nabla \chi^k_t(x_l) - \nabla \chi^k_t(x) \vert e^{\chi^k_t(x_l)} \big]. 
\end{align*} 
We now show how to control the second term on the r.h.s., the first one being similar (and even easier because it does not involve $\nabla \chi^k_t(x_l)$).
Since $( x_l )_{l \geq 1}$, $\nabla b^k$, $\nabla \sigma^k$ and $\sigma^k$ are bounded uniformly in $k$, it is standard from \eqref{eq:linearSDE} that for any $\alpha > 0$ and $q \geq 1$,
\[ \E \bigg[ \sup_{t \leq s \leq T} \big\vert \nabla Z^{t,x_l,k}_s \big\vert^q \bigg] \quad \text{and} \quad \E \exp \bigg[ \alpha \sup_{t \leq s \leq T} \big\vert Z^{t,x_l,k}_s \big\vert \bigg] \]
are bounded uniformly in $(k,l)$. 
From the linear growth of $c^k$ and $\psi^k$ (uniformly in $k$) together with the uniform bound on $\lVert \lambda^k \rVert_{L^1(0,T)}$, we deduce analogous uniform bounds on $\E [ \vert \nabla \chi^k_t (x_l) \vert^q ]$ and $\E [ e^{\alpha \vert \chi^k_t (x_l) \vert} ]$.
For $M > 0$, we then write
\begin{multline*} 
\E \big[ \lvert \nabla \chi^k_t(x_l) - \nabla \chi^k_t(x) \vert e^{\chi^k_t(x_l)} \big] = \E \big[ \lvert \nabla \chi^k_t(x_l) - \nabla \chi^k_t(x) \vert e^{\chi^k_t(x_l)} \1_{\sup_{t \leq s \leq T} \vert Z^{t,x_l,k}_s \vert + \vert Z^{t,x,k}_s \vert \leq M} \big] \\
+ \E \big[ \lvert \nabla \chi^k_t(x_l) - \nabla \chi^k_t(x) \vert e^{\chi^k_t(x_l)} \vert \1_{\sup_{t \leq s \leq T} \vert  Z^{t,x_l,k}_s \vert + \vert Z^{t,x,k}_s \vert > M} \big]. 
\end{multline*}
The moment bounds ensure that the second term can be made arbitrarily small by choosing a large enough $M$, uniformly in $(k,l)$. 
Moreover,
\begin{multline*} \lvert \nabla \chi^k_t(x_l) - \nabla \chi^k_t(x) \rvert \leq T \sup_{t \leq s \leq T} \lvert \nabla Z^{t,x_l,k}_s \nabla c^k_s (Z^{t,x_l,k}_s) - \nabla Z^{t,x,k}_s \nabla c^k_s (Z^{t,x,k}_s) \rvert \\
+ \lVert \lambda^k \rVert_{L^1(0,T)} \sup_{t \leq s \leq T} \vert \nabla Z^{t,x_l,k}_s \nabla \psi^k_s (Z^{t,x_l,k}_s) - \nabla Z^{t,x,k}_s \nabla \psi^k_s (Z^{t,x,k}_s) \vert. 
\end{multline*}
When $M$ is fixed, the conditioning ensures that $Z^{t,x_l,k}_s$ remain in a compact set $K$ that does not depend on $(k,l)$.
Since convolution in time does not affect the space regularity, the continuous functions $x \mapsto \nabla c^k_s (x)$ and $x \mapsto \nabla \psi^k_s (x)$ are uniformly continuous on $K$, uniformly in $(s,k,l)$.
From \cite[Chapter 5,Lemma 3.3]{friedman1975stochastic},
\begin{equation} \label{eq:ApproxControl}
\E \bigg[ \sup_{t \leq s \leq T} \big\lvert Z^{t,x_l,k}_s - Z^{t,x,k}_s \big\rvert^2 \bigg] \xrightarrow[l \rightarrow + \infty]{} 0, 
\end{equation}
and this holds uniformly in $k$, because the Lipschitz constant of $b^k$ and $\sigma^k$ are independent of $k$.
Using the Cauchy-Schwarz inequality, we can now conclude if we prove that
\[ \E \bigg[ \sup_{t \leq s \leq T} \big\lvert \nabla Z^{t,x_l,k}_s - \nabla Z^{t,x,k}_s \big\rvert^2 \bigg] \xrightarrow[l \rightarrow + \infty]{} 0, \] uniformly in $k$.
Following \cite[Chapter 5,Theorem 5.2]{friedman1975stochastic} for this proof, we use \eqref{eq:linearSDE} to write that for every $t \leq s \leq T$,
\begin{multline} \label{eq:GradCV} 
\nabla Z^{t,x_l,k}_s - \nabla Z^{t,x,k}_s = \gamma^{k,x_l}_s + \delta^{k,x_l}_s + \int_t^s \nabla b^k_r ( Z^{t,x_l,k}_r ) [ \nabla Z^{t,x_l,k}_r - \nabla Z^{t,x,k}_r ] \d r \\
+ \int_t^s \nabla \sigma^k_r ( Z^{t,x_l,k}_r ) [ \nabla Z^{t,x_l,k}_r - \nabla Z^{t,x,k}_r ] \d B_r,
\end{multline}
where
\[ \gamma^{k,x_l}_s := \int_t^s [ \nabla b^k_r ( Z^{t,x_l,k}_r ) - \nabla b^k_r ( Z^{t,x,k}_r ) ] \nabla Z^{t,x,k}_r \d r, \] \[ \delta^{k,x_l}_s := \int_t^s [ \nabla \sigma^k_r ( Z^{t,x_l,k}_r ) - \nabla \sigma^k_r ( Z^{t,x,k}_r ) ] \nabla Z^{t,x,k}_r \d B_r. \]
We recall that $\nabla b^k$ and $\nabla \sigma^k$ are bounded uniformly in $k$.
We now take the square of \eqref{eq:GradCV}, and we use the Jensen and the Burkholder-Davis-Gundy (BDG) inequalities to get that for every $t \leq s \leq T$,
\begin{multline} \label{eq:GradGronwall}
\E \bigg[ \sup_{t \leq r \leq s} \big\lvert \nabla Z^{t,x_l,k}_r - \nabla Z^{t,x,k}_r \big\rvert^2 \bigg] \leq 4 \E \bigg[ \sup_{t \leq r \leq T} \vert \gamma^{k,x_l}_r \vert^2 + \vert \delta^{k,x_l}_r \vert^2 \bigg] \\ + C \int_t^s \E \bigg[ \sup_{t \leq r \leq u} \big\lvert \nabla Z^{t,x_l,k}_r - \nabla Z^{t,x,k}_r \big\rvert^2 \bigg] \d u, 
\end{multline} 
for a constant $C > 0$ that does not depend on $(k,l)$. 
Using Gronwall's lemma, we can conclude if we show that the first term on the r.h.s. of \eqref{eq:GradGronwall} goes to $0$ as $l \rightarrow +\infty$, uniformly in $k$.
As previously, we write that
\begin{multline*} 
\E \bigg[ \sup_{t \leq r \leq T} \vert \gamma^{k,x_l}_r \vert^2 \bigg] = \E \bigg[ \sup_{t \leq r \leq T} \vert \gamma^{k,x_l}_r \vert^2 \1_{\sup_{t \leq s \leq T} \vert  Z^{t,x_l,k}_s \vert + \vert Z^{t,x,k}_s \vert + \vert \nabla Z^{t,x,k}_r \vert \leq M} \bigg] \\
+ \E \bigg[ \sup_{t \leq r \leq T} \vert \gamma^{k,x_l}_r \vert^2 \1_{\sup_{t \leq s \leq T} \vert  Z^{t,x_l,k}_s \vert + \vert Z^{t,x,k}_s \vert + \vert \nabla Z^{t,x,k}_r \vert > M} \bigg], 
\end{multline*}
and the second term can be made arbitrarily small by choosing a large enough $M$, uniformly in $(k,l)$. 
For the first term, when $M$ is fixed, $Z^{t,x_l,k}_s$, $Z^{t,x,k}_s$ and $\nabla Z^{t,x_l,k}_s$ remain in a compact set $K$ that does not depend on $(k,l)$, and $x \mapsto \nabla b^k_s (x)$ is uniformly continuous on $K$, uniformly in $(s,k,l)$.
The uniform vanishing of $\E [ \sup_{t \leq r \leq T} \vert \gamma^{k,x_l}_r \vert^2 ]$ then results from \eqref{eq:ApproxControl}. 
Using the BDG inequality, the same reasoning gives the uniform vanishing of $\E [ \sup_{t \leq r \leq T} \vert \delta^{k,x_l}_r \vert^2 ]$, completing the proof.
\end{proof}

Iterating the method of the above proof for differentiating $x \mapsto \nabla Z^{t,x,k}$ within \eqref{eq:ApproxFK}, we get the following corollary.

\begin{corollary}
In the setting of Lemma \ref{lem:thmIFCV},
if the coefficients $b^k$, $\sigma^k$, $c^k$ and $\psi^k$ are $C^\infty$ in $x$ with bounded derivatives, then $\varphi^k$ is $C^\infty$ in $x$.
\end{corollary}

\appendix

\section{Linear functional derivative}\label{app:Diff}

This appendix gathers some useful tools for differentiating functions on measures.
Let $E$ be a complete metric space, endowed with its Borel $\sigma$-algebra.

\begin{definition}[Linear functional derivative] \label{def:diff}
Let $\mathcal{C}$ be a convex subset of $\ps(E)$ and let $\mu \in \mathcal{C}$.
A map $F : \mathcal{C} \rightarrow \R$ is differentiable at $\mu$ w.r.t. to the set of directions $\mathcal{C}$ if there exists a measurable map 
\[
\frac{\delta F}{\delta\mu}(\mu):
\begin{cases}
E \rightarrow \R, \\
x \mapsto \frac{\delta F}{\delta\mu}(\mu,x),
\end{cases}
\]
such that for every $\mu'$ in $\mathcal{C}$, $\tfrac{\delta F}{\delta\mu}(\mu)$ is $\mu'$-integrable and satisfies
\begin{equation*} 
\varepsilon^{-1} \bigl[ F( (1-\varepsilon) \mu + \varepsilon \mu' ) - F ( \mu ) \bigr] \xrightarrow[\varepsilon \rightarrow 0^+]{} \biggl\langle \mu' - \mu, \frac{\delta F}{\delta \mu}(\mu) \biggr\rangle. 
\end{equation*}

To alleviate notations, the dependence on $\mathcal{C}$ is not emphasised in the notation $\tfrac{\delta F}{\delta \mu}(\mu)$. 
The map $\tfrac{\delta F}{\delta \mu}(\mu)$ being defined up to an additive constant, we adopt the usual convention that $\langle \mu, \tfrac{\delta F}{\delta \mu}(\mu) \rangle = 0$.
This map is called the \emph{linear functional derivative of $F$ at $\mu$} (w.r.t. the set of directions $\mathcal{C}$).
We notice that this definition does not depend on the behaviour of $F$ outside of an arbitrary small neighbourhood of $\mu$.
\end{definition}

\begin{example}[Linear case]
In the particular case $F(\mu) = \langle \mu, f \rangle$ for some measurable $f: E \rightarrow \R$ that is $\mu$-integrable for every $\mu$ in $\mathcal{C}$, we have $\frac{\delta F}{\delta \mu}(\mu,x) = f(x) - \langle \mu, f \rangle$.
\end{example}

\begin{rem}[Convex case]
If $F$ is convex on $\mathcal{C}$ and differentiable at $\mu \in \mathcal{C}$ w.r.t. $\mathcal{C}$, Definition \ref{def:diff} implies that
\begin{equation}\label{eq:convex-diff}
\forall \mu' \in \mathcal{C}, \qquad    F(\mu') \geq F(\mu) + \left\langle \mu'-\mu, \frac{\delta F}{\delta \mu}(\mu)\right\rangle.
\end{equation}
\end{rem}

\begin{lemma}[Integration]\label{lem:integ-diff}
    Let $\mu, \mu' \in \mathcal{C}$. Assume that for any $r \in [0,1]$, $F$ is differentiable at $(1-r)\mu' + r\mu$ w.r.t. the set of directions $\mathcal{C}$, and that
    \begin{equation*}
        \int_0^1 \left|\bigg\langle \mu-\mu', \frac{\delta F}{\delta \mu}((1-r)\mu' + r \mu) \bigg\rangle\right| \d r < \infty.
    \end{equation*}
    Then
    \begin{equation*}
        F(\mu)-F(\mu') = \int_0^1 \bigg\langle \mu-\mu', \frac{\delta F}{\delta \mu}((1-r)\mu' + r \mu) \bigg\rangle \d r.
    \end{equation*}
\end{lemma}

Notably, the relative entropy does not fit in the setting of Definition \ref{def:diff}, because the directional derivative may be $-\infty$.
However, the following classical result still holds, whose proof traces back to \cite[Lemma 2.1]{csiszar1975divergence}. 

\begin{lemma}[Derivative of $H$] \label{lem:HDiff}
Let $\mu$ and $\mu'$ be measures in $\ps(E)$ with $H(\mu \vert \nu)$, $H(\mu' \vert \nu) < +\infty$. 
For $\varepsilon \in [0,1]$, we define $\mu_\varepsilon := (1-\varepsilon)\mu + \varepsilon\mu'$. Then,
\[ \frac{\d}{\d \varepsilon}\bigg\vert_{\varepsilon=0} H(\mu_\varepsilon \vert\nu) = \int_E \log \frac{\d \mu}{\d \nu}\d \mu' - H(\mu\vert\nu), \]
the integral on the r.h.s. being in $\R \cup \{-\infty\}$.
\end{lemma}

\section{A variant of Theorem~\ref{thm:abstractGibbs} for linear constraints} \label{sec:appLin}

In this section we only work with linear inequality constraints as in Remarks~\ref{rem:lin-ineq-constr}-\ref{rem:LinIneqCons}. We present a set of relaxed assumptions on the family $(\psi_\t)_{\t \in \T}$ which ensures that the conclusion of Theorem~\ref{thm:abstractGibbs} remains in force, in the absence of equality constraints ($\S = \emptyset)$.

\begin{proposition}[Lagrange multiplier for linear constraints]\label{pro:LinAbs}
Let us consider $\nu$ in $\ps(E)$,  a continuous $\phi : E \to \R_+$, and a measurable $\mathcal{F} : \ps_\phi (E) \to \R$.
Let $(\psi_\t)_{\t \in \T}$ be a family of measurable functions $E \to \R$. 
We assume that there exists $C_\psi \geq 0$ such that, 
\begin{equation}\label{eq:lb-dom-psi}
\forall x \in E, \qquad \inf_{\t \in \T} \psi_t(x) \geq -C_\psi [ 1+\phi(x)].
\end{equation}
Then the $\Psi_\t : \ps_\phi(E) \to (-\infty,+\infty], \mu \mapsto \Psi_t ( \mu )$ are measurable functions. 
Let the set $\Adineq$ be defined accordingly. 
We further assume that for $\nu$-a.e. $x$ in $E$, the map $\t \in \T \mapsto \psi_t(x)$ is continuous, and that there exist $\tilde{\mu} \in \ps_\phi(E)$ and $\eta > 0$ such that
\begin{equation}\label{eq:EquivLinQualif}
\A \t \in \T, \quad \langle \tilde\mu, \psi_t \rangle \leq - \eta.
\end{equation}
In this setting, let $\overline{\mu} \in \Adineq$ be a minimiser for~\eqref{eq:min-pb} (with $\S=\emptyset$). 
We assume that $\F$ is differentiable at $\overline{\mu}$ w.r.t. the set of directions $\ps_\phi(E)$ in the sense of Definition \ref{def:diff}, and that $C_\F \in \R$ exists such that for $\nu$-a.e. $x$ in $E$,
\begin{equation*} 
\frac{\delta\F}{\delta\mu}(\overline\mu,x) \geq -C_\F \, [ 1+\phi(x) ].
\end{equation*} 
We eventually assume that 
\begin{equation} \label{eq:assIntImpr}
\langle \nu, e^{\overline{\alpha} \phi} \rangle < +\infty,
\end{equation}
for some $\overline{\alpha} \in \R$ with $C_\F + C_\psi C(\nu,\F,\eta,\tilde\mu,\overline{\mu}) < \overline{\alpha}$, where 
\[ C(\nu,\F,\eta,\tilde\mu,\overline{\mu}) := \eta^{-1} \bigg[ \log \bigg\langle \nu, \exp \bigg[ - \frac{\delta \F}{\delta\mu}(\overline{\mu}) \bigg] \bigg\rangle + H ( \tilde\mu \vert \nu ) + \biggl\langle \tilde\mu, \frac{\delta\F}{\delta\mu}(\overline{\mu}) \biggr\rangle \bigg]. \]
Then, there exists a positive Radon measure $\overline{\lambda}$ on $\T$ such that
\begin{equation} \label{eq:LinGibbsMeasure}
\frac{\d \overline\mu}{\d \nu} \bigl( x \bigr) = \overline{Z}^{-1} \exp \biggl[ - \frac{\delta \F}{\delta \mu}(\overline\mu,x) - \int_\T \psi_{\t}(x) \overline{\lambda}( \d \t) \biggr],
\end{equation} 
where $\overline{Z} \in (0,+\infty)$ is a normalising constant.
Moreover, $\overline\lambda(\T) \leq C(\nu,\F,\eta,\tilde\mu,\overline{\mu})$ and the complementary slackness condition is satisfied:
\begin{equation} \label{eq:LinCompSlackness}
\langle \overline{\mu}, \psi_{\t} \rangle = 0 \quad \text{for } \overline{\lambda}\text{-a.e. } \t \in\T.
\end{equation} 
\end{proposition}

Notice that, contrarily to the remainder of this work, here the functions $\Psi_t$ are not required to be l.s.c. on $\ps_\phi(E)$, but under~\eqref{eq:lb-dom-psi} they are so if the $\psi_t$ are l.s.c. on $E$; 
furthermore, they are allowed to take the value $+\infty$ on $\ps_\phi(E)$. 
The main strength of Proposition \ref{pro:LinAbs} is to remove the upper bound condition in Remark~\ref{rem:LinIneqCons}.
Only lower bounds are required, as in \cite[Assumption (A1)]{liu2020large}.
Another advantage of Proposition \ref{pro:LinAbs} is that its proof is elementary, in the sense that it does not rely on the use of the Hahn-Banach theorem.

\begin{proof}[Proof of Proposition \ref{pro:LinAbs}] 
Since $\F$ is not assumed to be convex, $\overline{\mu}$ may not be unique.

\medskip

\noindent\emph{\textbf{Step 1.} Linearisation.} 
Let us consider $\mu$ in $\Adineq$. 
Then $\F ( \mu)$ is finite by assumption. 
We first assume that $H(\mu \vert\nu) < +\infty$.
For $\varepsilon \in (0,1]$, we set $\overline{\mu}_\varepsilon := (1-\varepsilon)\overline{\mu} + \varepsilon \mu$. 
Since $\ps_{\phi}(E)$ is convex, we write for any $\varepsilon \in (0,1]$,
\[ H(\overline{\mu}|\nu) + \F(\overline{\mu}) \leq H(\overline{\mu}_\varepsilon|\nu) + \F(\overline{\mu}_\varepsilon) \leq (1-\varepsilon) H(\overline{\mu}|\nu) + \varepsilon H(\mu|\nu) + \F(\overline{\mu}_\varepsilon), \]
where we used the convexity of $H$ in the second inequality. This rewrites
\begin{equation*}
    H(\overline{\mu}|\nu) - H(\mu|\nu) \leq \varepsilon^{-1} \big[ \F(\overline{\mu}_\varepsilon)-\F(\overline{\mu}) \big] ,
\end{equation*}
which yields, thanks to Definition~\ref{def:diff}, 
\begin{equation*}
    H(\overline{\mu}|\nu) + \left\langle \overline{\mu}, \frac{\delta\F}{\delta\mu}(\overline{\mu})\right\rangle \leq H(\mu|\nu) + \left\langle \mu, \frac{\delta\F}{\delta\mu}(\overline{\mu})\right\rangle.
\end{equation*}
This inequality still holds if $H({\mu}|\nu) = +\infty$. 
We thus proved that $\overline{\mu}$ is a minimiser for
\begin{equation} \label{eq:AuxProb} \inf_{\mu \in \Adineq} H(\mu \vert\nu) + \biggl\langle \mu, \frac{\delta\F}{\delta\mu}(\overline{\mu}) \biggr\rangle. 
\end{equation}
By strict convexity of $H$, $\overline{\mu}$ is the unique minimiser for \eqref{eq:AuxProb}.

\medskip

\noindent\emph{\textbf{Step 2.} Dual problem.} 
For any $\lambda$ in $\M_+(\T)$ with $C_\F + C_\psi \lambda (\T) < \overline\alpha$, we define the Gibbs free energy 
\[ G( \lambda ) := -\log \bigg\langle \nu, \exp \biggl[ -\frac{\delta\F}{\delta\mu}(\overline{\mu})-\int_\I \psi_\t \lambda(\d \t) \biggr] \bigg\rangle, \]
together with the Gibbs measure $\mu_\lambda$ such that
\begin{equation*} 
\frac{\d \mu_\lambda}{\d \nu} \bigl( x \bigr) = e^{G(\lambda)} \exp \biggl[ - \frac{\delta \F}{\delta \mu}(\overline\mu,x) - \int_\I \psi_\t(x) \lambda( \d \t) \biggr].
\end{equation*}
From \eqref{eq:assIntImpr}, $\mu_\lambda$ is well-defined and satisfies $H( \mu_\lambda \vert \nu) < +\infty$; thus, $\mu_\lambda$ belongs to $\ps_\phi(E)$ thanks to Lemma~\ref{lem:integ}. 
For every $\mu$ in $\ps_\phi(E)$, we get as in the proof of Theorem~\ref{thm:convex-case} that
\[ H(\mu|\nu) + \biggl\langle \mu, \frac{\delta\F}{\delta\mu}(\overline{\mu}) \biggr\rangle + \int_\I \langle \mu , \psi_\t \rangle \lambda(\d \t) = H(\mu|\mu_\lambda) + G(\lambda) \]
this equality being in $\R \cup \{+\infty \}$.
Since $H(\mu|\mu_\lambda) \geq 0$, this proves that
\begin{equation} \label{eq:Gvar}
G( \lambda ) = \inf_{\mu \in \ps_\phi(E)} H( \mu \vert \nu ) + \biggl\langle \mu, \frac{\delta\F}{\delta\mu}(\overline{\mu}) \biggr\rangle + \int_\I \langle \mu , \psi_\t \rangle \lambda(\d \t),
\end{equation} 
the unique minimiser being $\mu_\lambda$. 
We now study the Lagrange dual problem
\begin{equation} \label{eq:DualLin}
\sup_{\substack{\lambda \in \M_+(\I) \\ C_\F + C_\psi \lambda (\T) < \overline\alpha}} G(\lambda).
\end{equation} 
The following weak duality relation is verified
\begin{equation} \label{eq:WeakDual}
\sup_{\substack{\lambda \in \M_+(\I) \\ C_\F + C_\psi \lambda (\T) < \overline\alpha}} G(\lambda) \leq \inf_{\mu \in \Adineq} H(\mu \vert\nu) + \biggl\langle \mu, \frac{\delta\F}{\delta\mu}(\overline{\mu}) \biggr\rangle, 
\end{equation}
guarantying the finiteness of the supremum \eqref{eq:DualLin}.

\medskip

\noindent\emph{\textbf{Step 3.} Existence for $\overline{\lambda}$.} 
Let $( \lambda_k )_{k\in \N}$ be a maximising sequence for \eqref{eq:DualLin}.
Noticing that
\[ \sup_{\substack{\lambda \in \M_+(\I) \\ C_\F + C_\psi \lambda (\T) < \overline\alpha}} G(\lambda) \geq G(0) = -\log \bigg\langle \nu, \exp \bigg[ - \frac{\delta \F}{\delta\mu}(\overline{\mu}) \bigg] \bigg\rangle,\]
we get that for every $\varepsilon > 0$, $G(\lambda_k) \geq G(0) - \eta \varepsilon$ for $k$ large enough.
Using \eqref{eq:Gvar}, 
\[ G(\lambda_k) \leq 
H ( \tilde\mu \vert \nu ) + \biggl\langle \tilde\mu, \frac{\delta\F}{\delta\mu}(\overline{\mu}) \biggr\rangle + \int_\I \langle \tilde\mu , \psi_\t \rangle \lambda_k (\d \t), \]
with $\langle \tilde\mu , \psi_\t \rangle \leq -\eta$, hence
\begin{equation} \label{eq:MassB}
\lambda_k(\I) \leq \varepsilon + C(\nu,\F,\eta,\tilde\mu,\overline{\mu}). 
\end{equation} 
Using the Prokhorov theorem as in the proof of Theorem \ref{thm:abstractGibbs}, the bound on $( \lambda_k ( \T ) )_{k \geq 1}$ implies that $( \lambda_k )_{k \geq 1}$ is relatively compact for the weak convergence of measures.
Up to re-indexing, we can thus assume that $( \lambda_k )_{k \in \N}$ weakly converges towards some $\overline\lambda \in \M_+(\T)$.
By assumption, $\t \mapsto \psi_\t(x)$ is $\nu$-a.e. continuous, and $\nu$-a.e. bounded because $\T$ is compact. 
As a consequence,
\[ \int_\I \psi_\t(x) \lambda_k(\d \t) \xrightarrow[k \rightarrow + \infty]{} \int_\I \psi_\t(x) \overline\lambda( \d \t) \, \text{ for $\nu$-a.e. } x \in E.\]
Fatou's lemma then yields
\[ \bigg\langle \nu , \exp \biggl[ - \frac{\delta\F}{\delta\mu}(\overline{\mu}) -\int_\I \psi_\t \overline\lambda(\d \t) \biggr] \bigg\rangle \leq \liminf_{k\rightarrow+\infty} \bigg\langle \nu, \exp \biggl[ - \frac{\delta\F}{\delta\mu}(\overline{\mu}) -\int_\I \psi_\t \lambda_k(\d \t) \biggr] \bigg\rangle, \]
proving that $\overline{\lambda}$ realises the supremum \eqref{eq:DualLin}. 
Since \eqref{eq:MassB} was true for every $\varepsilon > 0$ provided that $k$ was large enough, we get that $\overline\lambda(\T) \leq C(\nu,\F,\eta,\tilde\mu,\overline{\mu})$.
In particular, $C_\F + C_\psi \overline\lambda (\T) < \overline\alpha$ from our assumption on $\overline\alpha$.

\medskip

\noindent\emph{\textbf{Step 4.} Admissibility for the Gibbs measure $\mu_{\overline\lambda}$.} 
Let $\tilde\varepsilon > 0$ be such that $C_\F + C_\psi [ \overline\lambda (\T) + \tilde\varepsilon ] < \overline\alpha$.
Given any $(\t_0,\varepsilon)$ in $\T \times (0,\tilde\varepsilon]$, the perturbation $\overline{\lambda} + \varepsilon \delta_{\t_0}$ is admissible for \eqref{eq:DualLin}, $\delta_{\t_0}$ being the Dirac mass at $\t_0$. The optimality of $\overline{\lambda}$ yields
\[ \log \bigg\langle \nu , \exp \biggl[ - \frac{\delta\F}{\delta\mu}(\overline{\mu}) -\int_\I \psi_\t \overline{\lambda}(\d \t) \biggr] \bigg\rangle \leq \log \bigg\langle \nu , \exp \biggl[ - \frac{\delta\F}{\delta\mu}(\overline{\mu}) -\int_\I \psi_\t \big[ \overline{\lambda}(\d \t) + \varepsilon \delta_{\t_0}(\d \t) \big] \biggr] \bigg\rangle, \]
so that subtracting the r.h.s., and dividing by $\varepsilon$,
\[ \varepsilon^{-1} \big\langle \mu_{\overline\lambda} , 1 - e^{-\varepsilon \psi_{\t_0}} \big\rangle \leq 0, \]
that we rewrite as
\begin{equation} \label{eq:PreFatou}
\varepsilon^{-1} \big\langle \mu_{\overline\lambda}, 1 - e^{-\varepsilon \lvert \psi_{\t_0}\rvert_+} \big\rangle \leq \varepsilon^{-1} \big\langle \mu_{\overline\lambda},  e^{\varepsilon \lvert \psi_{\t_0} \rvert_-} -1 \big\rangle. 
\end{equation} 
By the mean value theorem,
\[ \A \varepsilon \in \big( 0,\tfrac{\tilde\varepsilon }{2} \big], \quad \varepsilon^{-1} \lvert e^{\varepsilon \lvert \psi_{\t_0} \rvert_-} -1 \rvert \leq \lvert \psi_{\t_0} \rvert_- e^{\varepsilon \lvert \psi_{\t_0} \rvert_-} \leq C e^{\tilde\varepsilon \lvert \psi_{\t_0} \rvert_-}, \]
for some constant $C >0$ that does not depend on $\varepsilon$.
The integrability condition \eqref{eq:assIntImpr} and our choice of $\tilde\varepsilon$ guarantee that $e^{\tilde\varepsilon \lvert \psi_{\t_0} \rvert_-}$ is $\mu_{\overline\lambda}$-integrable, so that
\[ \varepsilon^{-1} \big\langle \mu_{\overline\lambda}, e^{\varepsilon \lvert \psi_{\t_0} \rvert_-} -1 \big\rangle \xrightarrow[\varepsilon \rightarrow 0]{} \langle \mu_{\overline\lambda}, \lvert \psi_{\t_0} \rvert_- \rangle. \]
To send $\varepsilon$ to $0$ in the l.h.s. of \eqref{eq:PreFatou} which is non-negative, we apply Fatou's lemma.
At the end of the day, gathering everything, 
\[ \langle \mu_{\overline\lambda}, \psi_{\t_0} \rangle \leq 0. \]
Since the above inequality holds for every $\t_0$ in $\T$, $\mu_{\overline{\lambda}}$ belongs to $\Adineq$. 

\medskip

\noindent\emph{\textbf{Step 5.} Complementary slackness condition.} 
For any $\varepsilon$ in $(0,1)$, the perturbation $\overline{\lambda} - \varepsilon \overline{\lambda}$ is admissible for \eqref{eq:DualLin}.
As above, the optimality of $\overline{\lambda}$ implies
\[ \biggl\langle \mu_{\overline{\lambda}} , \exp \biggl[ \varepsilon \int_\T \psi_\t \overline{\lambda}(\d \t) \biggr]-1 \bigg\rangle \geq 0. \]
Dividing by $\varepsilon$ and sending $\varepsilon$ to $0$, the same splitting and domination arguments as above provide that
\[ \int_\T \langle \mu_{\overline{\lambda}}, \psi_\t \rangle \overline{\lambda}(\d \t) \geq 0. \]
From the previous step, $\langle \mu_{\overline{\lambda}}, \psi_\t \rangle$ is always non-positive. This proves that $\langle \mu_{\overline{\lambda}}, \psi_\t \rangle = 0$ for $\overline\lambda$-a.e. $\t$ in $\I$.

\medskip

\noindent\emph{\textbf{Step 6.} Conclusion.} 
Since $\mu_{\overline{\lambda}}$ is the unique minimiser for \eqref{eq:Gvar}, \eqref{eq:WeakDual} and the previous step show that $\mu_{\overline{\lambda}}$ is a minimiser for \eqref{eq:AuxProb}. 
From the uniqueness proved in \emph{\textbf{Step 1.}}, we eventually get that $\mu_{\overline{\lambda}} = \overline{\mu}$.
\end{proof}

The same method could be adapted when the $\Psi_\t$ are non-linear convex functions.

\section*{Acknowledgements}
\addcontentsline{toc}{section}{Acknowledgement}

The work of G.C. is partially supported by the Agence Nationale de la Recherche through the grants ANR-20-CE40-0014 (SPOT) and ANR-23-CE40-0003 (Conviviality). 
The work of J.R. is partially supported by the Agence Nationale de la Recherche through the grants ANR-19-CE40-0010-01 (QuAMProcs) and ANR-23-CE40-0003 (Conviviality).

\printbibliography
\addcontentsline{toc}{section}{References}

\end{document}